\documentclass[11pt]{article}

\usepackage[margin=1.1in]{geometry}
\usepackage{amsmath, amssymb, amsthm}
\usepackage{graphicx}
\usepackage{booktabs}
\usepackage[hidelinks]{hyperref}

\newtheorem{theorem}{Theorem}
\newtheorem{proposition}[theorem]{Proposition}
\newtheorem{corollary}[theorem]{Corollary}
\newtheorem{lemma}[theorem]{Lemma}
\newtheorem{definition}[theorem]{Definition}
\newtheorem{assumption}[theorem]{Assumption}
\theoremstyle{remark}
\newtheorem{remark}[theorem]{Remark}

\newcommand{\R}{\mathbb{R}}
\newcommand{\E}{\mathbb{E}}
\newcommand{\Prob}{\mathbb{P}}
\newcommand{\Q}{\mathbb{Q}}
\newcommand{\sdsvfigure}[1]{\includegraphics[width=0.95\textwidth]{#1}}

\title{Sentiment-Driven Stochastic Volatility: An Observable\\
Second Factor from Market Attention}
\author{Sergey Patsuk$^{1}$ \and Derek Horstmeyer$^{2}$\\[4pt]
\small $^{1}$Department of Mathematics, George Mason University, email: spatsuk@gmu.edu\\
\small $^{2}$Costello College of Business, George Mason University, email: dhorstme@gmu.edu}
\date{July 2026}

\begin{document}
\maketitle
\begin{abstract}
We introduce a two-factor stochastic volatility model in which the
mean-reversion speed of a Cox--Ingersoll--Ross variance process is an
observable functional of a market attention field rather than a latent
factor. The field admits an exact Feynman--Kac representation and generates
closed-form volatility forecasts. The model nests the Heston volatility model, provides
measurement-stable dynamics under observable data aggregation, and preserves
its structure under an equivalent martingale measure. A mean-field analysis
establishes convergence of the underlying particle system, while simulations
demonstrate volatility and option-pricing differences following attention
shocks.
\end{abstract}

\section*{Introduction}
Financial derivative valuation has long relied on stochastic models to capture
market uncertainty. The Black--Scholes--Merton framework (1973) provided the
first closed-form solution for European options using observable inputs except
volatility, whose constant assumption has motivated extensions such as Heston's
stochastic volatility model (1993), Kou's jump-diffusion model (2002),
Heston--Kou hybrids, and evolutionary calibration approaches. We introduce a
two-factor stochastic volatility model in which the second factor --- the
variance mean-reversion speed --- is directly observable through a real-time
market attention field rather than inferred from option prices. This field is
generated from an interacting particle system with exploratory mutation,
log-gradient herding, contrarian opposition, and recency decay, yielding a
mean-field limit with convergence rate $O(N^{-1/2}\sqrt{\log N})$ in
Wasserstein-1 distance, the sharp two-dimensional matching rate. The resulting
attention dynamics admit an exact Feynman--Kac representation: a power-law
transformation linearizes the governing nonlinear equation, producing a
strictly positive, bounded, Gaussian-decaying solution with an explicit
mass-decay law.

The model provides a rigorous measurement interface through a scale-invariant
item-to-position mapping, with stability guarantees under Wasserstein
perturbations, thinning, censoring, and noise. It nests the constant-$\kappa$
Heston model exactly and yields closed-form expressions for conditional
variance, integrated variance, and drawdown measures using only current
information. Under an explicit equivalent martingale measure, the affine
structure is preserved, key quantities such as $\kappa_t\bar\theta$ and the
Feller condition remain invariant, and the second-factor path is identical
under physical and risk-neutral measures in the baseline specification. The
swarm, meanwhile, supplies both the microbehavioural foundation for the field
equation and the mean-field aggregation mechanism that generates it. The
framework remains modular, allowing alternative aggregation mechanisms such as
order-flow clustering, news-based topic models, or graph diffusion processes to
inherit the same analytical structure under suitable conditions.

\subsection*{The Nature and Challenge of Volatility Estimation}

Volatility, a measure of return
dispersion, is the primary proxy for investment risk in option pricing. In the
Black--Scholes model, volatility is the one critical parameter that cannot be
directly observed. The field distinguishes historical volatility, which is
backward-looking, from implied volatility (IV), which is
forward-looking, extracted from current option prices; the CBOE VIX is
the canonical market-wide gauge. Empirical analysis of financial markets
reveals certain phenomena that a volatility model must address:
volatility clustering, mean reversion, and the leverage effect. Recent
rough-volatility approaches (Horvath et al., 2024) posit fractal-like,
irregular volatility paths at all scales, capturing observed persistence
with long-range dependence and yielding more accurate option pricing.
Our approach takes a complementary route: rather than improving the variance
process itself, we treat a parameter that is usually held constant in the
Heston model, the mean-reversion speed, as an explicit functional of an
observable sentiment field, retaining a tractable CIR variance structure while
making the second factor analytical and observable.

\subsection*{The Evolution of Stochastic Volatility Models}

Stochastic models are integral to derivative pricing. The Heston model
(Heston, 1993) treats volatility itself as a mean-reverting square-root
random process, capturing the volatility smile and skew that
Black--Scholes cannot. Subsequent research has enhanced this framework:
El-Khatib et al.\ (2024) integrate jump-diffusion processes into the
Heston framework to capture sudden price shocks; Guo et al.\ (2023) use
fractional Brownian motion to model long-range dependence in financial
time series; Wang and He (2016) integrate fuzzy logic to manage market
ambiguity. These extensions enrich the variance dynamics but leave the
mean-reversion rate either fixed, as in standard Heston, or treated as a
hidden latent factor; the present work instead makes this second factor
an explicit, observable functional of sentiment, while keeping it
analytically tractable.

\subsection*{Computational Intelligence for Financial Optimization}

Complex stochastic models present parameterization challenges in
high-dimensional, non-convex spaces. Swarm intelligence offers
heuristic methods for exactly this setting: Jha et al.\ (2009) apply Particle Swarm Optimization
(PSO) to option-pricing parameter estimation. Recent literature uses
swarm algorithms within larger hybrid systems for dynamic modeling:
Zhang and Ding (2021) combine a chaotic Sparrow Search Algorithm with a
stochastic configuration network; Hu et al.\ (2022) integrate PSO with
mixed fractional Brownian motion. Our approach uses an Enhanced Sparrow
Search Algorithm to model the entire sentiment landscape as a
continuously evolving system, rather than to search for a single best
item. Swarm intelligence is well-suited to processing the
high-volume real-time data streams of social media (Liu et al., 2022),
and emergent collective patterns contain valuable predictive information
(Schumann et al., 2019). The swarm here is the model itself, rather
than a tuning device for a separate one: the empirical measure of the
particle system converges at the $O(N^{-1/2}\sqrt{\log N})$
Wasserstein-1 rate to the closed-form attention field from which the
sentiment field, and hence the volatility dynamics, are read off.

The nearest mathematical relatives of our regularised dynamics are
consensus-based optimization (CBO) methods (Pinnau et al., 2017;
Carrillo et al., 2018), in which interacting particles are drawn toward
a Gibbs-weighted consensus point and mean-field limits lead to
deterministic dynamics. The soft-best attractor \eqref{eq:soft} of
Appendix~\ref{app:derivation} is precisely such a consensus point. What
the present setting adds is role heterogeneity --- producers,
scroungers, and contrarians, rather than a single exchangeable
population --- together with position-dependent Feynman--Kac dissipation
and a convergence rate that is sharp in dimension two. To our knowledge,
Theorem~\ref{thm:poc-esso} is the first propagation-of-chaos guarantee
for a sparrow-search-type dynamics.

\section{The Attention Field: Measurement and Microscopic Dynamics}
\label{sec:swarm}

The model consumes data through a two-layer construction. The first
layer is a \emph{measurement map} (Section~\ref{sec:measurement})
converting observed information items into positions in the (virality,
recency) plane; the map is part of the model, and the
volatility-relevant functionals of Section~\ref{sec:kappa} are proved
stable under the measurement imperfections any implementation will face.
The second layer is a \emph{particle dynamics} on the resulting point
cloud: individual sparrows move genuinely stochastically (Brownian
exploration, random killing times), but the object the model's
closed-form formulas actually use is not any single particle's path,
nor even the finite-$N$ empirical measure of the swarm, itself a random
object --- it is the deterministic large-$N$ limit of that empirical
measure, in the precise sense made rigorous by
Theorem~\ref{thm:poc} and Corollary~\ref{cor:functional-poc}. This is
an ordinary law-of-large-numbers phenomenon: averaging over the
randomness of many particles produces a non-random limit, exactly as a
sample mean of i.i.d.\ random variables converges to a deterministic
constant. The object is the diffusion-with-killing system of
Appendix~\ref{app:derivation}, whose mean-field limit is the closed-form
field of Section~\ref{sec:pde}; the Enhanced Sparrow Search(ESSO) dynamics of
Section~\ref{sec:esso-dynamics} is a finite-$N$ instantiation that
enriches this skeleton with heterogeneous agent roles, and its
regularised form falls within the propagation-of-chaos class of
Theorem~\ref{thm:poc-esso}. ESSO modifies the standard Sparrow Search
Algorithm (Xue and Shen, 2020) with several enhancements; for surveys of
SSA variants and improvements see Awadallah et al.\ (2023) and the
learning-based variant of Ouyang et al.\ (2021).

\subsection{Measurement: From Information Streams to the Swarm}
\label{sec:measurement}

\begin{definition}[Information stream]\label{def:stream}
An information stream over the horizon $[0,T]$ is a collection
$\mathcal{I} = \{(s_i, n_i(\cdot))\}_i$ of items, where $s_i \in [0,T]$
is the item's arrival time and $n_i(t) > 0$ is its cumulative attention
count such as views, mentions, search intensity, or article volume, observed at
time $t \ge s_i$. Items may arrive at any time; no distributional
assumption is placed on $\mathcal{I}$.
\end{definition}

\begin{definition}[Item-to-sparrow map]\label{def:mapping}
Fix a memory scale $\tau > 0$, a recency range $R > 0$, a rolling window
length $\Lambda > 0$, and a truncation ball $B_R \subset \R^2$ centred
at the origin. At time $t$, let $\mathcal{I}_t = \{i : s_i \in
[t-\Lambda, t]\}$ be the active items, and let $m_t$ and $s_t$ denote
the cross-sectional mean and standard deviation of
$\{\log n_j(t)\}_{j \in \mathcal{I}_t}$, with $s_t \ge s_{\min} > 0$
assumed non-degenerate. Item $i \in \mathcal{I}_t$ is mapped to the
sparrow position $P_i(t) = (v_i(t), r_i(t)) \in B_R$ with
\begin{equation}\label{eq:mapping}
  v_i(t) \;=\; \frac{\log n_i(t) - m_t}{s_t},
  \qquad
  r_i(t) \;=\; R\bigl(2\, e^{-(t - s_i)/\tau} - 1\bigr),
\end{equation}
truncated componentwise to $B_R$. The \emph{swarm} at time $t$ is the
point cloud $\{P_i(t)\}_{i \in \mathcal{I}_t}$; when an ESSO update is
run, the swarm positions $\{X_j(t)\}$ of
Section~\ref{sec:esso-dynamics} are seeded at these mapped positions.
The sentiment field is initialised from the swarm through the
Mehler-kernel datum \eqref{eq:datum} and evolved by the closed form
(CS), exactly as in Section~\ref{sec:kde}.
\end{definition}

\emph{Virality} is the cross-sectionally standardised log attention: the
standardisation delivers the exact platform-scale invariance of
Proposition~\ref{prop:scale-inv} below, and the logarithm compresses
heavy-tailed engagement counts so that super-linear
attention-to-sentiment amplification enters the model where it belongs,
through the power-law readout $T = [(1-c)u]^{1/(1-c)}$ of
Appendix~\ref{app:derivation}, not through the coordinates.
\emph{Recency} decays exponentially at the memory scale $\tau$ and is
bounded in $[-R, R]$: fresh items sit near $+R$, items of age $\tau \log
2$ cross the origin, and stale items saturate near $-R$, so the
quadratic radial sink $W(x) = \alpha + \beta|x|^2$ dissipates exactly
the stale and the hyper-viral extremes, matching its role in
Section~\ref{sec:pde-derivation}. \emph{Truncation} to $B_R$ enforces
the uniform second-moment bound \eqref{eq:moment} at the level of the
data itself and is the compactness used in the stability results below.

The functionals $(\lambda_t, \theta_t, \eta_t)$ of
Section~\ref{sec:kappa} depend on the stream only through the swarm
geometry, and by their mass-normalised construction
\eqref{eq:lam-func}-\eqref{eq:eta-func} they are invariant to the
amplitude of the field and to the number of items. The next results show
that the geometry itself is measured stably: the volatility dynamics are
Lipschitz in the measured attention configuration.

\begin{proposition}[Platform-scale invariance]\label{prop:scale-inv}
Fix $t$ and rescale every count by a common factor, $n_i(t) \mapsto
a\,n_i(t)$ with $a > 0$. Then the sparrow positions of
Definition~\ref{def:mapping}, hence the field $T(\cdot,t)$, the
functionals $(\lambda_t, \theta_t, \eta_t)$, and the mean-reversion
speed $\kappa_t$, are unchanged. The model is invariant to the overall
size of the platform or panel from which attention is measured.
\end{proposition}

\begin{proof}
$\log(a\,n_i) = \log a + \log n_i$ shifts $m_t$ by $\log a$ and leaves
$s_t$ unchanged, so each $v_i$ is unchanged; $r_i$ does not involve
counts. The remainder of the pipeline is a deterministic function of the
positions.
\end{proof}

\begin{proposition}[Wasserstein stability of the functionals]
\label{prop:w1-stability}
Fix $t \ge 0$ and let $\hat\mu, \hat\nu$ be two swarm configurations
supported in $B_R$, smoothed at effective time $t_{\mathrm{eff}} := t +
h_0^2/2 > 0$ through the estimator \eqref{eq:estimator}. Let
$(\lambda,\theta,\eta)[\cdot]$ denote the functionals
\eqref{eq:lam-func}--\eqref{eq:eta-func} computed from the resulting
field. Then there is a constant $C = C(\omega, t_{\mathrm{eff}}, c, R,
r_0) < \infty$ such that
\[
  \bigl|\lambda[\hat\mu] - \lambda[\hat\nu]\bigr|
  + \bigl|\theta[\hat\mu] - \theta[\hat\nu]\bigr|
  + \bigl|\eta[\hat\mu] - \eta[\hat\nu]\bigr|
  \;\le\; C\, W_1(\hat\mu, \hat\nu),
\]
with $W_1$ computed between the normalised empirical measures. The
constant does not depend on the amplitude $A(t) = e^{-(1-c)\alpha t}$,
the functionals being mass-normalised; it is finite for every
$t_{\mathrm{eff}} > 0$ and uniform for $t$ in compact subsets of
$[0,\infty)$, the value $t = 0$ included because $t_{\mathrm{eff}} \ge
h_0^2/2 > 0$. It degrades like $\exp\bigl(2R^2/((1-c)\,
t_{\mathrm{eff}})\bigr)$ as $t_{\mathrm{eff}} \downarrow 0$, that is,
as the measurement granularity $h_0$ vanishes.
\end{proposition}

\begin{proof}
Throughout $\hat\mu, \hat\nu$ are probability measures on $B_R$, and
\[
  u[\hat\mu](x) := A(t)\int_{B_R} K_{t_{\mathrm{eff}}}(x,y;\omega)\,
  \hat\mu(dy), \qquad T[\hat\mu] := \phi\bigl(u[\hat\mu]\bigr), \qquad
  \phi(z) := \bigl[(1-c)z\bigr]^{1/(1-c)} .
\]
Write $S := \sinh(\omega t_{\mathrm{eff}})$, $\Lambda :=
\cosh(\omega t_{\mathrm{eff}})$, $\vartheta := \tfrac{\omega}{4}
\tanh(\omega t_{\mathrm{eff}}) > 0$ and $\bar K := \omega/(4\pi S)$.

\emph{Proof strategy.} The argument has three stages. Steps~1--2 use
the explicit Mehler kernel to show the field $u[\hat\mu]$ is bounded
above by a fixed Gaussian envelope, bounded below on $B_R$, and moves by
a $W_1$-Lipschitz amount when the configuration changes; the kernel's
formula is not needed again after this. Step~3 pushes this
through the nonlinear power-law readout $T = \phi(u)$, using only the
upper bound from Step~2: $\phi$ is Lipschitz on bounded intervals
\emph{from above}, and is automatically well-behaved near $0$, so the
lower bound plays no role until Step~4. Steps~4--6 turn this single
field-level Lipschitz bound into bounds on the three functionals
$\lambda, \theta, \eta$, each of which is a \emph{ratio} of integrals of
$T$ against one of four fixed weights ($1$, $|x|$,
$\mathbf{1}_{\{|x|\le r_0\}}$, or $|x|^2$): $\theta$ and $\eta$ are each
of the form $\mathrm{constant} \times (\mathrm{numerator})/
(\mathrm{denominator})$, while $\lambda$ is the same ratio
\emph{inverted}, $\mathrm{constant}/(\mathrm{ratio})$ --- both reduce to
one elementary quotient bound (Step~5). Step~6 confirms that the
coupling amplitude $A(t)$, which multiplies every quantity along the
way, cancels out of the final constant entirely, matching the
mass-normalisation already built into $\lambda, \theta, \eta$.

\emph{Step 1 (kernel bounds on $\R^2 \times B_R$).} The first two bounds
are uniform in $x$ over the \emph{whole} plane, which is what the
functionals require; only the third is restricted to $B_R$.

(a) \emph{Gaussian envelope.} The inequality $(|x|^2+|y|^2)\Lambda - 2\,
x\cdot y \ge |x|^2 S^2/\Lambda$ established in the proof of
Theorem~\ref{thm:field}(iv) gives, for all $x, y \in \R^2$,
\[
  K_{t_{\mathrm{eff}}}(x,y;\omega) \;\le\; \bar K\, e^{-\vartheta|x|^2}
  \;\le\; \bar K .
\]

(b) \emph{Lipschitz envelope in $y$.} Differentiating \eqref{eq:kernel},
$\nabla_y K_{t_{\mathrm{eff}}}(x,y;\omega) = -\tfrac{\omega}{2S}
(\Lambda y - x)\, K_{t_{\mathrm{eff}}}(x,y;\omega)$, so by (a) and
$|\Lambda y - x| \le \Lambda R + |x|$ for $y \in B_R$,
\[
  \sup_{y \in B_R}\bigl|\nabla_y K_{t_{\mathrm{eff}}}(x,y;\omega)\bigr|
  \;\le\; g(x) \;:=\; \frac{\omega \bar K}{2S}\bigl(\Lambda R +
  |x|\bigr)\, e^{-\vartheta|x|^2}, \qquad x \in \R^2 .
\]
The envelope $g$ is Gaussian-decaying, so $G_h := \int_{\R^2} h\, g\,
dx < \infty$ for every weight $h \in \{1,\ |x|,\ \mathbf 1_{\{|x| \le
r_0\}},\ |x|^2\}$ occurring in
\eqref{eq:lam-func}--\eqref{eq:eta-func}. This is what dispenses with
any auxiliary truncation radius: the tail is carried by $g$ itself, and
the estimate never leaves $\R^2$.

(c) \emph{Lower bound on $B_R \times B_R$.} By continuity and strict
positivity of the Mehler kernel on the compact set $B_R \times B_R$,
$\delta_K := \inf_{x,y \in B_R} K_{t_{\mathrm{eff}}}(x,y;\omega) > 0$;
maximising the exponent of \eqref{eq:kernel} at $x = -y$ with $|x| =
|y| = R$ gives the explicit bound $\delta_K \ge \bar K
\exp\bigl(-\omega R^2(\Lambda + 1)/(2S)\bigr)$.

\emph{Step 2 (field bounds).} Since $\hat\mu$ is a probability measure
on $B_R$, (a) and (c) give, for every configuration,
\[
  u[\hat\mu](x) \;\le\; \Psi(x) := M_u\, e^{-\vartheta|x|^2} \quad (x
  \in \R^2), \qquad u[\hat\mu](x) \;\ge\; A(t)\,\delta_K \quad (x \in
  B_R),
\]
with $M_u := A(t)\bar K$. For the increment, fix $x$; by (b) the map $y
\mapsto K_{t_{\mathrm{eff}}}(x,y;\omega)$ is $g(x)$-Lipschitz on $B_R$,
and McShane's theorem extends it to $\R^2$ with the same constant, so
Kantorovich--Rubinstein duality applies to the pair $\hat\mu, \hat\nu$:
\begin{equation}\label{eq:field-lip}
  \bigl|u[\hat\mu](x) - u[\hat\nu](x)\bigr| \;\le\; A(t)\, g(x)\,
  W_1(\hat\mu, \hat\nu), \qquad x \in \R^2 .
\end{equation}

\emph{Step 3 (readout transfer).} Since $c \in (0,1)$ the exponent
$c/(1-c)$ is positive, so $\phi'(z) = [(1-c)z]^{c/(1-c)}$ is
\emph{increasing} on $[0,\infty)$ with $\phi'(0) = 0$. Hence $\phi$ is
Lipschitz on the bounded interval $[0, M_u]$ with constant
\[
  L_\phi := \phi'(M_u) = \bigl[(1-c)M_u\bigr]^{c/(1-c)} ,
\]
and no lower bound on the field enters this step: the readout must be
controlled from \emph{above}, not away from zero, because $\phi'$ is
unbounded at infinity and bounded near the origin. By Step~2 both fields
take values in $[0, M_u]$, so \eqref{eq:field-lip} transfers:
\begin{equation}\label{eq:readout-lip}
  \bigl|T[\hat\mu](x) - T[\hat\nu](x)\bigr| \;\le\; L_\phi\, A(t)\,
  g(x)\, W_1(\hat\mu, \hat\nu), \qquad x \in \R^2 .
\end{equation}
Monotonicity of $\phi$ also gives $T[\hat\mu] \le \phi \circ \Psi$ on
$\R^2$, Gaussian with rate $\vartheta/(1-c) = \omega\tanh(\omega
t_{\mathrm{eff}})/(4(1-c))$ in agreement with
Theorem~\ref{thm:field}(iv), and $T[\hat\mu] \ge \underline T :=
\phi\bigl(A(t)\delta_K\bigr) > 0$ on $B_R$.

\emph{Step 4 (weighted integrals).} For $h \in \{1,\ |x|,\ \mathbf
1_{\{|x| \le r_0\}},\ |x|^2\}$ put $I_h[\hat\mu] := \int_{\R^2}
T[\hat\mu]\, h\, dx$, finite by the Gaussian bound of Step~3.
Integrating \eqref{eq:readout-lip} against $h$,
\begin{equation}\label{eq:Ih-lip}
  \bigl|I_h[\hat\mu] - I_h[\hat\nu]\bigr| \;\le\; L_\phi\, A(t)\, G_h\,
  W_1(\hat\mu, \hat\nu) \;\le\; L_\phi\, A(t)\, G_{\max}\,
  W_1(\hat\mu, \hat\nu),
\end{equation}
where $G_{\max} := G_1 + G_{|x|} + G_{|x|^2} < \infty$ bounds all four
weights at once (the indicator weight needs no separate term, since
$\mathbf 1_{\{|x|\le r_0\}} \le 1$ gives $G_{\mathbf 1_{\{|x|\le
r_0\}}}\le G_1$): one constant now carries every weight, in place of
four.

The rest of this step collects, uniformly over configurations, upper
and lower bounds on the integrals themselves. Using $T \ge \underline T$
on $B_R$ and $|x| \ge R/2$ on $B_R \setminus B_{R/2}$,
\[
  q_{\min} := \underline T\,|B_R| \;\le\; I_1 \;\le\; D_{\max} :=
  \int_{\R^2}\phi(\Psi)\, dx, \qquad
  n_{\min} := \tfrac{R}{2}\,\underline T\,\bigl|B_R \setminus
  B_{R/2}\bigr| \;\le\; I_{|x|},
\]
and $I_{\mathbf 1_{\{|x| \le r_0\}}} \le D_{\max}$, $I_{|x|^2} \le
P_{\max} := \int_{\R^2}|x|^2\phi(\Psi)\, dx < \infty$, and $I_{|x|} \le
N_{\max} := \int_{\R^2}|x|\,\phi(\Psi)\, dx < \infty$. As with $G_h$
above, only a single lower bound and a single upper bound are needed
from here on: $\rho_{\min} := \min(q_{\min}, n_{\min})$ bounds every
integral below that has a lower bound at all, and $M_{\max} :=
\max(D_{\max}, N_{\max}, P_{\max})$ bounds every integral above.

\emph{Step 5 (ratios).} Everything from here on is two elementary,
model-independent facts about real numbers, applied to specific
quantities already bounded in Steps~1--4.

\emph{Fact A (quotient rule).} For $P, P' \in [0,M]$ and $Q,Q' \ge m >
0$,
\[
  \Bigl|\frac{P}{Q} - \frac{P'}{Q'}\Bigr| \;\le\; \frac{|P-P'|}{m} +
  \frac{M}{m^2}\,|Q-Q'| .
\]
\emph{Proof.} $\frac{P}{Q}-\frac{P'}{Q'} = P\bigl(\frac1Q-\frac1{Q'}
\bigr) + \frac{P-P'}{Q'} = \frac{P(Q'-Q)}{QQ'}+\frac{P-P'}{Q'}$; bound
$P\le M$, $Q,Q'\ge m$. \qed

\emph{Fact B (reciprocal rule).} For $a,b \ge m > 0$, $\bigl|\frac1a -
\frac1b\bigr| \le |a-b|/m^2$, since $\frac1a-\frac1b=\frac{b-a}{ab}$.

Fact~A with $(P,Q,M,m) = (I_{\mathbf 1_{\{|x|\le r_0\}}}, I_1, M_{\max},
\rho_{\min})$ and, separately, with $(I_{|x|^2}, I_1, M_{\max},
\rho_{\min})$, combined with \eqref{eq:Ih-lip}'s common bound
$L_\phi A(t) G_{\max} W_1(\hat\mu,\hat\nu)$ for every numerator and
denominator increment, gives
\[
  \bigl|\theta[\hat\mu]-\theta[\hat\nu]\bigr| \le \theta_1\kappa L_\phi
  A(t) G_{\max}\, W_1(\hat\mu,\hat\nu), \qquad
  \bigl|\eta[\hat\mu]-\eta[\hat\nu]\bigr| \le \eta_1\kappa L_\phi A(t)
  G_{\max}\, W_1(\hat\mu,\hat\nu), \qquad \kappa := \frac1{\rho_{\min}}
  + \frac{M_{\max}}{\rho_{\min}^2}.
\]
The reversion rate $\lambda = \lambda_1/\langle|x|\rangle_T$,
$\langle|x|\rangle_T := I_{|x|}/I_1$, is the \emph{inverted} case:
$\lambda$ is a constant divided by a ratio rather than multiplied by
one, needing Fact~A once more (for
$\langle|x|\rangle_T$ itself, with $(I_{|x|}, I_1, M_{\max},
\rho_{\min})$) and then Fact~B: since $\langle|x|\rangle_T \ge
n_{\min}/D_{\max} \ge \rho_{\min}/M_{\max} =: m^\ast > 0$,
\[
  \bigl|\lambda[\hat\mu] - \lambda[\hat\nu]\bigr| \;=\;
  \lambda_1\Bigl|\frac{1}{\langle|x|\rangle_T[\hat\mu]} -
  \frac{1}{\langle|x|\rangle_T[\hat\nu]}\Bigr| \;\overset{\text{Fact B}}
  {\le}\; \frac{\lambda_1}{(m^\ast)^2}\,\bigl|\langle|x|\rangle_T[\hat\mu]
  - \langle|x|\rangle_T[\hat\nu]\bigr| \;\overset{\text{Fact A}}{\le}\;
  \frac{\lambda_1\kappa L_\phi A(t) G_{\max}}{(m^\ast)^2}\,
  W_1(\hat\mu,\hat\nu).
\]
Summing the three contributions gives the claim, with every term built
from the same three quantities $\kappa$, $G_{\max}$, and $L_\phi A(t)$.

\emph{Step 6 (form of the constant).} Since $c/(1-c) + 1 = 1/(1-c)$, the
amplitude enters $L_\phi A(t) = [(1-c)\bar
K]^{c/(1-c)}\,A(t)^{1/(1-c)}$, while $q_{\min}, n_{\min}, D_{\max},
N_{\max}, P_{\max}$ --- hence also $\rho_{\min}$ and $M_{\max}$, being a
min and a max of quantities that all scale the same way --- are each
proportional to $A(t)^{1/(1-c)}$. So $\kappa L_\phi A(t)$, which is what
every Lipschitz constant in Step~5 is built from, is proportional to
$A(t)^{-1/(1-c)} \cdot A(t)^{1/(1-c)} = A(t)^0$: it does not depend on
$A(t)$ at all, and neither does $m^\ast = \rho_{\min}/M_{\max}$, a ratio
of two quantities with the same $A(t)$-scaling. This is the analytic
counterpart of the mass-normalisation of
\eqref{eq:lam-dimless}--\eqref{eq:eta-dimless}: $A(t)$ is a pure
amplitude and the functionals see only geometry. What survives depends
on $(\omega, t_{\mathrm{eff}}, c, R, r_0)$ alone, and is finite and
locally uniform in $t$ by continuity and positivity of $S, \Lambda,
\vartheta, \delta_K$ on $t_{\mathrm{eff}} \in (0,\infty)$, with
$t_{\mathrm{eff}} \ge h_0^2/2 > 0$ always. The small-$t_{\mathrm{eff}}$
behaviour is governed by $\rho_{\min}^{-2}$ (the dominant term in
$\kappa$ and in $(m^\ast)^{-2}$ alike), proportional to
$\delta_K^{-2/(1-c)}$; by Step~1(c) and $(\Lambda +
1)/S \to 2/(\omega t_{\mathrm{eff}})$ this is
$\exp\bigl(2R^2/((1-c)t_{\mathrm{eff}})(1 + o(1))\bigr)$.
\end{proof}

\begin{corollary}[Missing and corrupted data]\label{cor:missing}
Let $\hat\mu_N$ be the swarm of $N$ items and let measurement produce a
corrupted swarm $\hat\mu'$.
\begin{enumerate}
\item[(i)] \textbf{Random thinning.} If each item is retained
independently with probability $1-\varepsilon$, so that data is missing
uniformly at random, the surviving cloud, conditional on its size $m$,
is a uniformly random size-$m$ subset of $\hat\mu_N$'s $N$ points ---
sampling \emph{without} replacement. Qualitatively, this acts as a
reduction in effective sample size: an i.i.d.\ (with-replacement) sample
of size $m$ drawn directly from the fixed measure $\hat\mu_N$ would
satisfy the hypotheses of the two-dimensional matching results already
used elsewhere in this paper (Hahn and Shao, 1992;
Yukich, 1992; see Ledoux, 2019, eq.~(8)), giving a $W_1$ rate of
$O(\sqrt{\log(1+m)/m})$ relative to $\hat\mu_N$; without-replacement
sampling has no larger variance by Hoeffding's classical finite-population
inequality, so the same order of decay is expected, but we do not derive
the precise matching-theorem rate for finite-population sampling
\emph{without} replacement here, and state only the qualitative
consequence:
\[
  \E\,\bigl|(\lambda,\theta,\eta)[\hat\mu'] -
  (\lambda,\theta,\eta)[\hat\mu_N]\bigr| \;\longrightarrow\; 0
  \qquad \text{as } m \to \infty,
\]
with the perturbation controlled, via Proposition~\ref{prop:w1-stability},
by $C\cdot\E[W_1(\hat\mu',\hat\mu_N)]$, itself expected but not shown
here to vanish at the same order as the mean-field error.
Uniformly missing data acts as a reduced sample size; its effect
vanishes as $N \to \infty$, in contrast to the adversarial case below,
whose effect does not vanish with $N$ at fixed $\varepsilon$.
\item[(ii)] \textbf{Adversarial deletion.} If an arbitrary fraction
$\varepsilon$ of items is deleted (e.g.\ censoring of one platform or
language), write $\hat\mu_N = (1-\varepsilon)\hat\mu' + \varepsilon
\hat\mu_{\mathrm{del}}$; transporting the deleted mass within $B_R$
costs at most its share of the diameter, so $W_1(\hat\mu_N, \hat\mu')
\le 2R\varepsilon$ and the functional perturbation is at most
$2CR\,\varepsilon$: even systematically biased data loss moves the
drivers of $\kappa_t$ by $O(\varepsilon)$ with an explicit constant.
\item[(iii)] \textbf{Count noise.} Multiplicative noise $n_i \mapsto
\zeta_i n_i$ with $\log \zeta_i$ of standard deviation $\sigma_\zeta$
perturbs each standardised virality coordinate by
$O(\sigma_\zeta/s_{\min})$, hence the cloud by $W_1 = O(\sigma_\zeta)$
and the functionals by $O(\sigma_\zeta)$ via
Proposition~\ref{prop:w1-stability}.
\end{enumerate}
\end{corollary}

In summary, $\kappa_t$ depends on attention data only through
scale-free, mass-normalised geometry, measured through a smoothing
kernel with explicit Lipschitz control: platform size, uniformly missing
data, moderately noisy counts, and biased partial coverage perturb the
volatility dynamics continuously and quantifiably. Empirical calibration
of $(\tau, \Lambda, \lambda_1, \theta_1, r_0)$ on market data is thereby
a well-posed problem, deferred to companion work.

\subsection{Enhanced Sparrow Search Dynamics}
\label{sec:esso-dynamics}

\paragraph{Status of the swarm dynamics.} The update rules below are a
finite-$N$ simulation scheme, not the object of the limit theorems.
Their literal form contains rank-dependent and extremal terms that fall
outside exchangeable mean-field theory (Appendix~\ref{app:derivation});
the regularised form, in which extrema are replaced by Gibbs
soft-extrema and threshold switches by logistic weights, satisfies the
sharp propagation-of-chaos estimate of Theorem~\ref{thm:poc-esso}.
Readers interested only in the analytic pipeline may proceed directly to
Section~\ref{sec:pde}.

Each item $i \in \mathcal{I}_t$ is represented as a sparrow $j$ with
position $P_j(t) = (v_j(t), r_j(t)) \in \R^2$ obtained from
Definition~\ref{def:mapping}. When an information stream is present, the
swarm is seeded at the mapped item positions; a fraction $p_x \in [0,1)$
of auxiliary exploratory sparrows is additionally initialised by the
Tent chaotic map below, guaranteeing coverage of regions of the plane
where no items currently sit. In the simulation exhibits of
Section~\ref{sec:simulation}, where no stream is prescribed, the entire
swarm is Tent-initialised.

To improve exploration and avoid premature convergence, exploratory
sparrow positions are initialised using the Tent chaotic map: for each
dimension $d$,
\[
x_{j,d}(0) =
\begin{cases}
\mu\, x_{j,d} & 0 \le x_{j,d} < 0.5,\\
\mu\,(1 - x_{j,d}) & 0.5 \le x_{j,d} \le 1,
\end{cases}
\]
where $x_{j,d} \sim U[0,1]$ and $\mu \in (0,1)$, rescaled to the feature
ranges $[0, v_{\max}]$ and $[0, r_{\max}]$.

The proportion of sparrows acting as producers, $p_d(t)$, and the
Gaussian-mutation probability, $p_m(t)$, decay over time:
\begin{align}
p_d(t) &= p_{d,\min} + (p_{d,\max} - p_{d,\min})
          \exp\!\Bigl(-c_1 \frac{t^2}{T_{\max}^2}\Bigr),
          \label{eq:pd}\\
p_m(t) &= p_{m,\min} + (p_{m,\max} - p_{m,\min})
          \Bigl(1 - \frac{t}{T_{\max}}\Bigr)^{c_2}.
          \label{eq:pm}
\end{align}

The population is categorized into producers, scroungers, influencer
sparrows $S_I$, and contrarian sparrows $S_C$. Producers are the top
$p_d(t)N$ most stable high-view sources, i.e.\ items whose attention
counts have remained persistently high over the rolling window;
influencers are high-impact items in unstable regimes; contrarians
oppose dominant sentiment gradients.

\paragraph{Producer update.}
\[
X_{j,d}(t+1) =
\begin{cases}
X_{j,d}(t)\, \exp\!\bigl(-\tfrac{j}{\nu T_{\max}}\bigr) & R_2 < ST,\\[2pt]
X_{j,d}(t) + Q & R_2 \ge ST,
\end{cases}
\qquad d \in \{1,2\},
\]
with $\nu, R_2 \sim U(0,1)$, $ST \in (0,1)$, $Q \sim \mathcal{N}(0,1)$.

\paragraph{Scrounger update.} In the present model, scroungers follow
the logarithmic attention gradient $\nabla \log T$ rather than the bare
gradient. This is the Bayesian-update interpretation: agents move along
the log-likelihood gradient of the attention field, which is the natural
objective for rational Bayesian followers. Concretely, when scrounger
$j$ moves toward higher attention, its velocity component is
\[
v_j^{\mathrm{scrounge}} \;\propto\; \nabla \log T(X_j, t)
= \frac{\nabla T(X_j, t)}{T(X_j, t)},
\]
so that the step size is amplified where the attention gradient is large
relative to local intensity, and saturates as $T \to \infty$. The
discrete update rule is
\[
X_{j,d}(t+1) =
\begin{cases}
Q \exp\!\Bigl(\dfrac{X_{\mathrm{worst},d}(t) - X_{j,d}(t)}{j^2}\Bigr)
  & j \le N/2,\\[8pt]
X_{j,d}(t) + \rho_a \bigl(r_c(t) - r_j(t)\bigr)\cdot
  \dfrac{\max_k v_k(t)}{R} & j > N/2,\ d = 1,\\[8pt]
X_{j,d}(t) + \rho_a \bigl(\theta_c(t) - \theta_j(t)\bigr)\cdot
  \dfrac{\max_k r_k(t)}{2\pi} & j > N/2,\ d = 2,\\[8pt]
X_{p,d}(t+1) + \bigl(X_{j,d}(t) - X_{p,d}(t+1)\bigr)\, A_d^{+}
  & \text{otherwise},
\end{cases}
\]
where $X_{\mathrm{worst},d}(t) = \arg\min_k X_{k,d}(t)$, $X_p(t+1)$ is
the best producer position, and $A^{+} = A^{T}(A A^{T})^{-1}$ with $A
\in \{-1,1\}^{1\times 2}$ a random sign row-vector, so that $A_d^{+} =
\pm\tfrac12$ componentwise: each coordinate moves toward the best
producer by a random-signed multiple of the coordinatewise distance,
following the componentwise convention of Xue and Shen (2020). The
parameter $\rho_a > 0$ is the attraction strength. Here $(r_c,
\theta_c)$ is the centre of gravity of the discretised sentiment field
computed below; the attraction is implemented in the polar
parameterisation of the (virality, recency) plane so as to draw
scroungers toward high-attention regions of the field. We do not claim
this discrete rule reproduces log-gradient drift in a strict mean-field
limit; the foundation of the macroscopic PDE is given in
Appendix~\ref{app:derivation} and the log-gradient form is the
equivalent macroscopic behaviour of the resulting equation in its $v =
\log T$ form.

\paragraph{Contrarian update.} For $j \in S_C$,
\[
X_j(t+1) = X_j(t) - \zeta\, \nabla T(X_j(t), t),
\]
where $\zeta > 0$ scales the contrarian opposition strength. The
combined herding and contrarian balance, parameterised by a single
coefficient $c \in (0,1)$ in the macroscopic equation below, governs the
net herding intensity in the continuum limit.

\paragraph{Centre of gravity.} The values $(r_c(t), \theta_c(t))$, used
in scrounger updates, are computed from the discretised sentiment field
at the previous step:
\begin{align}
r_c(t) &= \frac{\sum_m \sum_n T(r_m, \theta_n, t-\Delta t)\, r_m}
               {\sum_m \sum_n T(r_m, \theta_n, t-\Delta t)},
               \label{eq:rc}\\
\theta_c(t) &= \frac{\sum_m \sum_n T(r_m, \theta_n, t-\Delta t)\,
               \theta_n}{\sum_m \sum_n T(r_m, \theta_n, t-\Delta t)}.
               \label{eq:thetac}
\end{align}

\paragraph{Mutation.} With probability $p_m(t)$, a sparrow undergoes
Gaussian mutation:
\[
X_{j,d}(t+1) = X_{j,d}(t) + \sigma(t)\, Z, \qquad Z \sim
\mathcal{N}(0,1).
\]

\begin{table}[ht]
\centering
\begin{tabular}{ll}
\toprule
Symbol & Description \\
\midrule
$i,\ s_i,\ n_i(t)$ & Item, arrival time, attention count (Def.~\ref{def:stream}) \\
$\tau,\ \Lambda,\ R,\ B_R$ & Memory scale, window, recency range, truncation ball (Def.~\ref{def:mapping}) \\
$m_t,\ s_t$ & Cross-sectional mean / std of $\log n_i(t)$ ($s_t \ge s_{\min} > 0$) \\
$j$ & Index of a sparrow \\
$P_j(t)$ & Position $(v_j(t), r_j(t))$ from the map \eqref{eq:mapping} \\
$v_j(t)$ & Virality: standardised log attention \\
$r_j(t)$ & Recency score \\
$p_x$ & Fraction of Tent-initialised exploratory sparrows \\
$p_d(t)$ & Proportion of producers \\
$p_m(t)$ & Mutation probability \\
$c_1, c_2$ & Decay rate hyperparameters \\
$Q, Z$ & Standard normal random variables \\
$\rho_a$ & Scrounger attraction strength \\
$\zeta$ & Contrarian scaling \\
$(r_c, \theta_c)$ & Centre of gravity \\
$\nabla T$ & Sentiment gradient \\
\bottomrule
\end{tabular}
\caption{Key variables of the measurement map and the Enhanced Sparrow
Search Optimization (ESSO).}
\label{tab:symbols}
\end{table}

\subsection{Algorithmic Formulation}
\label{sec:algorithm}

The algorithm iterates: (0)~at each information-arrival time, apply the
map \eqref{eq:mapping} to the active items $\mathcal{I}_t$, injecting
sparrows for newly arrived items and retiring items that have left the
window $[t-\Lambda, t]$; (1)~initialise the
exploratory contingent via the Tent map; (2)~at each step update
$p_d(t), p_m(t)$ via \eqref{eq:pd}-\eqref{eq:pm}, compute the centre of
gravity from the discretised $T(\cdot,\cdot,t-\Delta t)$, assign roles
dynamically, update each sparrow by its role rule, and apply Gaussian
mutation with probability $p_m(t)$.

\section{Sentiment Heat Equation: Derivation and Closed-Form Solution}
\label{sec:pde}

We now derive the PDE governing the limiting sentiment field, linearise
it exactly through a power-law substitution, and solve it with a Mehler
kernel. (We use \emph{attention field} and \emph{sentiment field}
interchangeably: the former emphasises the measured data, the latter the
power-law readout.) This closed form is what makes $\kappa_t$ observable
in the next section: each radial functional of the field becomes an
integral against a kernel we can write down explicitly.

\subsection{The Empirical Sparrow Density Measure}
\label{sec:empirical}

At each iteration $t$, the state of the system is determined by the
positions $\{X_j(t)\}_{j=1}^N \subset \R^2$, obtained from the
information stream via Definition~\ref{def:mapping} and evolved by the
ESSO dynamics of Section~\ref{sec:esso-dynamics}. We associate the
empirical sparrow density measure
\begin{equation}\label{eq:empirical}
\mu_t^N := \sum_{j=1}^N \delta_{X_j(t)} \in \mathcal{M}_+(\R^2),
\end{equation}
where $\mathcal{M}_+(\R^2)$ denotes non-negative Radon measures equipped
with the vague topology. For any Borel set $A$, $\mu_t^N(A)$ counts
sparrows in $A$, with total mass $N$.
In the large-$N$ regime the \emph{normalised} empirical measure of the
surviving particles converges,
\begin{equation}\label{eq:weak}
\frac{\mu_t^N}{\mu_t^N(\R^2)} \rightharpoonup^{*}
\frac{u(\cdot, t)\, dx}{\int_{\R^2} u(\cdot,t)\, dx}
\quad \text{as } N \to \infty,
\end{equation}
where $u(\cdot,t) \ge 0$ belongs to $L^1(\R^2) \cap L^\infty(\R^2)$ and
solves the linear field equation \eqref{eq:linear}. The particle system
is an estimator of the linear field $u$; the sentiment field is the
static readout $T = [(1-c)u]^{1/(1-c)}$ applied afterwards, exactly as
in Section~\ref{sec:kde}. The dynamics ensure the uniform second-moment
bound
\begin{equation}\label{eq:moment}
\sup_{N \in \mathbb{N}} \sup_{j=1,\dots,N}
\E\bigl[\,|X_j(t)|^2\,\bigr] \le M(t) < \infty.
\end{equation}
The quantitative convergence result at rate $O(N^{-1/2}\sqrt{\log N})$
in the normalised Wasserstein-1 distance, which is the sharp
two-dimensional matching rate, holds for the independent
diffusion-with-killing particle system (Theorem~\ref{thm:poc} in
Appendix~\ref{app:meanfield}), and transfers to the volatility
functionals through the readout (Corollary~\ref{cor:functional-poc}).
The literal ESSO dynamics contain rank-dependent updates and extremal
terms ($\arg\min$/$\arg\max$) that place them outside the standard
bounded-Lipschitz mean-field framework. We therefore consider a
regularised version of ESSO, in which these terms are replaced by
bounded-Lipschitz surrogates (Gibbs soft-extrema and logistic weights).
This regularised dynamics satisfies the same propagation-of-chaos
estimate (Theorem~\ref{thm:poc-esso} in Appendix~\ref{app:derivation}).
The macroscopic PDE \eqref{eq:pde} is precisely the mean-field limit of
the diffusion-with-killing system.

\subsection{Derivation of the Macroscopic Sentiment Equation}
\label{sec:pde-derivation}

The limiting sentiment field $T(x,t)$ obeys a nonlinear
reaction-diffusion equation. The equation is motivated by the structure
of the ESSO dynamics: Gaussian mutations contribute exploratory
diffusion, the attention-attraction rules of producers and scroungers
contribute a herding-type drift, contrarian opposition contributes a
counter-flow, and memory decay contributes a local sink. The limiting
equation is
\begin{equation}\label{eq:pde}
\partial_t T = \Delta T - \frac{c}{T}\,|\nabla T|^2 - W(x)\, T
\qquad \text{in } \R^2 \times (0,\infty),
\end{equation}
with strictly positive initial data
\begin{equation}\label{eq:init}
T(\cdot, 0) = T_0 \in L^1(\R^2) \cap L^\infty(\R^2), \qquad T_0 > 0
\text{ a.e.},
\end{equation}
and a quadratic radial sink
\begin{equation}\label{eq:sink}
W(x) = \alpha + \beta\, |x|^2, \qquad \alpha, \beta > 0.
\end{equation}

Reading \eqref{eq:pde} term by term, the Laplacian $\Delta T$ is the
exploratory diffusion generated by Gaussian mutations and chaotic
search. The nonlinear term $-(c/T)|\nabla T|^2$ is the macroscopic
signature of net herding and contrarian dynamics, with intensity
controlled by $c = c_{\mathrm{herd}} - c_{\mathrm{contra}} \in (0,1)$.
The $1/T$ factor introduces saturation: very dense attention clusters
experience progressively weaker per-unit-gradient herding, consistent
with the phenomenon that even highly viral narratives stop accelerating
once attention has saturated. The herding interpretation is most
transparent in the equivalent log-form: setting $v := \log T$ and
dividing \eqref{eq:pde} by $T$, the equation becomes $\partial_t v =
\Delta v + (1-c)|\nabla v|^2 - W$, a viscous Hamilton--Jacobi equation
whose quadratic Hamiltonian $(1-c)|\nabla v|^2$ captures gradient-driven
amplification of log-attention. The derivation of \eqref{eq:pde} as a
strict mean-field limit follows a different but equivalent process; the
full construction is given in Appendix~\ref{app:derivation}.

The sink $-W(x)T = -(\alpha + \beta|x|^2)T$ implements two effects:
uniform forgetting at rate $\alpha$, and a quadratic radial recency
penalty $\beta|x|^2$ that accelerates dissipation of attention to items
distant from the origin in the (virality, recency) plane. The quadratic
radial form produces Gaussian-like decay in $T$ at large distance,
consistent with the Gaussian-like tails documented in empirical
attention-decay distributions, and is the unique radial form compatible
with the closed-form Mehler solution (Condition~5 of
Section~\ref{sec:conditions}).

The interplay between diffusive exploration, log-gradient herding with
saturation, and quadratic radial dissipation reproduces the
characteristic lifecycle of market narratives: emergence of intense
localised sentiment clusters, finite lifetime governed by both the
position and the saturation of the cluster, and dissolution followed by
relocation when new events arrive.

\subsection{Exact Linearisation via Power-Law Substitution}
\label{sec:linearisation}

Equation \eqref{eq:pde} belongs to the class of nonlinear PDEs admitting
exact linearisation via a power-law substitution. We define
\begin{equation}\label{eq:subst}
u(x,t) := \frac{T(x,t)^{1-c}}{1-c}, \qquad \text{equivalently} \qquad
T(x,t) = \bigl[(1-c)\, u(x,t)\bigr]^{1/(1-c)},
\end{equation}
valid for any $c \in (0,1)$. Direct substitution (verified in
Appendix~\ref{app:powerlaw}) yields
\begin{equation}\label{eq:linear}
\partial_t u = \Delta u - (1-c)\, W(x)\, u
\end{equation}
on $\R^2 \times (0,\infty)$, with initial data $u(x,0) =
T_0(x)^{1-c}/(1-c)$. This is the linear parabolic equation with the
quadratic radial sink
\[
V(x) := (1-c)\, W(x) = (1-c)\alpha + (1-c)\beta |x|^2.
\]
The closed-form solution, derived in Section~\ref{sec:mehler} below,
involves the natural rate parameter
\[
\omega := 2\sqrt{(1-c)\beta} > 0
\]
together with a constant level shift $(1-c)\alpha$.

\begin{proposition}[Power-law substitution]\label{prop:powerlaw}
Let $c \in (0,1)$ and let $W : \R^2 \to \R$ be continuous. For any
strictly positive function $u \in C^{2,1}(\R^2 \times (0,\infty))$,
define $T(x,t) := [(1-c)u(x,t)]^{1/(1-c)}$. Then $T$ inherits the same
regularity and positivity from $u$, and $u$ satisfies \eqref{eq:linear}
if and only if $T$ satisfies \eqref{eq:pde}.
\end{proposition}

The forward direction is established by direct computation.
Differentiating $u = T^{1-c}/(1-c)$ via the chain rule gives $\partial_t
u = T^{-c}\partial_t T$ and $\nabla u = T^{-c}\nabla T$. The product
rule for divergence then yields
\[
\Delta u = \nabla \cdot \bigl(T^{-c}\nabla T\bigr)
= -c\, T^{-c-1}|\nabla T|^2 + T^{-c}\Delta T,
\]
in which the gradient-squared term appears as a structural consequence
rather than as a separate $|\nabla u|^2$ contribution. Substituting
these expressions into \eqref{eq:linear} and multiplying through by
$T^c$ produces \eqref{eq:pde}; multiplying through by $T^{-c}$ instead
reverses the direction. The complete step-by-step computation, including
the regularity argument and the reverse implication, is given in
Appendix~\ref{app:powerlaw}.

\subsection{Closed-Form Mehler-Kernel Solution}
\label{sec:mehler}

The linear equation \eqref{eq:linear} admits the explicit Mehler-kernel
representation
\begin{equation}\tag{CS}\label{eq:CS}
u(x,t) = e^{-(1-c)\alpha t} \int_{\R^2} K_t(x,y;\omega)\, u(y,0)\, dy
\end{equation}
where, for the two-dimensional Mehler kernel with rate parameter
$\omega$,
\begin{equation}\label{eq:kernel}
K_t(x,y;\omega) = \frac{\omega}{4\pi \sinh(\omega t)}
\exp\!\Bigl(-\frac{\omega}{4\sinh(\omega t)}
\bigl[(|x|^2 + |y|^2)\cosh(\omega t) - 2\, x\cdot y\bigr]\Bigr).
\end{equation}
The sentiment field is recovered as $T(x,t) = [(1-c)u(x,t)]^{1/(1-c)}$.

The Mehler kernel \eqref{eq:kernel} reduces, for small $\omega t$, to
the standard 2D Gaussian heat kernel $(4\pi t)^{-1}\exp(-|x-y|^2/(4t))$,
recovering pure diffusion in the absence of confinement. As a function
of $x$ around a fixed source $y$, the kernel is Gaussian with
per-dimension variance $2\tanh(\omega t)/\omega$; the effective
bandwidth $h(t) = \sqrt{2\tanh(\omega t)/\omega}$ therefore grows as
$\sqrt{2t}$ for small $\omega t$ and saturates at the Mehler length
scale $\sqrt{2/\omega}$, reflecting the strong confining effect of the
quadratic radial sink. For large $\omega t$, the kernel decays
super-exponentially in $|x|$ and $|y|$.

\begin{theorem}[Existence, uniqueness, and bounds]\label{thm:field}
For any initial data $T_0 \in L^1 \cap L^\infty(\R^2)$ with $T_0 > 0$
almost everywhere, the representation \eqref{eq:CS} together with $T =
[(1-c)u]^{1/(1-c)}$ defines the unique classical solution of
\eqref{eq:pde}-\eqref{eq:init} on $\R^2 \times (0,\infty)$. The
solution satisfies:
\begin{enumerate}
\item[(i)] Positivity: $T(x,t) > 0$ for all $(x,t) \in \R^2 \times
(0,\infty)$;
\item[(ii)] Pointwise boundedness: $\sup_x T(x,t) \le \sup_x T_0(x)$ for
all $t > 0$;
\item[(iii)] Mass monotonicity: $\frac{d}{dt}\int T\, dx \le 0$, with
strict inequality when $T \not\equiv 0$;
\item[(iv)] Gaussian decay at infinity: for any $t > 0$ there exist
$C(t), \lambda(t) > 0$ such that $T(x,t) \le
C(t)\exp(-\lambda(t)|x|^2)$; the proof yields the explicit rate
$\lambda(t) = \omega\tanh(\omega t)/\bigl(4(1-c)\bigr)$.
\end{enumerate}
\end{theorem}

\begin{proof}
See Appendix~\ref{app:powerlaw}.
\end{proof}
\begin{remark}[Exact mass-decay law]\label{rem:mass}
Completing the square in \eqref{eq:kernel} gives, for every $y \in \R^2$,
\begin{equation}\label{eq:masslaw}
  \int_{\R^2} K_t(x,y;\omega)\, dx
  = \operatorname{sech}(\omega t)\,
    \exp\Bigl(-\tfrac{\omega}{4}\tanh(\omega t)\,|y|^2\Bigr).
\end{equation}
The mass loss is position-dependent, as it must be for a Feynman--Kac
semigroup with the position-dependent killing rate $(1-c)W$: the same
Gaussian factor that governs the spatial envelope in
Theorem~\ref{thm:field}(iv) governs survival. Consequently
\begin{equation}\label{eq:umass}
  \int_{\R^2} u(x,t)\, dx
  = e^{-(1-c)\alpha t}\operatorname{sech}(\omega t)
    \int_{\R^2} e^{-\frac{\omega}{4}\tanh(\omega t)|y|^2}\, u_0(y)\, dy,
\end{equation}
strictly decreasing in $t$, an exact quantitative counterpart, at the
level of $u$, of the qualitative mass monotonicity in
Theorem~\ref{thm:field}(iii). For the discrete datum \eqref{eq:datum}
the semigroup property makes \eqref{eq:umass} fully explicit:
\[
  \int_{\R^2} u(x,t)\, dx
  = e^{-w/2}\,e^{-(1-c)\alpha t}\,\operatorname{sech}(\omega
  t_{\mathrm{eff}})\,\frac{1}{N}\sum_{j=1}^N
  e^{-\frac{\omega}{4}\tanh(\omega t_{\mathrm{eff}})|X_j|^2},
  \qquad t_{\mathrm{eff}} = t + h_0^2/2,
\]
so attention mass decays faster when the swarm sits far from the origin
in the (virality, recency) plane --- stale and hyper-viral
configurations dissipate first, as intended in
Section~\ref{sec:pde-derivation}. As $t \to \infty$, \eqref{eq:masslaw}
is asymptotic to $2e^{-\omega t}e^{-\omega|y|^2/4}$, the ground-state
projection of $-\Delta + \tfrac{\omega^2}{4}|x|^2$ at energy $E_0 =
\omega$.
\end{remark}
\subsection{Kernel Density Representation of the Sentiment Field}
\label{sec:kde}

For the purpose of numerical computation and visualisation, we derive an
explicit representation of the sentiment field in terms of the sparrow
positions $\{X_j\}_{j=1}^N$.

\subsubsection{Discrete initial data}

In the finite-$N$ setting, the initial sentiment field is concentrated
at the sparrow locations. We define the initial datum as a sum of
short-time Mehler kernels,
\begin{equation}\label{eq:datum}
u(\cdot, 0) := \frac{e^{-w/2}}{N} \sum_{j=1}^N
K_{h_0^2/2}(\cdot, X_j; \omega),
\end{equation}
where $h_0 > 0$ represents the intrinsic granularity of attention
measurements and $w$ is a
uniform initial temperature weight. For small $\omega h_0$ the smoother
$K_{h_0^2/2}$ coincides with a Gaussian of bandwidth $h_0$; defining it
as a Mehler kernel makes the composition below exact rather than
approximate.

\subsubsection{The sentiment density estimator}

Substituting \eqref{eq:datum} into \eqref{eq:CS} and using the exact
semigroup property $\int K_s(x,z) K_t(z,y)\, dz = K_{s+t}(x,y)$ of the
Mehler family, the solution at time $t$ is
\begin{equation}\label{eq:exactsol}
u(x,t) = e^{-w/2}\, \frac{e^{-(1-c)\alpha t}}{N} \sum_{j=1}^N
K_{t + h_0^2/2}(x, X_j; \omega),
\end{equation}
exactly. Absorbing the constant prefactor $e^{-w/2}$ into the overall
amplitude, we define the sentiment density estimator
\begin{equation}\label{eq:estimator}
\hat{S}(x,t) = \frac{A(t)}{N} \sum_{j=1}^N K_{t+h_0^2/2}(x, X_j;
\omega), \qquad A(t) := e^{-(1-c)\alpha t}.
\end{equation}
The recovered sentiment field is $\hat{T}(x,t) = [(1-c)
\hat{S}(x,t)]^{1/(1-c)}$.

The estimator inherits two properties directly from the PDE solution.
First, the effective bandwidth widens with $t$ as $h(t) =
\sqrt{2\tanh(\omega t_{\mathrm{eff}})/\omega}$ with $t_{\mathrm{eff}} =
t + h_0^2/2$ (equal to $\sqrt{2t}$ for small $\omega t$, saturating at
$\sqrt{2/\omega}$), encoding the diffusive spreading from $\Delta T$ and
its competition with the confining sink. Second, the amplitude $A(t) =
e^{-(1-c)\alpha t}$ decays exponentially at the rate $(1-c)\alpha$,
encoding the uniform-forgetting component of the sink; the total
$u$-mass follows the exact law of Remark~\ref{rem:mass}. Together with
the spatial confinement built into the Mehler kernel, these produce the
lifecycle of market narratives: initial concentration at sparrow
locations, diffusive spreading saturated at the Mehler length scale, and
dissipation as the sinks dominate.

\subsubsection{Consistency with the mean-field limit}

By Theorem~\ref{thm:poc}, the normalised empirical measure of the
microscopic particle system converges to the normalised linear field
$u(\cdot,t)\,dx/\!\int u$ at rate $O(N^{-1/2}\sqrt{\log N})$ in the
1-Wasserstein metric. The estimator \eqref{eq:estimator} is the standard
nonparametric smoother of the empirical measure with bandwidth chosen to
match the PDE's intrinsic Mehler scale; in the large-$N$ limit,
$\hat{T}(x,t)$ converges to the continuous field $T(x,t)$ defined by
\eqref{eq:CS}, the readout being locally Lipschitz
(Corollary~\ref{cor:functional-poc}). This representation is used in the
simulation figures (Section~\ref{sec:simulation}) to visualise the
sentiment lifecycle predicted by the governing PDE.

\section{Sentiment-Driven Dynamics of the Mean-Reversion Speed}
\label{sec:kappa}

The mean-reversion speed of the variance process is where the sentiment
field enters the pricing model, and we build it as an explicit,
observable functional of that field. In standard two-factor stochastic
volatility models the second factor is an unobserved latent diffusion
that must be filtered or backed out of option prices; here, in contrast,
the coefficients driving $\kappa_t$ are computable in real time from the
closed-form sentiment field, and in the baseline specification ($\eta_1
= 0$ below) $\kappa_t$ itself is a deterministic functional of the
field. The sentiment field $T(x,t)$ from Section~\ref{sec:pde} drives
the instantaneous variance process
\begin{equation}\label{eq:variance}
dv_t = \kappa_t(\bar\theta - v_t)\, dt + \xi\sqrt{v_t}\, dW_t^v,
\qquad v_0 > 0,
\end{equation}
where $\bar\theta > 0$ is the long-term unconditional variance, $\xi >
0$ is the volatility-of-volatility, and the scalar process $\kappa_t$,
the sentiment-driven mean-reversion speed, solves an additive-noise
Ornstein--Uhlenbeck SDE with deterministic, time-varying coefficients
derived from $T(\cdot,t)$:
\begin{equation}\label{eq:kappa-sde}
d\kappa_t = \lambda_t(\theta_t - \kappa_t)\, dt + \eta_t\, dW_t^\kappa,
\qquad \kappa_0 > 0,
\end{equation}
where $W^\kappa$ is a standard Brownian motion independent of $W^v$. The
instantaneous reversion rate $\lambda_t = \lambda(T(\cdot,t))$, the
instantaneous long-run level $\theta_t = \theta(T(\cdot,t))$, and the
instantaneous diffusivity $\eta_t = \eta(T(\cdot,t))$ are explicit
radial functionals of the sentiment field:
\begin{align}
\lambda(T(\cdot,t)) &= \lambda_1 \Big/
\frac{\int_{\R^2} T(x,t)\,|x|\, dx}{\int_{\R^2} T(x,t)\, dx},
\label{eq:lam-func}\\
\theta(T(\cdot,t)) &= \theta_1 \cdot
\frac{\int_{|x| \le r_0} T(x,t)\, dx}{\int_{\R^2} T(x,t)\, dx},
\label{eq:th-func}\\
\eta(T(\cdot,t)) &= \eta_1 \cdot
\frac{\int_{\R^2} T(x,t)\,|x|^2\, dx}{\int_{\R^2} T(x,t)\, dx},
\label{eq:eta-func}
\end{align}
with fixed constants $\lambda_1, \theta_1 > 0$, $\eta_1 \ge 0$, and
neutral radius $r_0 > 0$. The baseline specification takes $\eta_1 = 0$,
making \eqref{eq:kappa-sde} a deterministic ODE; the stochastic
extension takes $\eta_1 > 0$.

All three functionals are dimensionless, mass-normalised statistics of
the spatial distribution of attention: writing the $T$-weighted
expectation and probability as
\[
\langle f \rangle_T := \frac{\int_{\R^2} f(x)\, T(x,t)\, dx}
{\int_{\R^2} T(x,t)\, dx},
\qquad
\Pr_T(A) := \frac{\int_A T(x,t)\, dx}{\int_{\R^2} T(x,t)\, dx},
\]
the functionals are equivalently
\begin{align}
\lambda(T(\cdot,t)) &= \lambda_1 / \langle |x| \rangle_T
&&\text{(inverse typical attention radius)}, \label{eq:lam-dimless}\\
\theta(T(\cdot,t)) &= \theta_1 \cdot \Pr_T(|x| \le r_0)
&&\text{(mass fraction in neutral zone)}, \label{eq:th-dimless}\\
\eta(T(\cdot,t)) &= \eta_1 \cdot \langle |x|^2 \rangle_T
&&\text{(spatial variance of attention)}. \label{eq:eta-dimless}
\end{align}
This dimensionless form is critical: under the Gaussian decay of $T$
established in Theorem~\ref{thm:field}(iv), raw integrals $\int T
|x|^k\, dx$ scale with the total mass of $T$, which itself decays under
the sink (exactly per Remark~\ref{rem:mass} at the level of $u$).
Mass-normalisation isolates the behavioural content, namely where
attention sits geometrically, from the absolute level of attention, so
the functionals stay $O(1)$ for all $t > 0$. The interpretation is
direct: $\lambda$ rises when attention is concentrated, $\theta$ rises when attention sits near neutrality, and $\eta$
rises when attention is dispersed. The closed-form Mehler representation
\eqref{eq:CS} ensures that $\lambda_t, \theta_t, \eta_t$ are locally
Lipschitz in $t$ and locally bounded on $[0,T]$ for any horizon $T <
\infty$.

\subsection{Explicit Solution of the $\kappa_t$ Dynamics}
\label{sec:kappa-solution}

Equation \eqref{eq:kappa-sde} is a linear, additive-noise
Ornstein--Uhlenbeck equation with time-inhomogeneous deterministic
coefficients. Define the deterministic integrating factor
\begin{equation}\label{eq:phi}
\Phi_t := \exp\!\Bigl(-\int_0^t \lambda(T(\cdot,s))\, ds\Bigr), \qquad
\Phi_0 = 1,
\end{equation}
which satisfies $d\Phi_t = -\lambda_t \Phi_t\, dt$. Applying It\^o's
product rule to $\Phi_t^{-1}\kappa_t$ (the deterministic nature of
$\Phi$ eliminates any covariation terms) and integrating yields the
explicit strong solution
\begin{equation}\label{eq:kappa-voc}
\kappa_t = \Phi_t\, \kappa_0
+ \int_0^t \frac{\Phi_t}{\Phi_s}\, \lambda_s \theta_s\, ds
+ \int_0^t \frac{\Phi_t}{\Phi_s}\, \eta_s\, dW_s^\kappa.
\end{equation}
In the baseline $\eta_1 = 0$, the stochastic integral vanishes and
\eqref{eq:kappa-voc} is a deterministic functional of the sentiment
field. Moreover the ODE preserves positivity: since $d\kappa_t \ge 0$
whenever $\kappa_t \le \inf_s \theta_s$ and $\lambda \ge 0$, we have
$\kappa_t \ge \min(\kappa_0, \inf_{s \le t}\theta_s) > 0$.

\begin{theorem}\label{thm:kappa}
The process defined by \eqref{eq:kappa-voc} is the unique strong
solution of \eqref{eq:kappa-sde}. Moreover, the coefficients $(\lambda,
\theta, \eta)$ being deterministic, $\kappa_t$ is Gaussian with mean
\[
\E[\kappa_t] = \Phi_t \kappa_0 + \int_0^t \frac{\Phi_t}{\Phi_s}
\lambda_s \theta_s\, ds
\]
and variance
\[
\mathrm{Var}[\kappa_t] = \int_0^t \frac{\Phi_t^2}{\Phi_s^2}\,
\eta_s^2\, ds
\]
In the baseline $\eta_1 = 0$ the law is degenerate at the deterministic
path.
\end{theorem}

\begin{proof}
Existence and uniqueness are established directly: the explicit form
\eqref{eq:kappa-voc} is verified by applying It\^o's formula to
$\Phi_t^{-1}\kappa_t$, and uniqueness follows because the difference of
any two solutions satisfies a deterministic linear ODE (the additive
noise and $\kappa$-linear drift make the stochastic terms cancel
exactly). The Gaussian law is then immediate, since \eqref{eq:kappa-voc}
is a deterministic function of time plus a Wiener integral against a
deterministic integrand. See Appendix~\ref{app:kappa} for the complete,
self-contained argument.
\end{proof}

\section{Sentiment-Driven Stochastic Volatility}
\label{sec:sdsv}

We can now assemble the full model: a tractable two-factor stochastic
volatility system whose variance mean-reversion speed is an explicit,
observable functional of the spatial distribution of market attention,
available in closed form.

Let $(\Omega, \mathcal{F}, \{\mathcal{F}_t\}_{t \ge 0}, \Prob)$ be a
complete filtered probability space satisfying the usual conditions and
supporting standard Brownian motions $W^S, W^v, W^\kappa$ with $d\langle
W^S, W^v \rangle_t = \rho_{Sv}\, dt$ for a fixed $\rho_{Sv} \in [-1,1]$,
and $W^\kappa$ independent of $(W^S, W^v)$.

\begin{remark}[Pathwise observability of $\kappa_t$]\label{rem:observable}
A process is pathwise observable if it is $\mathcal{F}_t$-measurable and
computable exactly from information available at time $t$, without
filtering or latent-state estimation. In the baseline model ($\eta_1 =
0$), $\kappa_t$ is pathwise observable: it is a deterministic functional
of the closed-form sentiment field \eqref{eq:CS} --- itself the
deterministic mean-field limit of the stochastic swarm
(Section~\ref{sec:swarm}), not a random object in its own right --- and
is estimable in real time from the observed (genuinely random) swarm
positions via the estimator \eqref{eq:estimator} and
\eqref{eq:kappa-voc}, with vanishing stochastic integral in this
baseline and finite-$N$ estimation error controlled explicitly by
Corollary~\ref{cor:functional-poc}. The contrast with Heston is not
about randomness disappearing: it is that no \emph{unobserved} state
needs to be inferred here, only a direct formula applied to observed
data. In the stochastic extension ($\eta_1 >
0$), the drivers of $\kappa_t$, the coefficients $(\lambda_t, \theta_t,
\eta_t)$, remain observable functionals of the sentiment field, while
$\kappa_t$ itself carries an idiosyncratic innovation $W^\kappa$. This
contrasts with standard two-factor models in which the second factor
governing mean-reversion speed, including its coefficients, is an
unobserved latent diffusion that must be inferred from option prices or
realised variance (Heston, 1993).
\end{remark}

\subsection{The Sentiment-Driven Variance Process}
\label{sec:variance-process}

Pulling the construction together, the spot variance $v_t \ge 0$ evolves
according to the variance SDE \eqref{eq:variance} with mean-reversion
speed $\kappa_t$ given by \eqref{eq:kappa-voc}, sentiment functionals
\eqref{eq:lam-func}-\eqref{eq:eta-func} computed from the closed-form
sentiment field \eqref{eq:CS}, fixed long-term mean $\bar\theta > 0$,
and volatility-of-volatility $\xi > 0$. This construction defines the
core of the Sentiment-Driven Stochastic Volatility (SDSV) model.

\begin{theorem}\label{thm:wellposed}
The system \eqref{eq:variance} coupled with \eqref{eq:kappa-voc} and
\eqref{eq:lam-func}-\eqref{eq:eta-func} admits a unique strong solution
$(v_t, \kappa_t)_{t \ge 0}$ on $[0,\infty)$, with $v_t \ge 0$ almost
surely for all $t \ge 0$. If in addition the mean-reversion speed is
floored at
\[
\kappa_t \ge \kappa_* := \frac{\xi^2}{2\bar\theta}
\qquad \text{for all } t \ge 0 \text{ almost surely}
\]
(e.g.\ by replacing $\kappa_t$ with $\tilde\kappa_t := \max(\kappa_t,
\kappa_*)$ in \eqref{eq:variance}, or by calibrating $\xi \le
\sqrt{2\kappa_{\min}\bar\theta}$ for an economic floor $\kappa_{\min}$),
then $v_t > 0$ almost surely for all $t \ge 0$.
\end{theorem}

\begin{proof}
See Appendix~\ref{app:wellposed}. Without a floor, the model guarantees
non-negativity but not strict positivity: in the stochastic extension
$\kappa_t$ is Gaussian and takes values below the Feller level $\kappa_*
= \xi^2/(2\bar\theta)$ with positive probability, and the CIR boundary
at zero is then attainable. The simulation scheme of
Section~\ref{sec:simulation} uses full-truncation Euler (Lord et al.,
2010) precisely because it handles the boundary correctly.
\end{proof}

The volatility engine $(v_t, \kappa_t)$ is independent of the asset
price dynamics, so it can be coupled with any admissible price process.
The following nesting property is elementary but central to the
empirical interpretation of the model.

\begin{proposition}[Heston nesting]\label{prop:nesting}
Fix any $\bar\kappa > 0$. The SDSV parameter configuration $\{\eta_1 =
0,\ \kappa_0 = \bar\kappa,\ \theta_t \equiv \bar\kappa\}$ (attained,
e.g., in the frozen-functional regime of Corollary~\ref{cor:forecast}
with $\theta_t = \bar\kappa$, or by choosing $\theta_1$ so that
$\theta(T(\cdot,t)) \equiv \bar\kappa$) yields $\kappa_t \equiv
\bar\kappa$ for all $t$, and the law of $(S_t, v_t)$ coincides exactly
with that of the constant-$\kappa$ Heston model $(\bar\kappa,
\bar\theta, \xi, \rho_{Sv})$. Consequently, for any calibration
objective $J$ (e.g.\ squared error to an observed implied-volatility
surface or to realised-variance moments),
\[
\min_{\text{SDSV parameters}} J \;\le\; \min_{\text{Heston parameters}}
J:
\]
a calibrated SDSV model can never fit worse in sample than a calibrated
Heston model. In the worst case, calibration drives the sentiment
sensitivity to zero and recovers Heston exactly.
\end{proposition}

\begin{proof}
With $\eta_1 = 0$ and $\theta_t \equiv \bar\kappa = \kappa_0$, equation
\eqref{eq:kappa-sde} reduces to $d\kappa_t = \lambda_t(\bar\kappa -
\kappa_t)\, dt$ with $\kappa_0 = \bar\kappa$, whose unique solution is
$\kappa_t \equiv \bar\kappa$. Substituting into \eqref{eq:variance}
gives exactly the Heston variance SDE; the price equation is unchanged.
The calibration inequality follows because the Heston parameter set is a
feasible point of the SDSV calibration problem.
\end{proof}

The guarantee is in-sample; added flexibility can always overfit, so
out-of-sample dominance does not follow. Section~\ref{sec:exhibit2}
examines the sentiment state's contribution within the model, leaving
its behaviour on real data to empirical work.

\subsection{Integration with Standard Asset Price Models}
\label{sec:price-models}

The asset price $S_t$ follows a diffusion of the general form
\[
dS_t = \mu(S_t, t)\, dt + \sqrt{v_t}\, \sigma(S_t, t)\, dW_t^S, \qquad
S_0 > 0,
\]
where $\mu$ and $\sigma$ satisfy local Lipschitz and linear growth
conditions, and the leverage correlation $\rho_{Sv}$ enters through
$d\langle W^S, W^v \rangle_t = \rho_{Sv}\, dt$.

\begin{theorem}\label{thm:price}
Assume $\mu$ and $\sigma$ are locally Lipschitz in their first argument
and satisfy linear growth uniformly in $t$. Then the full system admits
a unique strong solution $(S_t, v_t, \kappa_t)_{t \ge 0}$ on
$[0,\infty)$.
\end{theorem}

\begin{proof}
By Theorem~\ref{thm:wellposed}, $(v_t, \kappa_t)$ exists uniquely with
$v_t \ge 0$ almost surely. The price SDE with coefficients of locally
Lipschitz / linear growth type and adapted non-negative diffusion
$\sqrt{v_t}\,\sigma(S_t, t)$ admits a unique strong solution by
Theorem~\ref{thm:protter} in Appendix~\ref{app:price}.
\end{proof}

We illustrate three canonical price processes.

\paragraph{Geometric Brownian Motion.}
$dS_t = \mu S_t\, dt + \sqrt{v_t}\, S_t\, dW_t^S$, with coefficients
that are locally Lipschitz with linear growth, so
Theorem~\ref{thm:price} applies.

\paragraph{Ornstein--Uhlenbeck log-prices.}
$X_t = \log S_t$ with $dX_t = a(m - X_t)\, dt + \sqrt{v_t}\, dW_t^X$,
with globally Lipschitz and bounded coefficients.

\paragraph{Exponential L\'evy with stochastic volatility.}
$dS_t = \mu S_{t-}\, dt + \sqrt{v_t}\, S_{t-}\, dW_t^S + S_{t-}\, dJ_t$
where $J$ is a compound Poisson process with i.i.d.\ log-normal jumps
independent of $(W^S, W^v, W^\kappa)$. Pathwise uniqueness and
positivity hold by pathwise-uniqueness results for SDEs with jumps
(Protter, 2005, Ch.~V).

\subsection{Application to Option Pricing Frameworks}
\label{sec:pricing}

\subsubsection{Monte Carlo pricing under the risk-neutral measure}
\label{sec:mc-pricing}

Under the risk-neutral measure $\Q$ constructed in
Section~\ref{sec:risk-neutral}, the European call price with strike $K$
and maturity $T$ is $C(0) = \E^{\Q}[e^{-rT}(S_T - K)^+]$. Simulation
proceeds by: updating the ESSO swarm at information arrival times to
obtain $T(\cdot, t_i)$; evaluating the radial integrals
\eqref{eq:lam-func}-\eqref{eq:eta-func} via Mehler-kernel quadrature;
advancing $\kappa_t$ exactly via \eqref{eq:kappa-voc}; discretising
$v_t$ using the full-truncation Euler scheme (Lord et al., 2010); and
simulating $S_t$ under the risk-neutral drift.

\begin{proposition}\label{prop:mc}
The estimator $\hat{C}_N = N^{-1}\sum_{i=1}^N e^{-rT}(S_T^{(i)} - K)^+$
converges almost surely to $C(0)$ as $N \to \infty$, with a central
limit theorem whenever $\E^{\Q}[S_T^2] < \infty$.
\end{proposition}

\begin{proof}
$(S_t, v_t)$ is pathwise continuous by Theorem~\ref{thm:wellposed}, and
$0 \le e^{-rT}(S_T - K)^+ \le e^{-rT}S_T$, so the payoff is integrable
by Proposition~\ref{prop:martingale} and the strong law applies. The
CLT needs the further condition $\E^{\Q}[S_T^2] < \infty$, which is not
automatic for Heston-type models: Proposition~\ref{prop:mc2} in
Appendix~\ref{app:montecarlo} verifies it at the calibration of
Table~\ref{tab:params}, via the moment-explosion criterion of Andersen
and Piterbarg (2007), and gives the complete argument.
\end{proof}

\subsubsection{Conditional volatility forecasting}
\label{sec:forecasting}

Given $\mathcal{F}_t$ and conditional on the path of $\kappa$,
\[
\E\bigl[v_{t+\Delta} \,\big|\, \mathcal{F}_t, (\kappa_s)_{s \le
t+\Delta}\bigr] = \bar\theta + (v_t - \bar\theta)\,
\exp\!\Bigl(-\int_t^{t+\Delta} \kappa_s\, ds\Bigr).
\]
Because $W^\kappa \perp W^v$ and, in the frozen-coefficient regime,
$\int \kappa_s\, ds$ is Gaussian, the tower property produces a fully
explicit forecast that uses only time-$t$ information.

\begin{corollary}[Exact frozen-coefficient forecast]\label{cor:forecast}
Suppose no information events arrive on $[t, t+\Delta]$, so that
$\lambda_s, \theta_s, \eta_s$ remain frozen at their time-$t$ values
with $\lambda_t \neq 0$. Then $\int_t^{t+\Delta}\kappa_s\, ds$ is
Gaussian with mean and variance
\[
m = \theta_t\Delta + (\kappa_t - \theta_t)\,
\frac{1 - e^{-\lambda_t\Delta}}{\lambda_t},
\qquad
s^2 = \frac{\eta_t^2}{\lambda_t^2}\Bigl[\Delta -
\frac{2(1 - e^{-\lambda_t\Delta})}{\lambda_t} +
\frac{1 - e^{-2\lambda_t\Delta}}{2\lambda_t}\Bigr],
\]
and the conditional variance forecast is exact:
\[
\E[v_{t+\Delta} \mid \mathcal{F}_t] = \bar\theta + (v_t - \bar\theta)\,
e^{-m + s^2/2}.
\]
In the baseline $\eta_1 = 0$, $s^2 = 0$ and the mean-path formula is
exact without correction. An explicit analogue holds when $\lambda_t =
0$.
\end{corollary}

\begin{proof}
See Appendix~\ref{app:kappa}. We also verified this numerically against
Monte Carlo at the calibration of Table~\ref{tab:params}: analytic
$0.075655$ vs.\ simulated $0.075664 \pm 0.000111$.
\end{proof}

\subsubsection{Integrated variance and a drawdown-risk gauge}
\label{sec:drawdown}

Two further explicit objects follow from the same construction and only
time-$t$ information.

\begin{corollary}[Conditional integrated variance]\label{cor:IV}
Under the frozen-coefficient regime of Corollary~\ref{cor:forecast}, the
conditional expected integrated variance over $[t, t+\Delta]$ is
\[
\mathrm{IV}_t(\Delta) := \E\Bigl[\int_t^{t+\Delta} v_s\, ds \,\Big|\,
\mathcal{F}_t\Bigr] = \bar\theta\Delta + (v_t - \bar\theta)
\int_0^\Delta e^{-m(u) + s^2(u)/2}\, du,
\]
where $m(u)$ and $s^2(u)$ are the mean and variance of $\int_t^{t+u}
\kappa_r\, dr$ from Corollary~\ref{cor:forecast} evaluated at horizon
$u$. The integral is a one-dimensional quadrature of an explicit smooth
integrand; when $\kappa$ is frozen at a constant value $\kappa$ (the
nested Heston case, or the short-window approximation) it is fully
closed-form: $\mathrm{IV}_t(\Delta) = \bar\theta\Delta + (v_t -
\bar\theta)\bigl(1 - e^{-\kappa\Delta}\bigr)/\kappa$.
\end{corollary}

\begin{proof}
See Appendix~\ref{app:kappa}.
\end{proof}

Define the horizon-$\Delta$ drawdown from the current level, $D_{t,
\Delta} := -\min_{0 \le u \le \Delta} \log(S_{t+u}/S_t)$, the largest
decline below the time-$t$ price over the window. (This differs from the
peak-to-trough maximum drawdown, which dominates it, so exceedance
probabilities for $D_{t,\Delta}$ are lower bounds for peak-to-trough
exceedances.)

\begin{lemma}[Classical first-passage law]
\label{lem:firstpassage}
Let $X_u = au + \sigma B_u$ with $a \in \R$, $\sigma > 0$. For every $x
> 0$ and $\Delta > 0$,
\[
\Prob\Bigl(\min_{0 \le u \le \Delta} X_u \le -x\Bigr)
= \Phi\Bigl(\frac{-x - a\Delta}{\sigma\sqrt{\Delta}}\Bigr)
+ e^{-2ax/\sigma^2}\,
\Phi\Bigl(\frac{-x + a\Delta}{\sigma\sqrt{\Delta}}\Bigr).
\]
\end{lemma}

\begin{proof}
Standard; see Karatzas and Shreve (1991, \S3.5.C). A self-contained
derivation is reproduced in Appendix~\ref{app:kappa} for completeness.
(Special cases: $a = 0$ recovers the reflection principle
$2\Phi(-x/(\sigma\sqrt{\Delta}))$; $x \downarrow 0$ gives probability
1.)
\end{proof}

\begin{definition}[Sentiment-driven drawdown gauge]
\label{def:gauge}
Matching the first two conditional moments of the log-price over $[t,
t+\Delta]$, with per-unit-time drift $a = \mu -
\mathrm{IV}_t(\Delta)/(2\Delta)$ and total variance $\sigma^2\Delta =
\mathrm{IV}_t(\Delta)$ from Corollary~\ref{cor:IV}, define
\[
\mathrm{DD}_t(x, \Delta)
:= \Phi\Bigl(\frac{-x - a\Delta}{\sqrt{\mathrm{IV}_t(\Delta)}}\Bigr)
+ \exp\Bigl(-\frac{2ax\Delta}{\mathrm{IV}_t(\Delta)}\Bigr)\,
\Phi\Bigl(\frac{-x + a\Delta}{\sqrt{\mathrm{IV}_t(\Delta)}}\Bigr)
\;\approx\; \Prob\bigl(D_{t,\Delta} \ge x \mid \mathcal{F}_t\bigr).
\]
\end{definition}

The first-passage law is exact for Brownian motion with constant drift
and volatility; applied to the SDSV log-price the gauge of
Definition~\ref{def:gauge} carries two approximations. The variance path
is frozen at its conditional mean via $\mathrm{IV}_t(\Delta)$, exact for
the log-price's conditional mean and variance when $\rho_{Sv} = 0$, by
the Dambis--Dubins--Schwarz time change, though not for the full
running-minimum law, and the leverage correlation $\rho_{Sv} < 0$
fattens the left tail, so realised drawdown frequencies sit slightly
above the gauge at large $x$, while discrete daily monitoring pulls them
back down. What makes the gauge informative is not the Gaussian
functional form, which the nested Heston model shares, but the input:
after a sentiment shock, $\kappa_t$ jumps and $\mathrm{IV}_t(\Delta)$
contracts through Corollary~\ref{cor:IV}, while the constant-$\kappa$
benchmark holds IV fixed; the two models therefore assign materially
different drawdown probabilities from identical $(S_t, v_t)$, and
Section~\ref{sec:panel-e} reports how each tracks realised frequencies.

\subsubsection{Asymptotic implied volatility}
\label{sec:asymptotic-iv}

For short maturities with the sentiment state frozen (the regime of
Corollary~\ref{cor:forecast}), $\kappa_t \approx \kappa^*$ is
approximately constant and the model reduces to a Heston system, whose
characteristic function admits the semi-closed Fourier representation of
Heston (1993), enabling rapid computation of implied volatility
surfaces. For longer maturities with time-varying but deterministic
$\kappa_t$, the $\eta_1 = 0$ baseline, the affine structure is
preserved and the characteristic function follows from Riccati ODEs with
time-dependent coefficients, solved numerically or
piecewise-analytically on a grid of sentiment-update times (cf.\
piecewise-constant-parameter Heston, Mikhailov and N\"ogel, 2003); the
simulations of Section~\ref{sec:simulation}, however, use Monte Carlo
under $\Q$ throughout, so this semi-closed-form route is not exercised
in this paper's own numerical work.

Taken together, this section ties the spatial configuration of market
attention, through the attention field $T(x,t)$, directly to the
mean-reversion behaviour of stochastic volatility, all in closed form,
with constant-$\kappa$ Heston sitting inside as the nested special case.

\subsection{Risk-Neutral Measure and the Market Price of Sentiment Risk}
\label{sec:risk-neutral}

The simulations of Section~\ref{sec:simulation} price options under a
risk-neutral measure $\Q$; this subsection constructs $\Q$ explicitly
and records which parts of the model are altered by the measure change
and which are invariant. The construction has one structural consequence
worth stating up front: because the sentiment field $T(\cdot,t)$ is an
\emph{observable, data-driven object} rather than a stochastic factor
with its own driving Brownian motion, the field itself, and hence the
coefficient functionals $(\lambda_t, \theta_t, \eta_t)$ of
\eqref{eq:kappa-sde}, are identical under $\Prob$ and $\Q$. In the
baseline specification $\eta_1 = 0$ the entire path $t \mapsto \kappa_t$
is deterministic and therefore \emph{measure-invariant}: no market price
of risk attaches to the second factor at all. This collapses the
risk-premium ambiguity that afflicts latent two-factor models, in which
the drift of the unobserved factor under $\Q$ is a free object that must
be jointly identified with the factor itself from option prices.

\subsubsection{Market-price-of-risk specification}

Let $W^{S,\perp}$ denote the component of $W^S$ orthogonal to $W^v$, so
that $W^S = \rho_{Sv} W^v + \sqrt{1 - \rho_{Sv}^2}\, W^{S,\perp}$. We
specify the market prices of risk in the essentially affine class of
Cheridito, Filipovi\'c, and Kimmel (2007), adapted to the time-varying
mean-reversion speed. This specification is what defines $\Q$: without
an assumption pinning down a candidate density process and confirming
it is a genuine change of measure, there is no equivalent martingale
measure to price under, and the constructions of
Section~\ref{sec:pricing} would have nothing to substitute into.

\begin{assumption}[Market prices of risk]\label{ass:mpr}
Fix constants $\lambda_v \in \R$ (price of variance risk) and
$\gamma_\kappa \in \R$ (price of sentiment risk), and assume the
mean-reversion floor of Theorem~\ref{thm:wellposed} is in force, so that
$\kappa_t \ge \kappa_* = \xi^2/(2\bar\theta)$ and $v_t > 0$ almost
surely. Define the market-price-of-risk processes
\begin{equation}\label{eq:mpr}
\gamma_t^v = \frac{\lambda_v}{\xi}\sqrt{v_t}, \qquad
\gamma_t^\kappa = \gamma_\kappa\, \mathbf{1}_{\{\eta_1 > 0\}}, \qquad
\gamma_t^{S,\perp} = \frac{1}{\sqrt{1 - \rho_{Sv}^2}}
\Bigl(\frac{\mu - r}{\sqrt{v_t}} - \rho_{Sv}\, \gamma_t^v\Bigr)
\end{equation}
(the last written for the GBM price specification $\mu(S,t) = \mu S$,
$\sigma(S,t) = S$; for a general admissible price diffusion replace
$(\mu - r)/\sqrt{v_t}$ by the corresponding Sharpe process), and the
candidate density process
\[
Z_t = \mathcal{E}\Bigl(-\!\int_0^\cdot \gamma_s^v\, dW_s^v
-\!\int_0^\cdot \gamma_s^{S,\perp}\, dW_s^{S,\perp}
-\!\int_0^\cdot \gamma_s^\kappa\, dW_s^\kappa\Bigr)_t.
\]
If additionally $\lambda_v \ge \kappa_* - \inf_{t \le T}\kappa_t$ (in
particular any $\lambda_v \ge 0$, and any $\lambda_v > -\kappa_*$ when
the floor binds), then $Z$ is a true martingale on $[0,T]$ and $d\Q =
Z_T\, d\Prob$ defines an equivalent measure.
\end{assumption}

The martingale property of $Z$ follows from the criterion of
Wong and Heyde (2006) for CIR-type volatility (equivalently, from
non-explosion of the auxiliary $v$-process under the candidate measure),
which is exactly the parameter restriction on $\lambda_v$ stated above;
the sentiment component contributes only a bounded,
deterministic-coefficient Gaussian exponent and never obstructs the
martingale property.

\subsubsection{Dynamics under $\Q$}

\begin{proposition}[Risk-neutral dynamics]\label{prop:Q-dynamics}
Under Assumption~\ref{ass:mpr}, the processes $W^{v,\Q} = W^v +
\int_0^\cdot \gamma_s^v\, ds$, $W^{S,\perp,\Q} = W^{S,\perp} +
\int_0^\cdot \gamma_s^{S,\perp}\, ds$, $W^{\kappa,\Q} = W^\kappa +
\int_0^\cdot \gamma_s^\kappa\, ds$ are $\Q$-Brownian motions (up to the
fixed correlation $\rho_{Sv}$ between the reconstituted $W^{S,\Q}$ and
$W^{v,\Q}$), and the SDSV system takes the same functional form under
$\Q$ with transformed coefficients:
\begin{align}
dS_t &= r\, S_t\, dt + \sqrt{v_t}\, S_t\, dW_t^{S,\Q},
\label{eq:Q-price}\\
dv_t &= \kappa_t^{\Q}\bigl(\bar\theta_t^{\Q} - v_t\bigr) dt
+ \xi\sqrt{v_t}\, dW_t^{v,\Q}, \qquad
\kappa_t^{\Q} := \kappa_t + \lambda_v, \quad
\bar\theta_t^{\Q} := \frac{\kappa_t\,\bar\theta}{\kappa_t + \lambda_v},
\label{eq:Q-var}\\
d\kappa_t &= \lambda_t\bigl(\theta_t^{\Q} - \kappa_t\bigr) dt
+ \eta_t\, dW_t^{\kappa,\Q}, \qquad
\theta_t^{\Q} := \theta_t - \frac{\eta_t\, \gamma_\kappa}{\lambda_t},
\label{eq:Q-kappa}
\end{align}
with $(\lambda_t, \theta_t, \eta_t)$ the \emph{same} observable
functionals \eqref{eq:lam-func}-\eqref{eq:eta-func} of the sentiment
field as under $\Prob$. In the baseline $\eta_1 = 0$, equation
\eqref{eq:Q-kappa} coincides with its $\Prob$-counterpart and $\kappa_t$
is the identical deterministic path under both measures, so the only
free risk-premium parameter is $\lambda_v$.
\end{proposition}

\begin{proof}
Girsanov's theorem applied to $Z$ gives the stated Brownian motions. For
the variance equation, the $\Prob$-drift plus the Girsanov correction is
$\kappa_t(\bar\theta - v_t) - \xi\sqrt{v_t}\,\gamma_t^v =
\kappa_t\bar\theta - (\kappa_t + \lambda_v)v_t$, which factors as
stated. For the price, $\gamma^{S,\perp}$ is chosen exactly so that the
total drift correction removes $\mu - r$. For $\kappa$, the correction
is $-\eta_t\gamma_\kappa$, absorbed into the level $\theta_t^{\Q}$
whenever $\lambda_t \neq 0$ (when $\lambda_t = 0$ on an interval, the
drift is the constant $-\eta_t\gamma_\kappa$ and the
variation-of-constants solution \eqref{eq:kappa-voc} applies verbatim
with that constant). The coefficients $(\lambda_t, \theta_t, \eta_t)$
are functionals of $T(\cdot,t)$, which is a deterministic transformation
of observed data, hence unchanged.
\end{proof}

\begin{proposition}[Measure-change invariants]
\label{prop:invariants}
Under Assumption~\ref{ass:mpr}, for the transformed dynamics of
Proposition~\ref{prop:Q-dynamics}, for every $t \ge 0$ and every admissible
$\lambda_v$:
\begin{enumerate}
\item[(i)] the speed--level product is invariant:
$\kappa_t^{\Q}\,\bar\theta_t^{\Q} = \kappa_t\,\bar\theta$;
\item[(ii)] the Feller threshold is invariant: the floor $\kappa_t \ge
\kappa_* = \xi^2/(2\bar\theta)$ that Theorem~\ref{thm:wellposed} uses to
guarantee strict positivity of $v$ is, by (i), the same condition under
$\Q$ as under $\Prob$ --- $2\kappa_t^{\Q}\bar\theta_t^{\Q} \ge \xi^2$
holds if and only if $2\kappa_t\bar\theta \ge \xi^2$, i.e.\ if and only
if $\kappa_t \ge \kappa_*$ --- so no $\Q$-specific recalibration of the
floor is needed.
\end{enumerate}
Consequently the floor of Theorem~\ref{thm:wellposed} is a
\emph{measure-free} structural condition: enforcing it once guarantees
the CIR boundary behaviour under both the physical and the pricing
measure, and no choice of variance risk premium can manufacture or
destroy the mean-reversion pull $\kappa_t\bar\theta$. This is also why,
in any affine representation of the model under $\Q$ (its
characteristic function, say), the coefficient $\kappa_t\bar\theta$
carries no dependence on $\lambda_v$ at all: the variance risk premium
enters only through $\kappa_t^{\Q} = \kappa_t + \lambda_v$, never
through the product $\kappa_t\bar\theta$ itself.
\end{proposition}

\begin{proof}
From \eqref{eq:Q-var}, $\kappa_t^{\Q}\bar\theta_t^{\Q} = (\kappa_t +
\lambda_v) \cdot \kappa_t\bar\theta/(\kappa_t + \lambda_v) =
\kappa_t\bar\theta$, which is (i); (ii) is immediate from (i). The
admissibility restriction of Assumption~\ref{ass:mpr} guarantees
$\kappa_t + \lambda_v > 0$, so the algebra is non-degenerate.
\end{proof}

\begin{remark}[Economic content of $\bar\theta_t^{\Q}$]\label{rem:theta-q}
Under $\Q$ the long-run variance level $\bar\theta_t^{\Q} =
\kappa_t\bar\theta/(\kappa_t + \lambda_v)$ is no longer constant: it
loads on the sentiment state through $\kappa_t$. With $\lambda_v < 0$,
the empirically typical sign of the variance risk premium,
$\bar\theta_t^{\Q} > \bar\theta$, and the wedge $\bar\theta_t^{\Q} -
\bar\theta$ \emph{shrinks} when sentiment concentrates ($\kappa_t$
large). The model therefore predicts a state-dependent variance risk
premium that compresses during attention-concentrated regimes, a
testable implication that a constant-$\kappa$ Heston model cannot
generate, and one that is priced by observable data rather than by a
filtered latent state.
\end{remark}

\subsubsection{Pricing under $\Q$}

\begin{proposition}[Martingale property and pricing]\label{prop:martingale}
Under Assumption~\ref{ass:mpr} with $\rho_{Sv} \le 0$ and
$\kappa^{\Q}$ (Proposition~\ref{prop:Q-dynamics}) locally bounded (guaranteed by the local boundedness of
the functionals, Section~\ref{sec:kappa}), the discounted price
$e^{-rt}S_t$ is a true $\Q$-martingale on $[0,T]$ for the GBM and
OU-log-price specifications of Section~\ref{sec:price-models}, and
European claims are priced by $C(0) = \E^{\Q}[e^{-rT}(S_T - K)^+]$.
\end{proposition}

\begin{proof}
Let $\mathcal G := \sigma(W_s^\kappa : s \le T) \vee \sigma(T(\cdot, s)
: s \le T)$, augmented by the $\Q$-null sets. Under $\Prob$, $W^\kappa$
and the field are independent of $(W^v, W^{S,\perp})$ by the model's
construction (the sentiment/attention factor is an exogenous process,
not driven by the asset's own innovations), so $\mathcal G$ is
$\Prob$-independent of $\sigma(W^v, W^{S,\perp}) = \sigma(W^S, W^v)$.
The density of Assumption~\ref{ass:mpr} factors as $Z_T =
Z_T^{(1)}(W^v, W^{S,\perp}) \cdot Z_T^{(2)}(W^\kappa)$, since
$\gamma^v, \gamma^{S,\perp}$ depend only on $(W^v, W^{S,\perp})$ and
$\gamma^\kappa$ only on $W^\kappa$; both factors are true martingales
with $\E^\Prob[Z_T^{(1)}] = \E^\Prob[Z_T^{(2)}] = 1$. For $\mathcal
G$-measurable bounded $f$ and $\sigma(W^v,W^{S,\perp})$-measurable
bounded $g$, independence under $\Prob$ of $(f Z_T^{(2)})$ from $(g
Z_T^{(1)})$ gives $\E^\Q[fg] = \E^\Prob[fZ_T^{(2)}]\,
\E^\Prob[gZ_T^{(1)}]$, while $\E^\Q[f] = \E^\Prob[fZ_T^{(2)}]\cdot
\E^\Prob[Z_T^{(1)}] = \E^\Prob[fZ_T^{(2)}]$ and likewise $\E^\Q[g] =
\E^\Prob[gZ_T^{(1)}]$, so $\E^\Q[fg] = \E^\Q[f]\E^\Q[g]$: $\mathcal G$
remains independent of $\sigma(W^v, W^{S,\perp})$ under $\Q$, even
though $Z_T^{(2)} \neq 1$ in general means $\mathcal{G}$'s own marginal
law can change from $\Prob$ to $\Q$ when $\gamma_\kappa \neq 0$ ---
independence, which is what is needed below, does not require
$\gamma_\kappa = 0$. Since $W^{S,\Q}, W^{v,\Q}$ are $(W^S,W^v)$ shifted
by an adapted drift, $\sigma(W^{S,\Q}, W^{v,\Q}) = \sigma(W^S,W^v)$, so
$\mathcal G$ is independent of $\sigma(W^{S,\Q}, W^{v,\Q})$ under $\Q$
too.
Work under a regular conditional probability $\Q(\,\cdot \mid \mathcal
G)$, which exists on the Polish path space. Conditionally on $\mathcal
G$, the pair $(W^{S,\Q}, W^{v,\Q})$ remains a correlated Brownian pair
with correlation $\rho_{Sv}$, and $(S, v)$ solves a time-inhomogeneous
Heston system whose mean-reversion input $s \mapsto
\kappa_s^{\Q}(\omega)$ is a fixed, deterministic, locally bounded
function bounded below by $\kappa_* > 0$ (Assumption~\ref{ass:mpr}),
while $(\bar\theta^{\Q}, \xi, \rho_{Sv})$ satisfy the hypotheses of
Andersen and Piterbarg (2007, Prop.~3.1) with $\rho_{Sv} \le 0$. That
criterion requires non-positive correlation and positive variance
parameters but no upper bound on the reversion speed, since a larger
$\kappa^{\Q}$ only strengthens the downward pull on $v$ and cannot
induce a moment explosion; it yields, for $\Q$-a.e.\ $\omega$,
\[
\sup_{t \le T} \E^{\Q}\bigl[S_t \mid \mathcal G\bigr](\omega) < \infty
\qquad\text{and}\qquad
\E^{\Q}\bigl[e^{-rt}S_t \mid \mathcal F_s \vee \mathcal G\bigr]
= e^{-rs}S_s, \quad 0 \le s \le t \le T.
\]
Taking $\E^{\Q}[\,\cdot \mid \mathcal F_s]$ of the conditional
martingale identity and using $\mathcal F_s \subset \mathcal F_s \vee
\mathcal G$ together with the conditional Fubini theorem gives
$\E^{\Q}[e^{-rt}S_t \mid \mathcal F_s] = e^{-rs}S_s$; integrability of
$S_t$ follows from the tower property applied to the conditional moment
bound. Hence $e^{-rt}S_t$ is a true $\Q$-martingale on $[0,T]$ and $C(0)
= \E^{\Q}[e^{-rT}(S_T - K)^+]$. The OU-log-price case is identical with
bounded coefficients.
\end{proof}

\subsubsection{Identification of the risk premia}

The sentiment functionals $(\lambda_t,
\theta_t, \eta_t)$ and, in the baseline, the entire path $\kappa_t$ are
identified from \emph{physical} data alone, the information stream,
with no derivatives data required. The variance risk premium $\lambda_v$
is identified from the wedge between option-implied and realised
variance dynamics, exactly as in Heston, but with the sharper structure
that $a(t) = \kappa_t\bar\theta$ is pinned by observables, so
$\lambda_v$ is the \emph{only} free parameter in $b(t)$; mispricing of
the term structure of implied variance identifies it directly. The
sentiment risk premium $\gamma_\kappa$ exists only in the extension
$\eta_1 > 0$ and is identified only from derivatives data, through the
gap between the $\Prob$-drift of $\kappa_t$ (estimated from the observed
$\kappa$-path itself, which is measurable in real time) and the
$\Q$-drift implied by the option surface. A natural null hypothesis for
empirical work is $\gamma_\kappa = 0$: sentiment innovations orthogonal
to the variance and equity drivers are idiosyncratic and command no
premium, in which case the $\kappa$-dynamics coincide under both
measures even in the stochastic extension.

\section{Simulation Illustration}
\label{sec:simulation}

This section illustrates, within the model, the separation between the
sentiment-driven dynamics and the exactly nested constant-$\kappa$
benchmark of Proposition~\ref{prop:nesting}. All data below is
model-generated: the exhibits measure the conditional information
carried by the observable second factor and the stability of the induced
effects across parameter regimes, not out-of-sample forecasting skill,
which is the subject of the empirical companion work. Parameters are
listed in Table~\ref{tab:params}; robustness across parameter regimes is
summarised in Section~\ref{sec:exhibit3} and reported in full in
Appendix~\ref{app:robustness}. All figures and quoted statistics are
produced by the companion notebook implementing the closed-form solution
\eqref{eq:CS} and the exact forecast of Corollary~\ref{cor:forecast}.

\subsection{Validation}
\label{sec:exhibit1}

Figure~\ref{fig:exhibit1} presents two diagnostic panels. Panel~A tracks
the rolling standard deviation of the swarm's centre of mass over 80
iterations: it drops below $0.01$ within about 17 iterations and stays
there, the finite-$N$ stability one expects from a dynamics in the
propagation-of-chaos class of Theorem~\ref{thm:poc-esso}. Panel~B shows
the sentiment density estimator \eqref{eq:estimator} at four time
snapshots and exhibits the lifecycle the governing PDE \eqref{eq:pde}
predicts: early concentration at the particle cluster, diffusive
spreading with the bandwidth widening at the Mehler rate $h(t)$ toward
the saturation scale $\sqrt{2/\omega}$, and dissipation as the quadratic
radial sink and uniform decay dominate a direct visualisation of
Theorem~\ref{thm:field} and the exact mass law of Remark~\ref{rem:mass}.

The remaining validation results are quickly stated. The model
reproduces the stylised facts that motivated it: the autocorrelation of
absolute returns is significant at all 20 lags tested (volatility
clustering); the realised leverage correlation is $-0.699$ against the
target $\rho_{Sv} = -0.70$; and the ensemble mean of terminal variance
deviates from $\bar\theta = 0.04$ by $0.2\%$, within the natural
ensemble variance of a year-long CIR path. On a representative single
path, $\kappa_t$ varies smoothly on the sentiment timescale set by
$\lambda_t \sim 1$ while $v_t$ stays jagged at the daily
volatility-of-volatility scale $\xi = 0.30$: the two factors
live on different timescales.

\begin{figure}[t]
\centering
\sdsvfigure{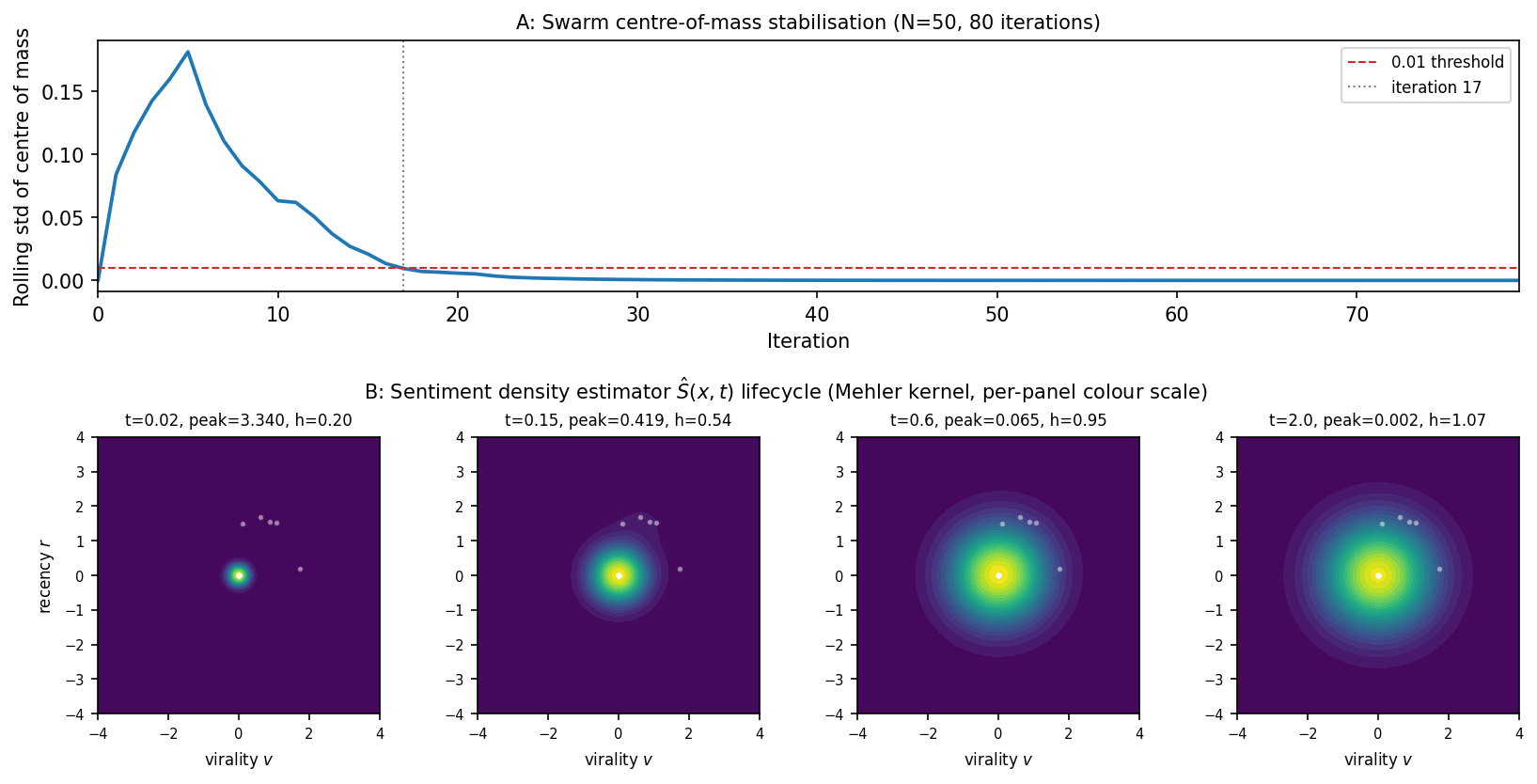}
\caption{\textbf{Validation.} Panel~A: rolling standard deviation of the
swarm centre of mass over 80 iterations; convergence below the $0.01$
threshold within ${\sim}15$ iterations, consistent with the finite-$N$
stability of the regularised dynamics (Theorem~\ref{thm:poc-esso}).
Panel~B: sentiment density estimator $\hat S(x,t)$ from
\eqref{eq:estimator} on the virality--recency plane at four snapshots,
computed with the Mehler kernel \eqref{eq:kernel}: the effective
bandwidth widens as $h(t) = \sqrt{2\tanh(\omega
t_{\mathrm{eff}})/\omega}$ toward the saturation scale $\sqrt{2/\omega}
= 1.07$ while the peak intensity decays monotonically, encoding
diffusive spreading against exponential dissipation $A(t) =
e^{-(1-c)\alpha t}$ (Theorem~\ref{thm:field}, Remark~\ref{rem:mass}).
Parameters as in Table~\ref{tab:params}.}
\label{fig:exhibit1}
\end{figure}

\begin{table}[p]
\centering
\small
\begin{tabular}{llll}
\toprule
Symbol & Description & Value & Justification \\
\midrule
\multicolumn{4}{l}{\emph{ESSO swarm}}\\
$N$ & Number of sparrows & 50 & Moderate swarm size \\
$T_{\max}$ & ESSO iterations & 80 & Convergence within ${\sim}15$ iterations \\
$c_1, c_2$ & Decay hyperparameters & 3.0, 2.0 & Standard exploration-exploitation balance \\
\midrule
\multicolumn{4}{l}{\emph{Sentiment field}}\\
$c$ & Net herding intensity & 0.25 & Herders moderately outweigh contrarians \\
$\alpha$ & Uniform sink rate & 1.0 & Calibrated to social-media engagement decay \\
$\beta$ & Quadratic radial sink coeff. & 1.0 & Matches Gaussian tail decay \\
$\omega = 2\sqrt{(1-c)\beta}$ & Mehler rate parameter & 1.73 & Derived \\
$\sqrt{2/\omega}$ & Mehler saturation bandwidth & 1.07 & Derived (Section~\ref{sec:mehler}) \\
$h_0$ & Initial smoother bandwidth & 0.04 & Granularity of attention measurement \\
\midrule
\multicolumn{4}{l}{\emph{$\kappa_t$ dynamics}}\\
$\lambda_1$ & Reversion scaling constant & 1.0 & Targets $\lambda_t \sim 1$ to $2$ \\
$\theta_1$ & Level scaling constant & 4.0 & Targets $\theta_t \sim 1$ to $2$ \\
$\eta_1$ & Diffusivity scaling constant & 2.0 & Stochastic extension (baseline: $\eta_1 = 0$) \\
$\kappa_0$ & Initial mean-reversion speed & 2.0 & \\
$\kappa_*$ & Feller floor $\xi^2/(2\bar\theta)$ & 1.125 & Theorem~\ref{thm:wellposed}: below it, $v_t \ge 0$ only \\
\midrule
\multicolumn{4}{l}{\emph{Variance process (CIR)}}\\
$\bar\theta$ & Long-term variance & 0.04 & 20\% annualised vol (S\&P 500) \\
$\xi$ & Volatility-of-volatility & 0.30 & Mid-range Heston estimate for equities \\
\midrule
\multicolumn{4}{l}{\emph{Asset price}}\\
$\rho_{Sv}$ & Leverage correlation & $-0.70$ & Empirical equity leverage effect \\
$r$ & Risk-free rate & 0.02 & \\
$S_0$ & Initial asset price & 100 & \\
\bottomrule
\end{tabular}
\caption{Simulation parameters. The variance parameters $\bar\theta,
\xi, \rho_{Sv}$ match typical values for large-cap US equity indices.
Note that with $\xi = 0.30$ and $\bar\theta = 0.04$, the Feller level is
$\kappa_* = 1.125$; the simulations use full-truncation Euler, and
simulated $\kappa$-paths spend roughly a third of their steps below
$\kappa_*$, illustrating why Theorem~\ref{thm:wellposed} claims $v_t \ge
0$ rather than strict positivity absent a floor. ESSO and
sentiment-field parameters are set for illustration; empirical
calibration, well-posed by the measurement results of
Section~\ref{sec:measurement}, is deferred to future work.}
\label{tab:params}
\end{table}

\subsection{Comparison with the Nested Heston Benchmark}
\label{sec:exhibit2}

By Proposition~\ref{prop:nesting}, constant-$\kappa$ Heston is the exact
nested special case of SDSV, so every comparison below quantifies the
incremental content of the sentiment-driven $\kappa_t$ over the best
constant-$\kappa$ benchmark. In sample, the nesting guarantees a
calibrated model never fits worse; whether the attention field captures
sentiment-driven variation in $\kappa$ on real data without overfitting
is left to empirical work.

\subsubsection{Behaviour during a sentiment shock}
\label{sec:shock}

Figure~\ref{fig:exhibit2} isolates the effect of the second factor
during a regime shift. Both models start from identical elevated
variance $v_0 = 0.09$ (30\% instantaneous volatility) with identical
Brownian increments across 400 paths; the only difference is the
mean-reversion speed. In SDSV the shock is imposed exogenously as the
post-shock initial condition $\kappa_0 = 6.0$, standing in for the
functionals' response to an information burst, after which the OU
dynamics \eqref{eq:kappa-sde} govern the decay; in Heston, $\kappa$
stays fixed at the equity-calibrated value $2.0$ (Panel~A). The SDSV
model reverts variance toward $\bar\theta$ more aggressively in the
post-shock window: the expected-volatility gap reaches 2.1 percentage
points by day 21 and 3.37 points by day 63 (Panel~B), and the ATM
call-price gap reaches 7.8\% at day 63 (Panel~C) --- a mispricing
invisible to a practitioner using a constant-$\kappa$ model calibrated
before the shock. All option values in Panel~C use only time-$t$
information via the exact forecast of Corollary~\ref{cor:forecast}; no
look-ahead into the realised $\kappa$ path is involved.

\subsubsection{Within-model value of the conditioning information}
\label{sec:panel-d}

At every evaluation date we compare two forecasts against realised
future variance on the same SDSV-generated path: the exact
Corollary-\ref{cor:forecast} formula from the current $(v_t, \kappa_t)$,
and the nested Heston forecast from the same observed $v_t$ with
$\kappa$ frozen at its benchmark value. Both share $v_t$, so the
incremental skill $1 -
\mathrm{MSE}_{\mathrm{SDSV}}/\mathrm{MSE}_{\mathrm{Heston}}$ isolates
exactly the information content of the reversion speed. The result is a
single story: the increment is negligible when sentiment is quiet, and
it grows with horizon and with how much $\kappa_t$ moves
from 4.8\% at 5 days to 43.5\% at 63 days in the post-shock window, and
scaling with $\mathrm{CoV}(\kappa_t)$ in the stationary regime
(Appendix~\ref{app:robustness}, Table~\ref{tab:cov}). The second factor
adds little when sentiment is quiet and adds materially when it shifts,
which is precisely the regime in which Panels~A--C show multi-point
volatility gaps and double-digit option mispricings.

\subsubsection{Drawdown prediction}
\label{sec:panel-e}

We also evaluated the drawdown gauge of Definition~\ref{def:gauge} in
the shock scenario, where the two models disagree most: the SDSV gauge,
fed by the contracting $\mathrm{IV}_t(\Delta)$ of
Corollary~\ref{cor:IV}, assigns uniformly lower exceedance probabilities
than the constant-$\kappa$ gauge. Against realised frequencies over the
400 paths, however, both gauges overstate exceedances at daily
monitoring by a similar margin (maximum calibration gap $0.104$ for SDSV
versus $0.112$ for the benchmark), for the leverage and
discrete-monitoring reasons set out in Section~\ref{sec:drawdown}. We
therefore treat the gauge as a closed-form, look-ahead-free estimate of
downside risk that tracks the shape of the realised curve, and claim no
calibration advantage over the nested benchmark.

\subsubsection{Regime pricing}
\label{sec:panel-f}

Pricing 63-day ATM calls from an elevated starting variance (28\% vol)
under three sentiment regimes, $\kappa \in \{0.3, 1.6, 4.0\}$, produces
an ATM price spread of 8.6\% across regimes. A constant-$\kappa$ model
assigns a single price to all three by construction; the spread is a
structural, not incremental, advantage of conditioning on the sentiment
state.

\begin{figure}[t]
\centering
\sdsvfigure{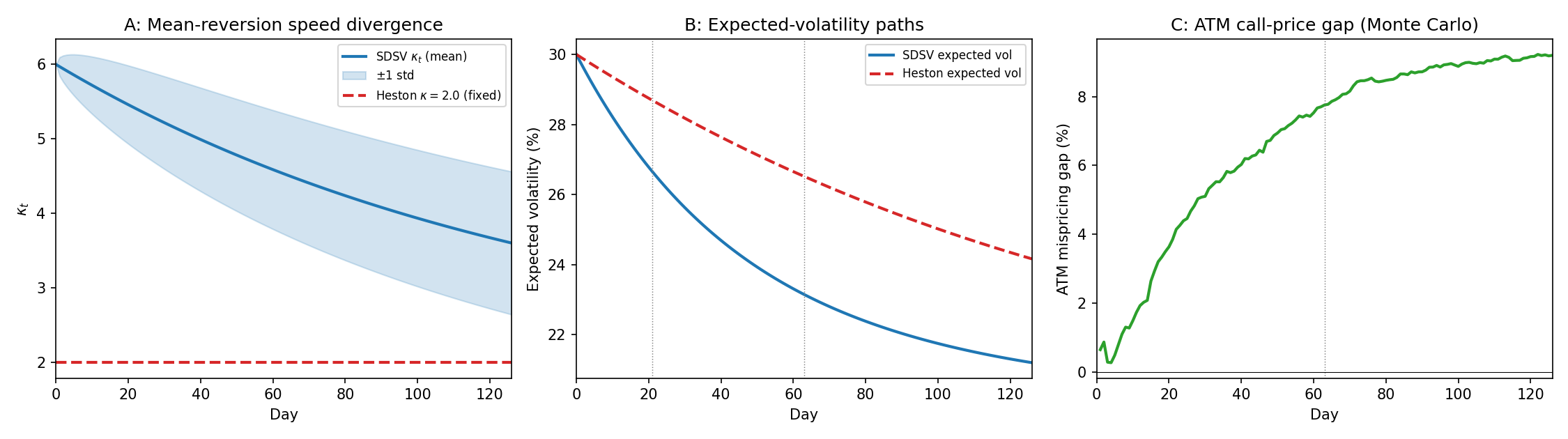}
\caption{\textbf{SDSV vs the nested Heston benchmark during a sentiment
shock.} Both models start from identical elevated variance $v_0 = 0.09$
(30\% vol) with identical Brownian increments across 400 paths over a
126-day window; at day 0 the SDSV $\kappa_t$ jumps to $6.0$ and decays
through the OU dynamics \eqref{eq:kappa-sde}, while the Heston baseline
$\kappa = 2.0$ stays fixed. Panel~A: $\kappa_t$ divergence (SDSV mean
$\pm 1$ std vs the constant). Panel~B: expected-volatility paths; the
gap reaches 2.1pp by day 21 and 3.37pp by day 63. Panel~C: approximate
ATM call prices evaluated with the exact time-$t$ forecast of
Corollary~\ref{cor:forecast}; the mispricing gap reaches 7.8\% at day
63. Parameters as in Table~\ref{tab:params}.}
\label{fig:exhibit2}
\end{figure}

\subsection{Robustness}
\label{sec:exhibit3}

Appendix~\ref{app:robustness} varies each parameter one at a time from
the baseline and reports the post-shock statistics. Three patterns
emerge, each of which is what a genuine mechanism rather than a
numerical accident would produce: the effects are essentially invariant
to vol-of-vol, leverage, and $\kappa$-reversion speed, the parameters
that should not drive an effect sourced in the reversion channel; they
scale monotonically with the shock size $\kappa_{\mathrm{shock}}$; and
they vanish as the initial variance approaches its long-run level, the
boundary at which SDSV reduces to Heston. Two limits should be stated
plainly. The study is one-at-a-time (local) sensitivity around a single
baseline, so it does not chart parameter interactions except along the
single slice of Figure~\ref{fig:exhibit3}; and, like all of this
section, it is within-model, since the data is model-generated, so these
are stability and identifiability results rather than out-of-sample
evidence.

\paragraph{Empirical deployment.} Deployment on market data requires
estimating the measurement parameters $(\tau, \Lambda, \lambda_1,
\theta_1, r_0)$ from an information stream and the risk premium
$\lambda_v$ from the option surface; by
Propositions~\ref{prop:scale-inv}--\ref{prop:w1-stability} and
Corollary~\ref{cor:missing} this calibration problem is well posed, and
by Proposition~\ref{prop:nesting} any calibrated instantiation is weakly
dominant in sample relative to a constant-$\kappa$ benchmark. The
empirical program, including forecasting architectures built on the
swarm state, is developed in companion work.

\section{Extensions to Multi-Asset Correlations}
\label{sec:multiasset}

A multi-asset extension introduces controlled sentiment spillovers
across related securities. For target asset $m = 1, \dots, M$, the
information stream $\mathcal{I}^m$ includes primary items (label $k =
m$) and selected items from related assets (label $l \neq m$, identified
via named-entity recognition or sector classification). Producer
selection runs over the full stream, with cross-asset producer
attraction scaled by $\beta_{ml} \in [-1, 1]$ (positive for narrative
alignment, negative for substitutes).

In the mean-field limit, this yields the coupled system
\begin{equation}\label{eq:multi}
\partial_t T^m = \Delta T^m - \frac{c^m}{T^m}\, |\nabla T^m|^2
- \sum_{l \neq m} \beta_{ml}\, \frac{\nabla T^m \cdot \nabla T^l}{T^m}
- W^m(x)\, T^m, \qquad m = 1, \dots, M,
\end{equation}
on non-negative initial data. The cross terms encode directed spillover.
The matrix $B = (\beta_{ml})$ admits calibration from historical
realised volatility correlations via $\beta_{ml} = \tanh(\gamma \cdot
\mathrm{corr}(v^m, v^l))$ for tunable $\gamma > 0$. Each asset's
individual radial functionals and resulting $\kappa_t^m$ admit the same
closed-form structure as the single-asset model when the cross-coupling
is treated perturbatively; the single-asset case is recovered when $B$
is the identity.

The cross-coupled system is no longer fully linearisable in closed form
because the gradient cross-term $(\nabla T^m \cdot \nabla T^l)/T^m$ does
not separate under the power-law substitution applied component-wise.
Existence and uniqueness must therefore be established by other means;
the following theorem provides this in the empirically relevant regime
of bounded cross-coupling.

\begin{proposition}[A priori bound, any coupling size]\label{prop:multi-apriori}
Let $M \ge 1$ and consider the coupled system \eqref{eq:multi} with
initial data $T_0^m \in L^1 \cap L^\infty(\R^2)$ strictly positive
almost everywhere, parameters $c^m \in (0,1)$, cross-coupling
coefficients $\{\beta_{ml}\}_{m\neq l} \subset \R$ of \emph{any} size,
and quadratic radial sinks $W^m(x) = \alpha^m + \beta^m|x|^2$ with
$\alpha^m, \beta^m > 0$. Every classical solution $(T^1,\dots,T^M)$ of
\eqref{eq:multi} on $\R^2\times(0,\infty)$, should one exist, satisfies:
(i) $T^m(x,t) \le \sup_x T_0^m(x)$ for all $(x,t)$ and all $m$; (ii)
$T^m(x,t) > 0$ for every $(x,t)$ (strict positivity, via the strong
maximum principle for the non-degenerate diffusion, applied
componentwise); and (iii) for every fixed $R>0$ and $t_0>0$, $T^m(x,t)
\ge \delta_R(t_0) > 0$ for all $x \in B_R$ and $t \ge t_0$, for some
$\delta_R(t_0)$ depending on $R, t_0$ but not on the coupling size.
No uniform-in-$x$ lower bound over all of $\R^2$ holds, or should be
expected: like the single-asset field of Theorem~\ref{thm:field}, each
$T^m$ decays at infinity (Remark~\ref{rem:multi-decay}), so (iii) is the
correct, and the only available, form of a lower bound. All three
statements hold with no coupling-dependent buffer, for every coupling
size.
\end{proposition}

\begin{remark}\label{rem:multi-decay}
The cross-coupling in \eqref{eq:multi} is a gradient term, which does
not counteract the confining sink $-W^m T^m$ at long range; there is no
mechanism by which coupling could produce a field that stays bounded
away from zero as $|x|\to\infty$. A single global constant $\delta>0$
with $T^m(x,t)\ge\delta$ for all $x$ would therefore be false; the
correct statement is the compact-set bound (iii), which is what the
proof (via the maximum principle at critical points)
establishes and is also the form used in Appendix~\ref{app:multiasset}'s
own construction of the unperturbed baseline $U^m$.
\end{remark}

\begin{theorem}[Multi-asset existence and uniqueness, small coupling]\label{thm:multi}
Under the hypotheses of Proposition~\ref{prop:multi-apriori}, there
exists $\varepsilon^* > 0$, depending on $\{c^m, \alpha^m, \beta^m,
\|T_0^m\|_{L^1 \cap L^\infty}\}_{m=1}^M$, such that whenever $\max_{m
\neq l} |\beta_{ml}| < \varepsilon^*$, \eqref{eq:multi} has a unique
classical solution $(T^1,\dots,T^M)$ with each $T^m \in
C^{2,1}(\R^2\times(0,\infty))$, satisfying the bound of
Proposition~\ref{prop:multi-apriori}.
\end{theorem}
\begin{remark}
Proposition~\ref{prop:multi-apriori}'s a priori bound holds regardless
of coupling size; existence of a classical solution attaining it beyond
the small-coupling regime is a natural direction for future work.
\end{remark}

\begin{proof}[Proof sketch]
Apply the power-law substitution $u^m := (T^m)^{1-c^m}/(1-c^m)$ to each
component. As in Proposition~\ref{prop:powerlaw}, the diagonal nonlinear
terms linearise; the system becomes
\[
\partial_t u^m = \Delta u^m - (1-c^m) W^m(x)\, u^m
- \sum_{l \neq m} \beta_{ml}\, G^{ml}(u^m, u^l, \nabla u^m, \nabla u^l),
\]
where each $G^{ml}$ is smooth and Lipschitz on bounded subsets of
$C^{2,1}$ away from $\{u^m = 0\}$. Proposition~\ref{prop:multi-apriori}'s
two-sided $L^\infty$ bound is fully established: freezing only the
\emph{other} assets $(u^l)_{l\neq m}$, rather than the full source
$G^{ml}$, turns the cross-coupling into a linear drift in the map's own
$\nabla u^m$, which vanishes at $u^m$'s own critical points exactly as
the original nonlinear term did, so the componentwise maximum principle
gives the exact bound, with no coupling-dependent buffer, throughout
the iteration --- at every intermediate step, before anything has
converged. This describes any solution that happens to exist; whether
one does is answered separately, by Theorem~\ref{thm:multi}. Its
existence and uniqueness, for small coupling, follows from Banach fixed
point in $C([t_0, T^*]; L^\infty \cap L^2)^M$ around the unperturbed
(decoupled) Mehler-kernel solution, with the cross-coupling treated as a
Lipschitz perturbation of order $\max|\beta_{ml}|$; this argument is
self-contained and does not rely on Proposition~\ref{prop:multi-apriori}.
The full argument is
given in Appendix~\ref{app:multiasset}.
\end{proof}

The smallness condition in Theorem~\ref{thm:multi} is interpretable
economically: it requires cross-asset sentiment spillovers to be modest
relative to the within-asset dissipation rates. Empirically realistic
spillover values $|\beta_{ml}| = |\tanh(\gamma \cdot \mathrm{corr}(v^m,
v^l))| < 1$ lie within the contraction regime for reasonable $\gamma$.
Proposition~\ref{prop:multi-apriori}'s two-sided $L^\infty$ bound holds
unconditionally in the coupling size, for any solution that exists;
whether a solution exists for coupling strong enough to leave that
regime is the open question flagged in Appendix~\ref{app:multiasset}.

\section{Conditions for Alternative Sentiment Aggregation Methods}
\label{sec:conditions}

The framework is modular. We use ESSO because it captures the herding,
contrarian, and attention-decay behaviour of market participants, but
the pipeline from agent dynamics to the sentiment PDE \eqref{eq:pde},
its closed-form solution \eqref{eq:CS}, and the $\kappa_t$ dynamics
\eqref{eq:kappa-sde} rests on only a few structural properties of those
dynamics. The conditions below are offered to facilitate future work
with alternative data sources (order-flow clustering from limit order
books, topic-model news aggregation, graph-based diffusion over
co-mention networks) while preserving the closed-form volatility
dynamics.

\paragraph{Condition 0: Measurement map.} The aggregation method
consumes data through a map into $B_R \subset \R^2$ that is invariant to
common rescaling of attention counts and Lipschitz from item attributes
to positions. Under Condition~0, Propositions~\ref{prop:scale-inv}
and~\ref{prop:w1-stability} and Corollary~\ref{cor:missing} extend
verbatim, since their proofs use only compact support, kernel smoothing,
and mass normalisation.

\paragraph{Condition 1: Bounded Lipschitz mean-field drift.} There is a
map $b : \mathcal{P}_2(\R^2) \times \R^2 \to \R^2$, bounded and jointly
Lipschitz in $(x,\mu)$ with $\R^2$ in Euclidean distance and
$\mathcal{P}_2(\R^2)$ in $W_1$, such that each agent's non-diffusive
velocity is $b[\mu_t](X_t)$ with $\mu_t$ the population state. A
pairwise interaction kernel, $b[\mu](x) = \int K(x-y)\,\mu(dy)$ with $K$
bounded and Lipschitz, is the special case in which $b$ is linear in
$\mu$; the soft-extremum attractor \eqref{eq:soft} and the role-blended
drift of \eqref{eq:role-blend} below are nonlinear instances of the same
hypothesis, verified directly rather than via a kernel. Condition~1
ensures global Lipschitz mean-field drift, supporting the
propagation-of-chaos estimate (Theorem~\ref{thm:poc}) at the sharp
$N^{-1/2}\sqrt{\log N}$ rate.

\paragraph{Condition 2: Uniform second-moment bound.} $\sup_N \sup_j
\E[|X_j(t)|^2] \le M(t) < \infty$ for a locally bounded function $M$,
guaranteeing tightness in $\mathcal{P}_2(\R^2)$ and finiteness of the
radial functionals.

\paragraph{Condition 3: Non-degenerate isotropic noise.} Each agent
receives an additive stochastic perturbation with isotropic,
non-degenerate covariance bounded away from zero on compact time
intervals. This produces the Laplacian $\Delta T$ in the limiting PDE
and provides regularising effects ensuring $T(\cdot,t) \in L^1 \cap
L^\infty$ for all $t > 0$.

\paragraph{Condition 4: Macroscopic equation has the saturating-herding
form.} The macroscopic limit of the agent dynamics, however it is
established, must yield the nonlinear term $-(c/T)|\nabla T|^2$ for some
$c \in (0,1)$. Within the family of gradient nonlinearities
$-g(T)|\nabla T|^2$, the choice $g(T) = c/T$ is the only one for which
the power-law substitution \eqref{eq:subst} both removes the
nonlinearity and leaves a sink that is linear in $u$, the two properties
jointly required for the Mehler closed form \eqref{eq:CS}. Other
gradient nonlinearities admit Hopf--Cole-type linearisations but not, in
general, both properties simultaneously; e.g.\ bare $\nabla T$ drift
giving $-\nabla\cdot(T\nabla\log T)$ does not inherit the closed-form
structure. The economic interpretation is net herding-versus-contrarian
behaviour with saturation; the rigorous microscopic foundation may
follow various paths, and one rigorous path (diffusion with killing
combined with a power-law readout) is detailed in
Appendix~\ref{app:derivation}.

\paragraph{Condition 5: Linear-in-$T$ sink with quadratic radial
penalty.} The dynamics incorporate (a) a uniform forgetting rate
producing the term $-\alpha T$, and (b) a quadratic radial penalty
producing $-\beta|x|^2 T$. The quadratic radial form is required for the
linearised equation \eqref{eq:linear} to admit the explicit
Mehler-kernel solution \eqref{eq:kernel}. Alternative radial penalties
such as those linear in $|x|$ or of higher polynomial order would require different
special-function kernels (e.g.\ Airy-type) and may not admit such
compact closed form.

\paragraph{Implications.} Any aggregation method satisfying Conditions~0
to~5 inherits the full analytical formulations: the macroscopic PDE
\eqref{eq:pde}, the power-law substitution \eqref{eq:subst}, the
Mehler-kernel closed-form solution \eqref{eq:CS} with its exact mass law
(Remark~\ref{rem:mass}), the KDE sentiment density estimator
\eqref{eq:estimator}, the explicit variation-of-constants solution for
$\kappa_t$ via \eqref{eq:kappa-voc}, the exact conditional forecast of
Corollary~\ref{cor:forecast}, and the measurement stability results of
Section~\ref{sec:measurement}. The ESSO algorithm remains the primary
instantiation: chaotic initialisation, role-based heterogeneity, and
adaptive mutation schedules are specifically designed to capture the
dynamics of social-media-driven market attention.

\section{Conclusion}
\label{sec:conclusion}

This paper develops a two-factor stochastic volatility model in which the
mean-reversion speed of a Cox--Ingersoll--Ross variance process is computed
directly from an observable market attention field, rather than filtered as a
latent quantity, as in every prior two-factor extension of Heston. The
attention field itself is a power-law readout of a Feynman--Kac semigroup,
with an explicit Mehler-kernel representation that keeps the whole framework
analytically tractable: it nests the constant-$\kappa$ Heston specification
exactly, yields closed-form conditional forecasts of variance and integrated
variance, and admits an equivalent martingale measure under which the
speed--level and Feller conditions are invariant
(Proposition~\ref{prop:invariants}). Its microscopic foundation is an
interacting swarm whose limiting dynamics generate the attention field; the
regularised, heterogeneous version of these dynamics satisfies propagation of
chaos at the sharp two-dimensional rate.

Two structural properties make the resulting pipeline usable in practice.
The attention field is strictly positive, bounded, Gaussian-decaying, and
mass-monotone (Theorem~\ref{thm:field}), so the volatility functionals built
from it are automatically well-posed; and the measurement interface that
constructs it is platform-scale invariant and Wasserstein-Lipschitz stable
under perturbations, censoring, and missing data
(Propositions~\ref{prop:scale-inv}--\ref{prop:w1-stability},
Corollary~\ref{cor:missing}). The swarm dynamics also give a behavioural
reading of the field equation --- diffusion, transport, and dissipation
correspond to identifiable microscopic mechanisms --- and the mean-reversion
speed they generate admits explicit representations and conditional
forecasts (Theorem~\ref{thm:kappa}; Corollary~\ref{cor:forecast}).

Because the framework contains Heston exactly as a special case
(Proposition~\ref{prop:nesting}), it preserves a classical benchmark while
adding a new volatility-regime variable in its place, at no extra
calibration cost. This suggests applications to volatility-surface
construction, sentiment-conditioned risk management, and systemic-risk
monitoring; the construction is also modular enough that future work can
substitute other aggregation mechanisms for the swarm without disturbing the
analytical structure established here.

\appendix

\section{Organisation of the Proofs}
\label{app:proofs}

The proofs of the results stated in the body of the paper are organised
across the appendices that follow.
Appendix~\ref{app:meanfield} establishes the quantitative mean-field
limit for the microscopic dynamics (Lemma~\ref{lem:moment},
Theorem~\ref{thm:poc}, and Corollary~\ref{cor:functional-poc}).
Appendix~\ref{app:derivation} constructs the
rigorous microscopic foundation of the sentiment PDE \eqref{eq:pde} and
proves propagation of chaos for the regularised ESSO
(Theorem~\ref{thm:poc-esso}). Appendix~\ref{app:powerlaw} proves the
power-law substitution (Proposition~\ref{prop:powerlaw}) and the
existence, uniqueness, and bounds for the sentiment field
(Theorem~\ref{thm:field}). Appendix~\ref{app:multiasset} proves
the multi-asset a priori bound (Proposition~\ref{prop:multi-apriori})
and small-coupling existence and uniqueness (Theorem~\ref{thm:multi}).
Appendix~\ref{app:kappa} proves the explicit solution of the $\kappa_t$
dynamics (Theorem~\ref{thm:kappa}), the frozen-coefficient forecast
(Corollary~\ref{cor:forecast}), the conditional integrated variance
(Corollary~\ref{cor:IV}), and the first-passage law
(Lemma~\ref{lem:firstpassage}) underlying the drawdown gauge
(Definition~\ref{def:gauge}). Appendix~\ref{app:wellposed} proves
well-posedness of the coupled $(v_t, \kappa_t)$ system
(Theorem~\ref{thm:wellposed}). Appendix~\ref{app:montecarlo} proves
Monte Carlo convergence (Proposition~\ref{prop:mc2}), and
Appendix~\ref{app:price} records the SDE existence and uniqueness result
(Theorem~\ref{thm:protter}) used in Theorem~\ref{thm:price}.
Appendix~\ref{app:robustness} collects additional simulation evidence.
The measure-change invariants (Proposition~\ref{prop:invariants}) are
proved in the body.

\section{Quantitative Mean-Field Limit for the Microscopic Dynamics}
\label{app:meanfield}

We justify the replacement of the discrete $N$-particle system by the
continuous attention field, for the diffusion-with-killing skeleton of
Appendix~\ref{app:derivation}.

This first lemma is deliberately about the literal, \emph{discrete-time}
ESSO recursion of Section~\ref{sec:esso-dynamics} (the update runs in
integer steps $t \to t+1$), so the tool it needs is the discrete form
of Gr\"onwall's inequality --- iterating a one-step bound $m(t+1) \le
2m(t) + C$ --- not the continuous, ODE-comparison version. The
continuous-time bound that Theorem~\ref{thm:poc} and
Theorem~\ref{thm:poc-esso} actually use, for their SDE dynamics, is a
separate result, Lemma~\ref{lem:moment-cts} below.

\begin{lemma}[Uniform second-moment bound]\label{lem:moment}
There exists a locally bounded function $M(t) > 0$, independent of $N$
and $j$, such that $\sup_N \sup_j \E[|X_j(t)|^2] \le M(t) < \infty$ for
all $t \ge 0$.
\end{lemma}

\begin{proof}
The particles evolve according to $X_j(t+1) = X_j(t) + \eta_j(t) +
\zeta_j(t)$, where $\|\eta_j(t)\| \le K$ almost surely (bounded
deterministic drift) and $\zeta_j(t) \sim \mathcal{N}(0,
\sigma_m^2(t)I)$ is independent of $\mathcal{F}_t$ with $\sigma_m^2(t)
\le \sigma_{\max}^2 < \infty$. Let $m_j(t) := \E[|X_j(t)|^2]$. Expanding
the squared norm and taking expectations, using independence of
$\zeta_j(t)$ from $X_j(t)$ and Young's inequality $|2X_j \cdot \eta_j|
\le |X_j|^2 + K^2$,
\[
m_j(t+1) \le 2 m_j(t) + C, \qquad C := 3K^2 + 2d\sigma_{\max}^2
,
\]
with $d = 2$. Iterating this linear discrete Gr\"onwall inequality from
$0$ to $n = \lfloor t \rfloor$ gives $m_j(t) \le M(t) := 2^{\lfloor t
\rfloor + 1}(\sup_j \E[|X_j(0)|^2] + C)$, locally bounded and
independent of $N, j$.
\end{proof}

This discrete-time bound is for the literal, iterated sparrow updates
of Section~\ref{sec:esso-dynamics} (e.g.\ the contrarian update); it is
not the object Theorem~\ref{thm:poc} or Theorem~\ref{thm:poc-esso}
evolve, which are continuous-time SDEs. The moment bounds those need
are a different, continuous-time fact, given next.

\begin{lemma}[Uniform second-moment bound, continuous time]\label{lem:moment-cts}
Let $Y_t \in \R^2$ solve $dY_t = b_t(Y_t)\,dt + \sqrt2\,dB_t$ for any
drift with $|b_t(y)| \le B$ for all $t, y$ (covering both
Theorem~\ref{thm:poc}'s baseline, $B=0$, and Theorem~\ref{thm:poc-esso}'s
regularised drift, $B < \infty$ by Condition~1). Then
\[
\E|Y_t|^2 \;\le\; M(t) \;:=\; e^t\bigl(\E|Y_0|^2 + B^2 + 4\bigr),
\]
locally bounded, and depending on the admissible drift only through its
bound $B$.
\end{lemma}

\begin{proof}
By It\^o's formula applied to $f(y) = |y|^2$ (using $\nabla f = 2y$,
$\sigma\sigma^T = 2I$ for the diffusion coefficient $\sigma = \sqrt2 I$,
so the It\^o correction $\frac12\mathrm{tr}(\sigma\sigma^TD^2f) = 4$),
\[
d|Y_t|^2 = 2Y_t\cdot b_t(Y_t)\,dt + 2\sqrt2\,Y_t\cdot dB_t + 4\,dt.
\]
Taking expectations (the stochastic integral is a true martingale after
the standard localisation by stopping times $\tau_n := \inf\{t :
|Y_t| \ge n\}$, using $M(T) < \infty$ on any fixed horizon to pass to
the limit by monotone convergence) and using $|b_t(y)| \le B$ with
Young's inequality $2|Y_t|B \le |Y_t|^2 + B^2$,
\[
\frac{d}{dt}\E|Y_t|^2 \;=\; 2\E[Y_t \cdot b_t(Y_t)] + 4 \;\le\;
\E|Y_t|^2 + B^2 + 4.
\]
Writing $m(t) := \E|Y_t|^2$, this is $m'(t) \le m(t) + (B^2+4)$; the
corresponding equality has exact solution $(m(0)+B^2+4)e^t - (B^2+4)$
(verified directly by differentiation), which is bounded above by
$(m(0)+B^2+4)e^t$ since $B^2+4>0$, giving the claim by the standard
comparison (Gr\"onwall) inequality for scalar ODEs.
\end{proof}

By Lemma~\ref{lem:moment} and Markov's inequality, for every $R > 0$,
$(\mu_t^N/N)(\{|x| > R\}) \le M(T)/R^2$ uniformly in $N$ and $t \in
[0,T]$, so the family $\{\mu_t^N/N\}$ is tight; by Prokhorov's theorem (Billingsley, 1999) it is relatively compact in the weak topology, and since second moments
are uniformly bounded, every weak limit lies in $\mathcal{P}_2(\R^2)$.

\begin{theorem}[Quantitative propagation of chaos]\label{thm:poc}
Consider the diffusion-with-killing system of
Appendix~\ref{app:derivation}: particles $\{X_i\}_{i \le N}$ evolve as
$dX_i = \sqrt{2}\,dB_i$ with independent Brownian motions $B_i$ and are
killed at the position-dependent rate $(1-c)W(X_i(t))$, independently
across $i$, with killing times $\tau_i$. Assume
\begin{enumerate}
\item[(H0)] the initial positions $X_i(0)$ are i.i.d.\ with law $\rho_0
:= u_0/\!\int u_0$, where $u_0 \in L^1 \cap L^\infty(\R^2)$ is strictly
positive with $\int_{\R^2} e^{\varepsilon_0 |y|^2}\rho_0(dy) < \infty$
for some $\varepsilon_0 > 0$ (satisfied by the datum \eqref{eq:datum},
a finite Gaussian mixture).
\end{enumerate}
Let $S_t := \{i \le N : t < \tau_i\}$ and define
\[
  \bar\mu_t^N := \frac{1}{|S_t|}\sum_{i \in S_t}\delta_{X_i(t)}
  \quad\text{on } \{|S_t| \ge 1\}, \qquad
  \bar\nu_t := \frac{u(\cdot,t)\,dx}{\int_{\R^2} u(\cdot,t)\,dx},
\]
with the convention $\bar\mu_t^N := \bar\nu_t$ on $\{|S_t| = 0\}$, and
with $u$ the unique solution \eqref{eq:CS} of \eqref{eq:linear}. Let
\[
  p_t := \Prob(t < \tau_1)
  = e^{-(1-c)\alpha t}\operatorname{sech}(\omega t)
    \int_{\R^2} e^{-\frac{\omega}{4}\tanh(\omega t)|y|^2}\,\rho_0(dy)
  \;\in\;(0,1]
\]
be the survival probability, given in closed form by
Remark~\ref{rem:mass}. Then:
\begin{enumerate}
\item[(i)] \emph{(Exact conditional law.)} For every $t \ge 0$ and $m
\ge 1$, conditionally on $\{|S_t| = m\}$ the surviving positions are
i.i.d.\ with law $\bar\nu_t$, and $|S_t| \sim \mathrm{Binomial}(N,
p_t)$.
\item[(ii)] \emph{(Quantitative convergence.)} There is an absolute
constant $C_0 < \infty$ and, for each $t$, a constant $C(t) < \infty$
with
\[
  \E\bigl[W_1(\bar\mu_t^N, \bar\nu_t)\bigr]
  \;\le\; C(t)\,\Bigl(\frac{\log(1+N)}{N}\Bigr)^{1/2},
  \qquad
  C(t) \;\le\; \frac{C_0\,\Theta(t)}{\sqrt{p_t}},
\]
where $\Theta(t)$ depends only on the moments of $\bar\nu_t$ and is
locally bounded. The rate is the sharp two-dimensional matching rate
(Remark~\ref{rem:sharpness}). Since $p_t \sim 2 m_\infty
e^{-((1-c)\alpha + \omega)t}$ as $t \to \infty$, with $m_\infty := \int
e^{-\omega|y|^2/4}\rho_0(dy) > 0$, the constant grows exponentially,
$C(t) = O\bigl(e^{((1-c)\alpha + \omega)t/2}\bigr)$: the effective
sample size is $p_t N$, not $N$.
\item[(iii)] \emph{(Survivor concentration.)}
$\Prob\bigl(||S_t|/N - p_t| \ge \varepsilon\bigr) \le
2e^{-2N\varepsilon^2}$.
\end{enumerate}
The flow $t \mapsto \bar\nu_t$ is the unique one with $u(\cdot,t) \in
L^1 \cap L^\infty$ for all $t \ge 0$, by \eqref{eq:CS}.
\end{theorem}

\begin{proof}
\emph{(i).} The particles are independent, so the survival indicators
are i.i.d.\ Bernoulli and $|S_t| \sim \mathrm{Binomial}(N,p_t)$. For a
single particle, the Feynman--Kac representation \eqref{eq:fk} with the
symmetry $K_t(x,y) = K_t(y,x)$ gives, for Borel $B$,
$\Prob(t < \tau_1,\, X_1(t) \in B) = (\int u_0)^{-1}\int_B u(x,t)\,dx$;
taking $B = \R^2$ identifies $p_t$ with \eqref{eq:umass} divided by
$\int u_0$. Dividing, $\Prob(X_1(t) \in B \mid t < \tau_1) =
\bar\nu_t(B)$; independence and exchangeability give the conditional
i.i.d.\ statement.

\emph{(ii).} By Theorem~\ref{thm:field}(iv) applied to $u$, $\bar\nu_t$
has a Gaussian envelope, so $M_q(t) := \int |x|^q\,\bar\nu_t(dx) <
\infty$ for every $q$. Conditionally on $\{|S_t| = m\}$, $\bar\mu_t^N$
is the empirical measure of $m$ i.i.d.\ draws from $\bar\nu_t$, so the
two-dimensional Kantorovich-$W_1$ estimate for empirical measures with a
finite moment of order $q > 2$ (Hahn and Shao, 1992;
Yukich, 1992; see Ledoux, 2019, eq.~(8) and Thm.~5) gives
\[
  \E\bigl[W_1(\bar\mu_t^N,\bar\nu_t) \,\big|\, |S_t| = m\bigr]
  \;\le\; C_0\,\Theta(t)\Bigl(\frac{\log(1+m)}{m}\Bigr)^{1/2}.
\]
Split on the Chernoff event $\Prob(|S_t| < \tfrac12 p_t N) \le e^{-p_t
N/8}$: on its complement $m \ge \tfrac12 p_t N$, and monotonicity of $m
\mapsto (\log(1+m)/m)^{1/2}$ for $m \ge 3$ gives the bound
$C_0\Theta(t)\sqrt 2\,\bigl(\log(1+N)/(p_t N)\bigr)^{1/2}$. On the
Chernoff event bound $W_1(\bar\mu_t^N,\bar\nu_t) \le
\int|x|\,d\bar\mu_t^N + M_1(t)$ and apply Cauchy--Schwarz; the
contribution is exponentially small by Lemma~\ref{lem:moment-cts} and is
absorbed into $C(t)$. The asymptotics of $p_t$ follow from
$\operatorname{sech}(\omega t) \sim 2e^{-\omega t}$ and dominated
convergence.

\emph{(iii).} Hoeffding's inequality for the i.i.d.\ bounded survival
indicators.
\end{proof}

\begin{corollary}[Convergence of the volatility functionals]
\label{cor:functional-poc}
Under the hypotheses of Theorem~\ref{thm:poc}, let
$(\lambda,\theta,\eta)[\cdot]$ denote the functionals
\eqref{eq:lam-func}--\eqref{eq:eta-func} computed from a configuration
through the estimator \eqref{eq:estimator} and the readout $T =
[(1-c)u]^{1/(1-c)}$. Then for every $t > 0$,
\[
  \E\bigl|(\lambda,\theta,\eta)[\bar\mu_t^N] -
  (\lambda,\theta,\eta)[\bar\nu_t]\bigr|
  \;\le\; C'(t)\,\Bigl(\frac{\log(1+N)}{N}\Bigr)^{1/2},
\]
with $C'(t)$ the product of the constants of
Proposition~\ref{prop:w1-stability} and Theorem~\ref{thm:poc}. In
particular $\kappa_t$ computed from $N$ particles converges to its
mean-field value at the same rate.
\end{corollary}

\begin{proof}
Immediate from Theorem~\ref{thm:poc}(ii) and
Proposition~\ref{prop:w1-stability}, whose proof uses only Lipschitz
regularity of the Mehler kernel, Kantorovich--Rubinstein duality, the
lower bound on the smoothed field, and the Gaussian envelope --- none of
which requires the second argument to be atomic, so it applies verbatim
with $\bar\nu_t$ absolutely continuous. The nonlinear readout enters
only through Step~3 of that proof, as the local Lipschitz constant
$L_\phi$.
\end{proof}

\section{Derivation of the Macroscopic Sentiment PDE}
\label{app:derivation}

This appendix derives the macroscopic sentiment PDE \eqref{eq:pde} from
a rigorous microscopic foundation. We adopt a two-layer construction: a
primitive interacting particle system whose mean-field limit produces a
linear field $u(x,t)$, the raw attention density, and a nonlinear
readout $T = [(1-c)u]^{1/(1-c)}$ that defines the sentiment field as a
static functional transformation of $u$. This separation is what makes
the analytic structure exactly tractable: the entire
microscopic-to-macroscopic passage is captured by a linear Kolmogorov
forward equation, and all nonlinearity in $T$ enters via the pointwise
readout.

\subsection{Microscopic particle system: diffusion with killing}

Consider $N$ particles $\{X_i\}_{i=1}^N \subset \R^2$ evolving
independently. Each follows Brownian diffusion $dX_i(t) = \sqrt{2}\,
dB_i(t)$ with independent standard 2-dimensional Brownian motions $B_i$
(microscopically, mutations and chaotic exploration in the (virality,
recency) plane), and is removed from the population at the instantaneous
position-dependent rate
\[
\mathrm{rate}_{\mathrm{kill}}(X_i(t)) := (1-c)\, W(X_i(t)) =
(1-c)(\alpha + \beta|X_i(t)|^2),
\]
where the prefactor $(1-c)$ is the structural calibration choice
producing the target equation \eqref{eq:pde} after the nonlinear
readout. Economically, $(1-c)\alpha$ is uniform attention decay and
$(1-c)\beta|x|^2$ is position-dependent loss (information far from
currently-relevant attention is dropped faster). Particles are mutually
independent, so propagation of chaos holds trivially.

\subsection{Mean-field limit: the linear PDE for the raw attention
density}

Let $\tau_i$ denote the killing time of particle $i$, and define the
empirical sub-probability measure of surviving particles $u_N(A, t) :=
\frac{1}{N}\#\{i : X_i(t) \in A,\ t < \tau_i\}$. Let $u_0 \in L^1 \cap
L^\infty(\R^2)$ be the initial density. By the strong law of large
numbers, $u_N(\cdot, t) \rightharpoonup^{*} u(\cdot, t)\, dx$, where $u$
is the single-particle sub-probability density admitting the
Feynman--Kac representation
\begin{equation}\label{eq:fk}
u(x, t) = \E^x\Bigl[u_0(X_t) \exp\Bigl(-(1-c)\int_0^t W(X_s)\,
ds\Bigr)\Bigr],
\end{equation}
with expectation over the Brownian motion $X_t = x + \sqrt{2}\, B_t$.
Since $W \ge \alpha > 0$ is continuous and bounded below, $-\Delta +
(1-c)W$ generates a strongly continuous, positivity-preserving semigroup
on $L^2(\R^2)$, and the functional \eqref{eq:fk} is that semigroup
acting on $u_0$: $u(x,t) = \bigl(e^{t(\Delta - (1-c)W)} u_0\bigr)(x)$
(Karatzas and Shreve, 1991, \S5.7). Hence $u$ is the unique bounded
classical solution of the linear PDE
\begin{equation}\label{eq:fk-pde}
\partial_t u = \Delta u - (1-c) W(x)\, u, \qquad u(\cdot, 0) = u_0,
\end{equation}
which is exactly \eqref{eq:linear}. The closed-form Mehler-kernel
solution \eqref{eq:CS} follows from the Mehler structure of the operator
$-\Delta + (1-c)W$. Independence of the particles makes propagation of
chaos exact at finite $N$ up to the empirical-measure fluctuations
quantified in Theorem~\ref{thm:poc}.

\subsection{Macroscopic readout: power-law transformation to sentiment}

The sentiment field is defined as the static, pointwise, nonlinear
readout $T(x,t) := \bigl[(1-c)\, u(x,t)\bigr]^{1/(1-c)}$. By
Proposition~\ref{prop:powerlaw}, $T$ satisfies exactly the nonlinear PDE
\eqref{eq:pde}; the transformation is invertible, $u = T^{1-c}/(1-c)$.
The exponent $1/(1-c) > 1$ encodes super-linear amplification of raw
attention counts into experienced sentiment, consistent with
super-linear social-media engagement scaling (viral cascades, retweet
networks, recommender amplification). The calibration $c = 0.25$
corresponds to the mild super-linearity exponent $1/(1-c) = 4/3$. The
macroscopic nonlinearity $-(c/T)|\nabla T|^2$ is therefore a direct
consequence of the power-law readout rather than of any nonlinear
interaction in the particle system, which is why the analytic structure
stays tractable.

\subsection{Log-gradient herding as effective macroscopic behaviour}

While the microscopic particles do not literally follow log-gradient
drifts, the macroscopic equation has an equivalent reformulation
exposing the herding interpretation. Setting $v := \log T$ and using
$\Delta \log T = \Delta T/T - |\nabla T|^2/T^2$, dividing \eqref{eq:pde}
by $T$ gives
\begin{equation}\label{eq:hj}
\partial_t v = \Delta v + (1-c)|\nabla v|^2 - W(x),
\end{equation}
a viscous Hamilton--Jacobi equation with quadratic Hamiltonian $H(p) =
(1-c)|p|^2$. The substitution $w := \exp((1-c)v)$ converts \eqref{eq:hj}
to the linear equation $\partial_t w = \Delta w - (1-c)W w$, i.e.\
\eqref{eq:fk-pde}; the mapping $w \mapsto T = w^{1/(1-c)}$ recovers the
sentiment field, exactly the inverse of the power-law substitution and a
consistency check that the same Mehler solution is reached from either
direction. At a fixed point where $T$ is approximately steady,
Hamilton--Jacobi--Bellman interpretations of \eqref{eq:hj} give optimal
drifts proportional to $\nabla v = \nabla \log T$, recovering the
log-gradient herding rule heuristically; the rigorous derivation does
not depend on this interpretation.

\subsection{Connection to the ESSO algorithm}

The ESSO algorithm of Section~\ref{sec:swarm} is more elaborate than the
bare diffusion-with-killing process: chaotic Tent-map initialisation
ensures exploration diversity, the producer/scrounger/contrarian role
structure captures heterogeneity, and adaptive mutation schedules adjust
the exploration-exploitation balance. These features improve finite-$N$
approximation quality and capture empirical structure absent from the
homogeneous independent-particle baseline. The full ESSO dynamics
violate the hypotheses of Theorem~\ref{thm:poc} in four identifiable
ways: the rank-dependent updates and the $j \le N/2$ split make the
interaction non-exchangeable; the $\arg\min$ term $X_{\mathrm{worst}}$
and the discontinuous role switching violate the bounded-Lipschitz
Condition~1; and the $A^+$ reflection step is not the gradient of a
bounded-Lipschitz potential. We therefore make no mean-field claim for
the literal ESSO dynamics. The extremal terms lie outside the
reach of exchangeable mean-field theory: $X_{\mathrm{worst}}$ and the
best-producer position are governed by extreme-value rather than bulk
statistics, are discontinuous functionals of the empirical measure in
$W_1$, and converge to the edge of the limiting support at
tail-dependent rates, so a drift driven by them cannot inherit the
$O(N^{-1/2}\sqrt{\log N})$ rate. The evidence of
Section~\ref{sec:exhibit1}, centre-of-mass stabilisation below $0.01$
within ${\sim}15$ iterations, shows the empirical dynamics settle
quickly, but it does not establish a mean-field limit for the literal
algorithm.

We instead prove propagation of chaos, at the sharp two-dimensional rate
of Theorem~\ref{thm:poc}, for a regularised dynamics that keeps the
swarm mechanism intact but replaces each obstruction with a
bounded-Lipschitz mean-field surrogate: the $\arg\min$/$\arg\max$
extrema by Gibbs soft-extrema \eqref{eq:soft}, and both the hard rank
split and the threshold role switches by a single logistic gate at the
population's smoothed $p_d(t)$-quantile of fitness
\eqref{eq:role}--\eqref{eq:role-blend} below, recovering the literal
top-fraction producer rule in the sharp limit. Concretely, the
attraction to
the best producer is replaced by attraction to the soft-best position,
composed from the start with the radial truncation $\tau_R(y) :=
y\min(1, R/|y|)$ onto the ball $B_R$ of Definition~\ref{def:mapping}
(no restriction: the measurement map already lands in $B_R$, and the
sink dissipates the exterior; this is what makes $X_\star[\mu]$
$W_1$-Lipschitz without a separate truncation step, verified in
Remark~\ref{rem:W1drift}):
\begin{equation}\label{eq:soft}
X_\star[\mu] := \frac{\int_{\R^2} \tau_R(y)\, e^{\beta\phi(y)}\, \mu(dy)}
{\int_{\R^2} e^{\beta\phi(y)}\, \mu(dy)}, \qquad
b_\star[\mu](x) := \gamma\, \psi\bigl(X_\star[\mu] - x\bigr),
\end{equation}
with $\phi$ a bounded Lipschitz fitness, $\beta < \infty$ an inverse
temperature, and $\psi(z) = z/(1 + |z|)$ a bounded-Lipschitz saturation
(the literal $\arg\max$ is the singular limit $\beta \to \infty$, at
which Lipschitzness is lost). The soft-worst term is defined analogously
with $e^{-\beta\phi}$, giving the producer and scrounger candidate
drifts $b^{\mathrm{prod}}[\mu](x) := b_\star[\mu](x)$ and
$b^{\mathrm{scrounge}}[\mu](x)$, the latter the bounded-Lipschitz
log-gradient surrogate of Section~\ref{sec:esso-dynamics}.

The hard rank split ($j \le N/2$) and the threshold role switches are
both replaced by a single device that preserves their literal meaning:
a smoothed version of the top-$p_d(t)$-fraction cutoff itself, rather
than a different statistic. For $\varsigma < \infty$ (a slope parameter,
distinct from the attraction strength $\gamma$ and inverse temperature
$\beta$ of \eqref{eq:soft}) and $\sigma_\varsigma(z) :=
1/(1+e^{-\varsigma z})$, define the smoothed upper-tail mass of the
fitness distribution under $\mu$,
\begin{equation}\label{eq:role}
G_\mu(z) := \int_{\R^2}\sigma_\varsigma\bigl(\phi(y) - z\bigr)\,\mu(dy),
\qquad z \in \R,
\end{equation}
strictly decreasing and smooth in $z$ from $1$ (at $z=-\infty$) to $0$
(at $z=+\infty$); let $q_p[\mu]$ denote the unique solution of
$G_\mu(z) = p$, $p := p_d(t) \in (0,1)$, so that $q_p[\mu]$ is the
smoothed $p$-quantile of the fitness values under $\mu$. Define the
producer-role probability and the resulting role-blended drift
\begin{equation}\label{eq:role-blend}
p^{\mathrm{prod}}(x,\mu) := \sigma_\varsigma\bigl(\phi(x) -
q_{p}[\mu]\bigr), \qquad
b[\mu](x) := p^{\mathrm{prod}}(x,\mu)\, b^{\mathrm{prod}}[\mu](x) +
\bigl(1 - p^{\mathrm{prod}}(x,\mu)\bigr)\, b^{\mathrm{scrounge}}[\mu](x),
\end{equation}
so an agent is a producer with probability given by how far its fitness
sits above the population's current $p_d(t)$-quantile, recovering
exactly "producer if in the top $p_d(t)N$ by fitness" as
$\varsigma \to \infty$.

Since $\phi$ is bounded, say $|\phi| \le M_\phi$, sandwiching
$\sigma_\varsigma(-M_\phi - z) \le G_\mu(z) \le \sigma_\varsigma(M_\phi -
z)$ for every $\mu$ shows that $q_p[\mu]$ always lies in the fixed
compact interval $[a_p, b_p] := [-M_\phi - z_p,\ M_\phi - z_p]$,
$z_p := \varsigma^{-1}\log\bigl(p/(1-p)\bigr)$, \emph{independently of
$\mu$}: no truncation device is needed to keep $q_p[\mu]$, and hence
$p^{\mathrm{prod}} \in (0,1)$, under control.

\begin{assumption}[Fitness-spread non-degeneracy]\label{ass:spread}
A configuration $\mu$ is \emph{$p$-non-degenerate} if there is $\varrho_0
> 0$, depending only on $(\varsigma, M_\phi, p)$ and not on $\mu$
itself, such that
\[
-\partial_z G_\mu(z) \;=\; \varsigma\int_{\R^2}
\sigma_\varsigma\bigl(\phi(y)-z\bigr)\bigl(1 -
\sigma_\varsigma(\phi(y)-z)\bigr)\,\mu(dy) \;\ge\; \varrho_0
\qquad \text{for every } z \in [a_p, b_p].
\]
\end{assumption}

\begin{remark}[Why this is needed, what it means, and why only one side]
\label{rem:spread}
The quantity bounded below is $\varsigma$ times the smoothed density of
the pushforward of $\mu$ under $\phi$ at level $z$; the assumption says
this density does not vanish across the whole quantile range
$[a_p,b_p]$. It cannot be dropped: for the literal, unsmoothed
top-fraction rule --- the limit $\varsigma \to \infty$ --- a population
with a fitness \emph{gap} straddling the $p$-quantile has an
ill-defined or discontinuous cutoff, exactly as the soft-extremum
\eqref{eq:soft} loses Lipschitzness in the singular limit $\beta \to
\infty$. Lemma~\ref{lem:spread-baseline} below verifies it explicitly,
with a computable constant, for the baseline ($b\equiv0$) mean-field
object $\bar\nu_t$ of Theorem~\ref{thm:poc}, and it fails only for
configurations whose fitness values cluster into well-separated groups
with a genuine gap at the current quantile level --- in particular, a
finite, atomic empirical measure $\hat\mu_t^N$ can fail it outright,
with no way around this at finite $N$. The reason this does not sink
the construction is Lemma~\ref{lem:quantile-lip} below: the Lipschitz
bound on $q_p$ needs Assumption~\ref{ass:spread} for only \emph{one} of
the two configurations being compared, the other being an arbitrary
probability measure. Every comparison inside Theorem~\ref{thm:poc-esso}'s
proof --- Step~0's Picard iterates, each a (weighted) law evolved from
the explicit $\rho_0$ under some admissible drift, and Step~1's
comparison of the atomic $\hat\mu_t^N$ to the smooth $\bar\nu_t$ ---
always has at least one non-degenerate member, so the atomic side never
needs to satisfy the assumption at all.
\end{remark}

\begin{lemma}[Non-degeneracy for the baseline flow]\label{lem:spread-baseline}
Fix $R' \ge R$ large enough that $\phi(B_{R'}) \supseteq [a_p,b_p]$
(possible for any bounded continuous $\phi$ whose extremal values over
$\R^2$ are approached within some finite region, since $B_{R'} \uparrow
\R^2$ and $\phi$ is continuous). For the diffusion-with-killing
skeleton of Appendix~\ref{app:derivation}, i.e.\ $b \equiv 0$ in
Theorem~\ref{thm:poc-esso} (exactly the setting of Theorem~\ref{thm:poc}),
Assumption~\ref{ass:spread} holds for $\mu = \bar\nu_t$ at every $t \ge
0$, with an explicit constant $\varrho_0 = \varrho_0(R',t)$ depending
only on $(R', R, \omega, c, \alpha, \varsigma, M_\phi, p, h_0)$.
\end{lemma}

\begin{proof}
By the semigroup identity \eqref{eq:exactsol}, $u(\cdot,t) =
\hat S(\cdot,t)$ exactly: $u(x,t) = (A(t)/N)\sum_j
K_{t+h_0^2/2}(x,X_j;\omega)$ (absorbing $e^{-w/2}$ into $A$). Retaining
a single summand and using strict positivity of the Mehler kernel, for
every $x \in B_{R'}$,
\[
u(x,t) \;\ge\; \frac{A(t)}{N}\,\inf_{x \in B_{R'},\, y \in B_R}
K_{t+h_0^2/2}(x,y;\omega) \;=:\; \frac{A(t)}{N}\,
\delta_K(R',R,t+h_0^2/2) \;>\; 0,
\]
the last quantity positive and explicit by the same computation as
Step~1(c) of Proposition~\ref{prop:w1-stability}'s proof, extended
verbatim from $B_R \times B_R$ to $B_{R'} \times B_R$ (maximise the
kernel's exponent over $|x| \le R'$, $|y| \le R$ in place of $|x|=|y|=R$).
Since $\bar\nu_t(dy) = u(y,t)\,dy/\int u(\cdot,t)$ and $\int u(\cdot,t)
\le \int u_0$ by Theorem~\ref{thm:field}(iii), $\bar\nu_t$ has a density
$f_t$ satisfying
\[
f_t(y) \;\ge\; \underline f(t) \;:=\; \frac{A(t)\,
\delta_K(R',R,t+h_0^2/2)}{N\int u_0} \;>\; 0, \qquad y \in B_{R'}.
\]
By compactness of $[a_p,b_p]$ and uniform continuity of $\phi$ on the
compact set $B_{R'}$, there are $\eta > 0$ and $\ell > 0$, depending
only on $(R', \mathrm{Lip}(\phi), \varsigma)$, such that for every $z
\in [a_p,b_p]$ some $y_z \in B_{R'}$ has $|\phi(y_z) - z| \le
1/\varsigma$ (by $\phi(B_{R'}) \supseteq [a_p,b_p]$) and $B_\eta(y_z)
\cap B_{R'}$ has Lebesgue measure at least $\ell$, on which
$\sigma_\varsigma(\phi(y)-z)(1-\sigma_\varsigma(\phi(y)-z))$ is bounded
below by a fixed $c_\eta > 0$ (continuity of the integrand, uniform in
$z$ over the compact range by the same compactness argument).
Consequently
\[
-\partial_z G_{\bar\nu_t}(z) \;=\; \varsigma\int_{\R^2}
\sigma_\varsigma(\phi(y)-z)\bigl(1-\sigma_\varsigma(\phi(y)-z)\bigr)\,
\bar\nu_t(dy) \;\ge\; \varsigma\, c_\eta\, \ell\, \underline f(t)
\;=:\; \varrho_0(R',t) \;>\; 0
\]
for every $z \in [a_p,b_p]$, which is Assumption~\ref{ass:spread}.
\end{proof}

\begin{lemma}[Smoothed-quantile stability]\label{lem:quantile-lip}
If $\mu$ is $p$-non-degenerate with constant $\varrho_0$
(Assumption~\ref{ass:spread}), then for \emph{every} probability
measure $\nu$ on $\R^2$ (no condition on $\nu$ required),
\[
\bigl|q_p[\mu] - q_p[\nu]\bigr| \;\le\;
\frac{\varsigma\,\mathrm{Lip}(\phi)}{4\varrho_0}\, W_1(\mu,\nu).
\]
\end{lemma}

\begin{proof}
Write $q := q_p[\mu]$, $q' := q_p[\nu]$; both lie in $[a_p,b_p]$
regardless of any non-degeneracy, by the sandwiching argument above
(which uses only $|\phi|\le M_\phi$). Since $y \mapsto
\sigma_\varsigma(\phi(y)-z)$ takes values in $[0,1]$ and is
$(\varsigma\,\mathrm{Lip}(\phi)/4)$-Lipschitz in $y$ for every fixed $z$
(chain rule, $\sup|\sigma_\varsigma'| = \varsigma/4$), Kantorovich--Rubinstein
duality gives, for every fixed $z$ and \emph{any} $\mu,\nu$,
\begin{equation}\label{eq:G-lip-mu}
\bigl|G_\mu(z) - G_\nu(z)\bigr| \;\le\; \frac{\varsigma\,
\mathrm{Lip}(\phi)}{4}\, W_1(\mu,\nu).
\end{equation}
By the fundamental theorem of calculus applied to $G_\mu$ alone and
Assumption~\ref{ass:spread} on $[\min(q,q'),\max(q,q')] \subset
[a_p,b_p]$,
\[
\Bigl|\int_{q'}^{q}\bigl(-\partial_zG_\mu(z)\bigr)\,dz\Bigr|
\;\ge\; \varrho_0\,|q-q'|,
\qquad\text{i.e.}\qquad
\bigl|G_\mu(q)-G_\mu(q')\bigr| \;\ge\; \varrho_0\,|q-q'|,
\]
this holding regardless of whether $q \le q'$ or $q > q'$ (the integrand
is bounded below by $\varrho_0$ throughout the interval either way).
Since $G_\mu(q) = p = G_\nu(q')$, $G_\mu(q) - G_\mu(q') = G_\nu(q') -
G_\mu(q')$, which by \eqref{eq:G-lip-mu} at $z=q'$ has absolute value at
most $\varsigma\,\mathrm{Lip}(\phi)\,W_1(\mu,\nu)/4$. Hence
$\varrho_0|q-q'| \le \varsigma\,\mathrm{Lip}(\phi)\,W_1(\mu,\nu)/4$,
which is the claim. Nothing above used any property of $\nu$ beyond its
being a probability measure.
\end{proof}

\begin{lemma}[Lipschitz role-blending]\label{lem:role-lip}
Suppose that in every pair $(\mu,\nu)$ compared, at least one of $\mu,\nu$
is $p$-non-degenerate with a constant $\varrho_0$ uniform over the pairs
considered (Assumption~\ref{ass:spread}; Lemma~\ref{lem:quantile-lip}
places no condition on the other member of the pair). If
$b^{\mathrm{prod}}, b^{\mathrm{scrounge}}$ are each bounded by $B$ and
jointly Lipschitz in $(x,\mu)$ with constant $L$, then $b$ of
\eqref{eq:role-blend} is bounded by $B$ and jointly Lipschitz, on these
pairs, with constant
\[
2L \;+\; B\,\frac{\varsigma\,\mathrm{Lip}(\phi)}{4}\Bigl(1 +
\frac{\varsigma}{4\varrho_0}\Bigr).
\]
\end{lemma}

\begin{proof}
By the chain rule and Lemma~\ref{lem:quantile-lip},
\[
\bigl|p^{\mathrm{prod}}(x,\mu) - p^{\mathrm{prod}}(x',\nu)\bigr|
\;\le\; \frac{\varsigma}{4}\Bigl[\,\mathrm{Lip}(\phi)\,|x-x'| +
\bigl|q_p[\mu] - q_p[\nu]\bigr|\Bigr]
\;\le\; \frac{\varsigma\,\mathrm{Lip}(\phi)}{4}\Bigl(1 +
\frac{\varsigma}{4\varrho_0}\Bigr) d,
\]
with $d := |x-x'| + W_1(\mu,\nu)$, and $p^{\mathrm{prod}} \in [0,1]$ is
bounded. Writing the increment of \eqref{eq:role-blend} between
$(x,\mu)$ and $(x',\nu)$ as
\[
p^{\mathrm{prod}}(x,\mu)\bigl[b^{\mathrm{prod}}[\mu](x) -
b^{\mathrm{prod}}[\nu](x')\bigr] +
\bigl(1-p^{\mathrm{prod}}(x,\mu)\bigr)\bigl[b^{\mathrm{scrounge}}[\mu](x)
- b^{\mathrm{scrounge}}[\nu](x')\bigr]
\]
\[
\qquad{}+\bigl[b^{\mathrm{prod}}[\nu](x') -
b^{\mathrm{scrounge}}[\nu](x')\bigr]\bigl[p^{\mathrm{prod}}(x,\mu) -
p^{\mathrm{prod}}(x',\nu)\bigr],
\]
the first two brackets are each bounded by $L\,d$, using
$p^{\mathrm{prod}} \in [0,1]$; the third factor is bounded by $2B$ and
the fourth by the display above. Summing gives the stated constant.
Boundedness of $b$ by $B$ is immediate, since $b$ is a convex
combination of two drifts each bounded by $B$.
\end{proof}

Both $b_\star$ (via Remark~\ref{rem:W1drift}) and the role-blended $b$
of \eqref{eq:role-blend} (via Lemma~\ref{lem:role-lip}) satisfy
Condition~1 on every pair of configurations in which at least one member
is $p$-non-degenerate. This is a \emph{weaker} property than
Condition~1's literal statement (Lipschitz for \emph{every} pair in
$\mathcal{P}_2^M$, with no restriction): a finite atomic configuration
can fail Assumption~\ref{ass:spread} outright (Remark~\ref{rem:spread}),
so $b$ of \eqref{eq:role-blend} is not globally Lipschitz over all of
$\mathcal{P}_2^M$ in the sense Condition~1 states verbatim. What Steps~0
and~1 of Theorem~\ref{thm:poc-esso}'s proof use is
exactly the weaker property, and Lemma~\ref{lem:spread-general} below
shows every pair those steps compare has a non-degenerate
member: Step~1 compares the atomic $\hat\mu_t^N$
only to the smooth $\bar\nu_t$, never to another atomic configuration,
so only $\bar\nu_t$'s side needs to be non-degenerate; Step~0's Picard
iterates are each the (weighted) law of a bounded-drift diffusion
evolved from the explicit, already-smooth $\rho_0$, so each iterate is
itself non-degenerate by the same lemma. This is what makes the
regularised ESSO drift of Theorem~\ref{thm:poc-esso}
concretely specified rather than assumed. These surrogates leave the
diffusive (mutation) and dissipative (sink) parts untouched. Write
$\mathcal{P}_2^M = \{\mu :
\int |x|^2\, \mu(dx) \le M\}$ for the moment class preserved by the
dynamics (Lemma~\ref{lem:moment-cts}).

\begin{lemma}[Non-degeneracy for general bounded drift]\label{lem:spread-general}
Fix $T < \infty$, $R' \ge R$ as in Lemma~\ref{lem:spread-baseline}, and
$\eta > 0$ small enough that $B_\eta(y_z) \subset B_{R'}$ for the points
$y_z$ used there. For every admissible drift $b$ satisfying Condition~1
with bound $B$, and every $t \in [0,T]$, Assumption~\ref{ass:spread}
holds for $\mu = \bar\nu_t$ (the weighted flow of
Theorem~\ref{thm:poc-esso}), with a constant $\varrho_0 = \varrho_0(T)$
depending only on $(B, T, R, R', \omega, c, \alpha, \beta, \varsigma,
\eta, \varepsilon_0, M_\phi, p)$ and \emph{not} on which admissible $b$
is used.
\end{lemma}

\begin{proof}
Write $Y_t$ for the (unweighted) diffusion $dY_t = b[\nu_t](Y_t)\,dt +
\sqrt2\,dB_t$, $Y_0 \sim \rho_0$, whose weighted law is $\bar\nu_t$.

\emph{Step 1 (raw small-ball probability).} By Aronson's theorem
(Aronson, 1967), since $|b[\nu_t]| \le B$ and the diffusion coefficient
is the constant, uniformly elliptic $2I$, the transition density
satisfies $p_t^b(x,w) \ge c_1 t^{-1} e^{-c_2|x-w|^2/t}$ for $t \in
(0,T]$, with $c_1, c_2$ depending only on $B$ (and $T$), not on $\nu_t$
or on which admissible $b$ is used. Combining with $\rho_0(x) \ge
\underline\rho_0 > 0$ for $x \in B_{R'}$ (established via $\delta_K$ in
the proof of Lemma~\ref{lem:spread-baseline}), for any $y_z \in B_{R'}$,
\[
\Pr\bigl(Y_t \in B_\eta(y_z)\bigr) \;\ge\; \underline\rho_0
\int_{B_\eta(y_z)} \int_{B_{R'}} c_1 t^{-1} e^{-c_2|x-w|^2/t}\,dx\,dw .
\]
Substituting $u = (x-w)/\sqrt t$ turns the inner integral into
$\int_{(B_{R'}-w)/\sqrt t} e^{-c_2|u|^2}\,du$, a quantity that is
\emph{decreasing} in $t$ for $w \in B_{R'}$ (dividing the fixed set
$B_{R'}-w$ by a larger $\sqrt t$ shrinks the domain of a positive
integrand), so its value at $t \in (0,T]$ is bounded below by its value
at $t = T$. Hence
\[
\Pr\bigl(Y_t \in B_\eta(y_z)\bigr) \;\ge\; \underline g_\eta(T) \;:=\;
\underline\rho_0\, c_1 \int_{B_\eta(y_z)}\int_{(B_{R'}-w)/\sqrt T}
e^{-c_2|u|^2}\,du\,dw \;>\; 0
\]
for every $t \in (0,T]$; at $t=0$, $\Pr(Y_0 \in B_\eta(y_z)) =
\rho_0(B_\eta(y_z)) \ge \pi\eta^2\underline\rho_0$ directly, of the same
order. This bound is uniform over admissible $b$, since it used only
the bound $B$ (via Aronson) and the fixed, known $\rho_0$.

\emph{Step 2 (excluding the exit event).} Let $\tau_{R''} := \inf\{s :
|Y_s| > R''\}$. By $|Y_s| \le |Y_0| + |Y_s - Y_0|$ and $|Y_s - Y_0| \le
Bs + \sqrt2\,|B_s| \le Bt + \sqrt2\sup_{s\le t}|B_s|$ (using $|b|\le B$
pathwise), a union bound gives
\[
\Pr(\tau_{R''} \le t) \;\le\; \Pr\bigl(|Y_0| > R''/2\bigr) +
\Pr\Bigl(\sup_{s\le t}|B_s| > \tfrac{R''/2 - Bt}{\sqrt2}\Bigr) .
\]
The first term is bounded via Markov's inequality applied to
$e^{\varepsilon_0|Y_0|^2}$, using the finite exponential moment of
$\rho_0$ assumed in Theorem~\ref{thm:poc-esso}: $\Pr(|Y_0|>R''/2) \le
M_{\rho_0}\, e^{-\varepsilon_0(R''/2)^2}$ with $M_{\rho_0} := \int
e^{\varepsilon_0|y|^2}\rho_0(dy) < \infty$. The second term is bounded
by the reflection principle for planar Brownian motion (coordinatewise
reflection plus a union bound; see Karatzas and Shreve, 1991, \S2.8),
giving $\Pr(\sup_{s\le t}|B_s| > a) \le 4\,e^{-a^2/(4t)}$. Both bounds
are exponentially small in $R''^2$, uniformly over $t \in (0,T]$ and
admissible $b$ (the bound depends on $b$ only through $B$), so $R''$
can be chosen, depending only on $(B,T,\varepsilon_0,M_{\rho_0})$ and
the target $\underline g_\eta(T)$, so that
\[
\Pr(\tau_{R''} \le t) \;\le\; \tfrac12\,\underline g_\eta(T)
\qquad \text{for every } t \in (0,T].
\]

\emph{Step 3 (deterministic weight bound on the good event).} On
$\{\tau_{R''} > t\}$, $W(Y_s) = \alpha + \beta|Y_s|^2 \le \alpha + \beta
R''^2$ for all $s \le t$, so $\bar w(t) = \exp(-(1-c)\int_0^t W(Y_s)\,
ds) \ge \exp\bigl(-(1-c)(\alpha+\beta R''^2)t\bigr) =: \underline
w(R'',t)$, deterministically.

\emph{Step 4 (combine).} By Steps 1--3,
\[
\E\bigl[\bar w(t)\,\mathbf 1_{\{Y_t \in B_\eta(y_z)\}}\bigr]
\;\ge\; \underline w(R'',t)\cdot\Pr\bigl(Y_t \in B_\eta(y_z),\,
\tau_{R''}>t\bigr)
\;\ge\; \underline w(R'',t)\cdot \tfrac12\underline g_\eta(T),
\]
using $\Pr(Y_t\in B_\eta(y_z),\tau_{R''}>t) \ge \Pr(Y_t\in
B_\eta(y_z)) - \Pr(\tau_{R''}\le t) \ge \underline g_\eta(T) -
\tfrac12\underline g_\eta(T)$. Since $W \ge \alpha$ always,
$\E[\bar w(t)] \le e^{-(1-c)\alpha t}$ unconditionally, so
\[
\bar\nu_t\bigl(B_\eta(y_z)\bigr) \;=\;
\frac{\E[\bar w(t)\mathbf 1_{\{Y_t\in B_\eta(y_z)\}}]}{\E[\bar w(t)]}
\;\ge\; e^{-(1-c)\beta R''^2 t}\cdot\tfrac12\underline g_\eta(T)
\;=:\; \underline\nu(t) \;>\;0,
\]
decreasing in $t$, so $\underline\nu(t) \ge \underline\nu(T) > 0$
for all $t \in [0,T]$.

\emph{Step 5 (translate to the non-degeneracy bound).} As in the proof
of Lemma~\ref{lem:spread-baseline}, for every $z \in [a_p,b_p]$ there is
$y_z \in B_{R'}$ with $|\phi(y_z)-z|\le1/\varsigma$, and
$\sigma_\varsigma(\phi(y)-z)(1-\sigma_\varsigma(\phi(y)-z)) \ge c_\eta$
on $B_\eta(y_z)$ for a constant $c_\eta > 0$ uniform over $z \in
[a_p,b_p]$ by compactness. Hence
\[
-\partial_z G_{\bar\nu_t}(z) \;=\; \varsigma \int_{\R^2}
\sigma_\varsigma(\phi(y)-z)\bigl(1-\sigma_\varsigma(\phi(y)-z)\bigr)\,
\bar\nu_t(dy) \;\ge\; \varsigma\,c_\eta\,\bar\nu_t(B_\eta(y_z))
\;\ge\; \varsigma\,c_\eta\,\underline\nu(T) \;=:\; \varrho_0(T) \;>\; 0
\]
for every $z \in [a_p,b_p]$ and every $t \in [0,T]$, and every
admissible bounded drift $b$. This is Assumption~\ref{ass:spread}.
\end{proof}

\begin{remark}[Scope of the extension]\label{rem:spread-general}
Lemma~\ref{lem:spread-general} extends Lemma~\ref{lem:spread-baseline}
from $b\equiv0$ to every drift satisfying Condition~1: every Picard
iterate in Theorem~\ref{thm:poc-esso}'s Step~0 is non-degenerate by this
lemma (each is exactly such a weighted law, evolved from $\rho_0$ under
\emph{some} admissible bounded drift), and $\bar\nu_t$ itself is covered
for the drift that is the eventual fixed point. Combined with
Lemma~\ref{lem:quantile-lip}'s one-sided hypothesis
(Remark~\ref{rem:spread}), Condition~1's weaker, one-sided form
(Lipschitz whenever at least one compared configuration is
non-degenerate) holds on \emph{every} pair Theorem~\ref{thm:poc-esso}'s
proof compares, both in Step~0 (iterate versus iterate, both
non-degenerate) and Step~1 (the atomic $\hat\mu_t^N$ against the
non-degenerate $\bar\nu_t$). The exit-time splitting used in the proof
needs only Aronson's theorem, a Gaussian-moment Markov bound, and the
planar reflection principle, and avoids any conditional density
identity for the Feynman--Kac weight.
\end{remark}

\emph{Soft (Feynman--Kac) reformulation.} A synchronous coupling with a
shared killing clock is not available, because the killing rate $(1-c)W$
is position-dependent: two coupled particles occupy distinct positions
and a shared clock would kill them at mismatched rates, generating a
decoupling mass that is first order in the interparticle distance and
degrades the elementary coupling to $O(N^{-1/4})$. We therefore
represent killing by Feynman--Kac weights rather than by removing
particles. Let $\{X_i\}$ solve the non-killed interacting system $dX_i =
b[\hat\mu_t^N](X_i)\,dt + \sqrt{2}\,dB_i$ driven by the self-normalised
weighted measure
\[
\hat\mu_t^N := \frac{\sum_i w_i(t)\,\delta_{X_i(t)}}{\sum_i w_i(t)},
\qquad w_i(t) := \exp\!\Bigl(-(1-c)\int_0^t W(X_i(s))\,ds\Bigr) \in
(0,1].
\]
The limit is the normalised Feynman--Kac flow of McKean--Vlasov type,
$\bar\nu_t(dy) = \E[w(t)\delta_{Y_t}]/\E[w(t)]$ with $Y$ the
$b[\bar\nu]$-driven diffusion, and it coincides with the normalised
linear field $u(\cdot,t)\,dx/\!\int u$. Each weight lies in $(0,1]$
irrespective of the particle's position, so the self-normalised measure
$\hat\mu_t^N$ is always well defined; on the moment class
$\mathcal{P}_2^M$ the normalising denominator $N^{-1}\sum_i w_i(t)$
concentrates at rate $O(N^{-1/2})$ around a deterministic mass bounded
below on compact time intervals (in the uncoupled case this mass is
exactly the $p_t$ of Theorem~\ref{thm:poc}, given in closed form by
Remark~\ref{rem:mass}). Killing therefore never degenerates the
normalisation, and the weight of a particle is a Lipschitz functional
of its trajectory on bounded-moment events, which is what the
synchronous coupling below exploits.

\begin{theorem}[Propagation of chaos for the regularised ESSO]
\label{thm:poc-esso}
Let $b : \mathcal{P}_2^M \times \R^2 \to \R^2$ be the regularised ESSO
drift \eqref{eq:role-blend}, assembled from the soft-extremum attractor
\eqref{eq:soft} and the population-mean logistic role gate
\eqref{eq:role}, and assume that for some $B, L < \infty$,
\begin{equation}\label{eq:driftLip}
\sup_{\mu, x} |b[\mu](x)| \le B, \qquad
|b[\mu](x) - b[\nu](y)| \le L\bigl(|x - y| + W_1(\mu, \nu)\bigr)
\end{equation}
for all $\mu, \nu \in \mathcal{P}_2^M$ and $x, y \in \R^2$
(Conditions~1 and~2 of Section~\ref{sec:conditions}; Remark~\ref{rem:W1drift}
fully verifies \eqref{eq:driftLip} for the soft-extremum attractor.
For the role-blended $b$ of \eqref{eq:role-blend}, Lemma~\ref{lem:role-lip}
verifies \eqref{eq:driftLip} on every pair with at least one
$p$-non-degenerate member; Lemma~\ref{lem:spread-general} shows every
pair compared in Step~0 and Step~1 below has such a member, so
\eqref{eq:driftLip} holds wherever the proof uses it, though
not in the literal global sense of holding for arbitrary $\mu,\nu \in
\mathcal{P}_2^M$ --- see Remark~\ref{rem:spread-general}). Let
$\{X_i\}_{i=1}^N$ solve the weighted interacting system of the soft
reformulation above, with i.i.d.\ initial positions of law $\rho_0 :=
u_0/\!\int u_0$ satisfying $\int_{\R^2} e^{\varepsilon_0|y|^2}
\rho_0(dy) < \infty$ for some $\varepsilon_0 > 0$ (the datum
\eqref{eq:datum} is a finite Gaussian mixture and qualifies), and let
$\bar\nu_t$ be the normalised Feynman--Kac McKean--Vlasov flow, which
for $b \equiv 0$ is the normalised linear field $u(\cdot,t)\,dx/\!\int
u$ of Theorem~\ref{thm:poc}. Then for every $T < \infty$ there is $C(T)
< \infty$, depending only on $(B, L, T, \alpha, \beta, c)$ and the
initial moments, such that
\[
\sup_{t \le T}\ \E\bigl[W_1(\hat\mu_t^N, \bar\nu_t)\bigr]
\;\le\; C(T) \Bigl(\frac{\log(1+N)}{N}\Bigr)^{1/2}.
\]
The rate is the two-dimensional matching rate and cannot be improved
beyond constants (Remark~\ref{rem:sharpness}). For $b \equiv 0$ the
estimate of Theorem~\ref{thm:poc} is recovered.
\end{theorem}

\begin{proof}
\emph{Step 0 (limit flow).} On $C([0,T]; \mathcal{P}_2^M)$ the map
sending a candidate flow $(\nu_t)$ to the normalised Feynman--Kac flow
of $dY = b[\nu_t](Y)\,dt + \sqrt{2}\,dB$ weighted by $\bar w(t) =
\exp(-(1-c)\int_0^t W(Y_s)\,ds)$ is a contraction for small $T$ in
$\sup_t W_1$: the drift is $(x,\mu)$-Lipschitz by \eqref{eq:driftLip},
the weights lie in $(0,1]$ and are Lipschitz functionals of the path on
bounded-moment events, and by Jensen the normalising mass is bounded
below,
\[
  \E[\bar w(t)] \;\ge\; \exp\Bigl(-(1-c)\bigl(\alpha t + \beta\int_0^t
  M(s)\,ds\bigr)\Bigr) \;=:\; \underline p_T \;>\; 0 .
\]
Iterating over subintervals gives a unique flow on $[0,T]$ (Sznitman,
1991; M\'el\'eard, 1996). This classical contraction argument compares
candidate flows generically, and Lemma~\ref{lem:role-lip} alone only
controls comparisons in which one side is non-degenerate; but every
Picard iterate arising here is, by construction, a weighted law
evolved from the explicit $\rho_0$ under some admissible bounded drift,
hence non-degenerate by Lemma~\ref{lem:spread-general}, so every pair
of iterates compared in the iteration has a non-degenerate member and
\eqref{eq:driftLip} applies wherever this argument needs it. Since $b$
is bounded and the initial law has a
Gaussian moment, Girsanov and Theorem~\ref{thm:field}(iv) give
$\bar\nu_t$ a Gaussian envelope, uniformly for $t \le T$; in particular
$M_q(t) := \int |x|^q \bar\nu_t(dx) < \infty$ for every $q$.

\emph{Step 1 (synchronous coupling).} Couple the particles with
i.i.d.\ copies $\{Y_i\}$ of the nonlinear diffusion, driven by the same
Brownian motions and initial conditions; let $\bar w_i$ be their weights
and $\hat\nu_t^N$ their self-normalised weighted empirical measure. The
Lipschitz drift gives
\[
  |X_i(t) - Y_i(t)| \;\le\; L\int_0^t \bigl(|X_i - Y_i| +
  W_1(\hat\mu_s^N, \bar\nu_s)\bigr)\,ds .
\]
Since $W(x) - W(y) = \beta(|x|^2 - |y|^2)$, weight differences obey
$|w_i - \bar w_i| \le (1-c)\beta\int_0^t (|X_i| + |Y_i|)|X_i -
Y_i|\,ds$. Transporting the weighted atoms of $\hat\mu_t^N$ onto those
of $\hat\nu_t^N$ costs at most
\[
  W_1(\hat\mu_t^N, \hat\nu_t^N) \;\le\;
  \sum_i \bar a_i\,|X_i(t) - Y_i(t)|
  \;+\; \sum_i |a_i - \bar a_i|\,\bigl(|Y_i(t)| + M_1(t)\bigr),
\]
where $a_i, \bar a_i$ are the normalised weights; by Cauchy--Schwarz and
the uniform fourth-moment bound (bounded drift, Gaussian noise) both
terms are controlled by $\sup_i \E[|X_i(t) - Y_i(t)|^2]^{1/2}$. On the
event $\{N^{-1}\sum_i \bar w_i \ge \underline p_T/2\}$, whose complement
has probability $\le e^{-N \underline p_T^2/2}$ by Hoeffding, the
self-normalisation contributes the factor $2/\underline p_T$; on the
complement, $W_1(\hat\mu_t^N,\hat\nu_t^N)$ is bounded by the sum of
first moments and Cauchy--Schwarz makes the contribution exponentially
small. Gr\"onwall then yields
\[
  \sup_{t \le T}\E\bigl[W_1(\hat\mu_t^N, \hat\nu_t^N)\bigr]
  \;\le\; C_1(T)\Bigl(\sup_{t \le T}\E\bigl[W_1(\hat\nu_t^N,
  \bar\nu_t)\bigr] + N^{-1/2}\Bigr).
\]

\emph{Step 2 (i.i.d.\ weighted fluctuation).} The pairs $(Y_i, \bar
w_i)_{i \le N}$ are i.i.d. Let $(U_i)_{i \le N}$ be i.i.d.\ uniform on
$[0,1]$, independent of everything else, and set $A_N := \{i \le N : U_i
\le \bar w_i(t)\}$ with accepted empirical measure $\tilde\nu_t^N :=
|A_N|^{-1}\sum_{i \in A_N}\delta_{Y_i(t)}$ on $\{A_N \ne \emptyset\}$.

\emph{(a) Accepted particles are i.i.d.\ from the limit law.} For Borel
$B$, conditioning on the path and using $U_i \perp Y_i$,
\[
  \Prob\bigl(i \in A_N,\; Y_i(t) \in B\bigr)
  = \E\bigl[\bar w(t)\,\mathbf 1_{\{Y_t \in B\}}\bigr]
  = \E[\bar w(t)]\;\bar\nu_t(B),
\]
so acceptance tilts the path law exactly by the Feynman--Kac weight; by
independence across $i$ and exchangeability, conditionally on $|A_N| = m
\ge 1$ the accepted positions are i.i.d.\ with law $\bar\nu_t$. By
Step~0 the law $\bar\nu_t$ has a Gaussian envelope, hence a finite
moment of some order $q > 2$, so the two-dimensional Kantorovich-$W_1$
estimate for empirical measures (Hahn and Shao, 1992;
Yukich, 1992; see Ledoux, 2019, eq.~(8) and Thm.~5) gives
\[
  \E\bigl[W_1(\tilde\nu_t^N, \bar\nu_t) \,\big|\, |A_N| = m\bigr]
  \;\le\; C\,\Bigl(\frac{\log(1+m)}{m}\Bigr)^{1/2}.
\]
The count $|A_N|$ is $\mathrm{Binomial}(N, \E[\bar w(t)])$ with
$\E[\bar w(t)] \ge \underline p_T > 0$, so Chernoff's lower-tail bound
makes $\{|A_N| < \tfrac12 \underline p_T N\}$ exponentially negligible;
on its complement, monotonicity of $m \mapsto (\log(1+m)/m)^{1/2}$ gives
the bound $C(T)(\log(1+N)/N)^{1/2}$.

\emph{(b) Weighted versus accepted empirical measures.} Conditionally on
the paths $(Y_i, \bar w_i)_{i \le N}$, the acceptance indicators are
independent Bernoullis with means $\bar w_i(t)$. For a \emph{single} $f$
with $\|f\|_\infty \le 1$ and $\mathrm{Lip}(f) \le 1$, the difference of the
unnormalised sums $N^{-1}\sum_i(\mathbf 1_{\{i \in A_N\}} - \bar
w_i)f(Y_i)$ has conditional mean zero and conditional variance
$O(N^{-1})$, each summand being bounded by $1$; taking the supremum
over the whole class of such $f$ (needed for $d_{\mathrm{BL}}$, not
just one function) is controlled by the same order via symmetrisation
and a generic chaining (entropy-integral) bound: the class is uniformly bounded and
1-Lipschitz, and both measures carry the Gaussian envelope of Step~0,
so the effective covering number of the class restricted to the $N$
points is polynomial in $N$ at the resolution $N$ points can resolve,
giving a Rademacher complexity of the same order $O(N^{-1/2})$ --- this
is the same chaining technique underlying the matching-theorem citation
of part (a) just above, applied to a bounded rather than an unbounded
function class. On the event
$\{N^{-1}\sum_i \bar w_i \ge \underline p_T/2\}$, whose complement is
exponentially small by Hoeffding, the same holds after
self-normalisation, so
\[
  \E\bigl[d_{\mathrm{BL}}(\hat\nu_t^N, \tilde\nu_t^N)\bigr]
  \;\le\; C_2(T)\,N^{-1/2}.
\]
To upgrade to $W_1$, let $f$ be $1$-Lipschitz with $f(0) = 0$, so $|f(x)|
\le |x|$, and truncate: $f_\rho := (-\rho)\vee(f \wedge \rho)$ satisfies
$\|f_\rho\|_\infty \le \rho$ and $\mathrm{Lip}(f_\rho) \le 1$, so
$f_\rho/\rho$ is admissible for $d_{\mathrm{BL}}$ once $\rho \ge 1$.
Then
\[
  \int f\,d(\hat\nu_t^N - \tilde\nu_t^N)
  \;\le\; \rho\, d_{\mathrm{BL}}(\hat\nu_t^N, \tilde\nu_t^N)
  \;+\; \int_{\{|x| > \rho\}} |x| \,d(\hat\nu_t^N + \tilde\nu_t^N).
\]
Both measures carry the uniform Gaussian envelope of Step~0, so the tail
integral has expectation $O(e^{-c\rho^2})$; choosing $\rho_N =
c^{-1/2}\sqrt{\log N}$ makes it $O(N^{-1/2})$, while the bulk term is
$\rho_N\,d_{\mathrm{BL}} \lesssim \sqrt{\log N}\,N^{-1/2}$. Taking the
supremum over $f$ gives $\E[W_1(\hat\nu_t^N, \tilde\nu_t^N)] \le
C_3(T)(\log(1+N)/N)^{1/2}$.

The triangle inequality $W_1(\hat\mu_t^N, \bar\nu_t) \le
W_1(\hat\mu_t^N, \hat\nu_t^N) + W_1(\hat\nu_t^N, \tilde\nu_t^N) +
W_1(\tilde\nu_t^N, \bar\nu_t)$ combines Step~1 with (a) and (b) and
completes the proof.
\end{proof}

\begin{remark}[Verifying the $W_1$-Lipschitz hypothesis]
\label{rem:W1drift}
Condition \eqref{eq:driftLip} is stronger than Lipschitz continuity in
$W_2$, since $W_1 \le W_2$, and it is what Step~1 requires in order to
close Gr\"onwall in the metric of the conclusion. It holds for
$X_\star[\mu]$ of \eqref{eq:soft} because the truncation $\tau_R$ is
already built into its definition: with $\phi$ bounded Lipschitz and
$\beta < \infty$, both integrands are bounded and Lipschitz, and the
denominator is bounded below by $e^{-\beta\|\phi\|_\infty}$, so
Kantorovich--Rubinstein duality makes $\mu \mapsto X_\star[\mu]$
$W_1$-Lipschitz; composing with the bounded-Lipschitz saturation
$\psi(z) = z/(1+|z|)$ gives \eqref{eq:driftLip}. Without the truncation
the integrand $y\,e^{\beta\phi(y)}$ is Lipschitz but unbounded, and
duality does not apply. The role-blending drift of
\eqref{eq:role-blend} needs no such truncation either: $\phi$ being
bounded already confines $q_p[\mu]$ to the fixed compact interval
$[a_p,b_p]$ independently of $\mu$, and $p^{\mathrm{prod}} \in (0,1)$
regardless, so Kantorovich--Rubinstein duality applies directly to
$G_\mu(z)$ at each fixed $z$; the Lipschitz bound for $p^{\mathrm{prod}}$
itself, under Assumption~\ref{ass:spread}, is
Lemma~\ref{lem:quantile-lip}, used in Lemma~\ref{lem:role-lip}.
\end{remark}

\begin{remark}[Sharpness, and why the estimate is stated in $W_1$]
\label{rem:sharpness}
The rate $\sqrt{\log N/N}$ is the two-dimensional matching rate of Ajtai
et al.\ (1984), and in $d = 2$ it is exactly attained in $W_1$: for the
standard Gaussian on $\R^2$, $\E[W_1(\mu_n,\mu)] \asymp \sqrt{\log n/n}$
with matching upper and lower bounds (Ledoux, 2019, Thm.~5), the upper
bound holding for any law with a finite moment of order $q > 2$
(\emph{ibid.}, eq.~(8)). It cannot be recovered in $W_2$ for a
Gaussian-tailed law: there $\E[W_2^2(\mu_n,\mu)] \asymp (\log n)^2/n$
(Ledoux, 2019, Thm.~1, upper bound; the matching lower bound is due to
Talagrand, by a scaling argument building on the Ajtai--Koml\'os--Tusn\'ady
theorem, reported in Ledoux, 2018), so
$\E[W_2] \asymp \log n/\sqrt n$ carries an irreducible extra
$\sqrt{\log n}$. The obstruction is structural rather than technical:
the matching rate $\sqrt{\log n/n}$ in $W_2$ is a theorem for densities
bounded above \emph{and below} on a compact domain, whereas $\bar\nu_t$
has Gaussian tails on all of $\R^2$ --- tails inherited from the
quadratic sink $\beta|x|^2$, which is precisely the structure that makes
the Mehler closed form available (Condition~5 of
Section~\ref{sec:conditions}). Compact confinement would restore the
compact-domain hypotheses at the cost of replacing \eqref{eq:kernel} by
a Bessel-type eigenfunction expansion. Since
Proposition~\ref{prop:w1-stability} and Corollary~\ref{cor:missing} are
stated in $W_1$, and $\kappa_t$ depends on the configuration only
through them, nothing downstream requires $W_2$.

One scope point is worth making precise. Ledoux's matching lower bound
(Thm.~5) is established for the standard Gaussian specifically, while
$\bar\nu_t$ is a Gaussian \emph{mixture}; we have not derived a
matching lower bound for $\bar\nu_t$ itself, and ``sharp'' should be
read accordingly. What is fully established, and is all this paper
needs ``sharp'' to mean: the $O(\sqrt{\log N/N})$ \emph{upper} bound
(eq.~(8)) applies to $\bar\nu_t$ directly, since it holds for
\emph{any} law with a finite moment of order $q>2$, and no faster rate
can hold as a \emph{general} theorem under that same hypothesis, since
the standard Gaussian --- itself an instance of the hypothesis ---
already saturates it. The rate is therefore sharp as a statement about
the moment class as a whole, which is exactly what
Theorem~\ref{thm:poc}'s upper bound requires.
\end{remark}

\section{Proofs of the Power-Law Substitution and the Sentiment-Field
Bounds}
\label{app:powerlaw}

\subsection{Proof of Proposition~\ref{prop:powerlaw}}

Let $c \in (0,1)$, $W$ continuous, $u \in C^{2,1}$ strictly positive,
and $T := \bigl[(1-c)u\bigr]^{1/(1-c)}$. The map $\phi(z) :=
[(1-c)z]^{1/(1-c)}$ is real-analytic and strictly positive on $(0,
\infty)$, so $T = \phi \circ u \in C^{2,1}$ with $T > 0$; the inverse
$\phi^{-1}(z) = z^{1-c}/(1-c)$ is $C^\infty$, so the substitution is
bijective in this regularity class. Differentiating $u = T^{1-c}/(1-c)$,
\[
\partial_t u = T^{-c}\, \partial_t T, \qquad \nabla u = T^{-c}\, \nabla
T,
\]
and applying $\nabla\cdot(f\mathbf{v}) = (\nabla f)\cdot\mathbf{v} +
f(\nabla\cdot\mathbf{v})$ with $f = T^{-c}$, $\mathbf{v} = \nabla T$,
using $\nabla(T^{-c}) = -cT^{-c-1}\nabla T$,
\begin{equation}\label{eq:divergence}
\Delta u = -c\, T^{-c-1}|\nabla T|^2 + T^{-c}\Delta T.
\end{equation}
Substituting into \eqref{eq:linear} and using $u = T^{1-c}/(1-c)$, the
factor $(1-c)$ cancels:
\[
T^{-c}\, \partial_t T = T^{-c}\Delta T - c\, T^{-c-1}|\nabla T|^2 - W\,
T^{1-c}.
\]
Multiplying by $T^c$ (invertible since $T > 0$) gives, term by term,
$\partial_t T = \Delta T - (c/T)|\nabla T|^2 - W(x)T$, which is
\eqref{eq:pde}. The reverse direction re-does the same calculation
multiplying by $T^{-c}$ instead: assuming $T$ solves \eqref{eq:pde},
reading \eqref{eq:divergence} in reverse gives $\partial_t u = \Delta u
- (1-c)W u$, i.e.\ \eqref{eq:linear}. \qed

\begin{remark}[Structural origin of the identity]\label{rem:structural}
The cancellation is that the $-cT^{-c-1}|\nabla T|^2$ term
produced by the divergence-of-product expansion is exactly what
\eqref{eq:pde} requires in its $-(c/T)|\nabla T|^2$ term after
multiplication by $T^c$. No additional algebraic identity is invoked;
the gradient-squared nonlinearity in $T$ corresponds, under the
substitution, to a linear sink-type term in $u$. This is precisely why
the substitution works only for the specific nonlinearity
$-(c/T)|\nabla T|^2$: any other functional form of the herding flux
would leave residual nonlinear $u$-terms after substitution, and the
closed-form Mehler solution would not be available.
\end{remark}

\subsection{Proof of Theorem~\ref{thm:field}}

\emph{Existence.} For $T_0 \in L^1 \cap L^\infty$ with $T_0 > 0$ a.e.,
set $u_0 := T_0^{1-c}/(1-c) \in L^1 \cap L^\infty$, and define $u$ by
\eqref{eq:CS}. The Mehler kernel is smooth and strictly positive on
$\R^2 \times \R^2 \times (0,\infty)$ with the semigroup property.
Semigroup theory for the harmonic-oscillator Hamiltonian $-\Delta +
(1-c)\beta|x|^2$ on $L^2(\R^2)$ (Evans, 2010, Ch.~6) gives $u \in
C^\infty(\R^2 \times (0,\infty))$, $u > 0$ for $t > 0$, and $u$
satisfies \eqref{eq:linear}. Setting $T := [(1-c)u]^{1/(1-c)}$,
Proposition~\ref{prop:powerlaw} gives $T \in C^{2,1}$, $T > 0$, solving
\eqref{eq:pde}, with $T(\cdot,t) \to T_0$ in $L^1 \cap L^\infty$ as $t
\downarrow 0$ by semigroup continuity.

\emph{Uniqueness.} If $T_1, T_2$ are strictly positive classical
solutions with the same initial data, then $u_i := T_i^{1-c}/(1-c)$ both
solve \eqref{eq:linear} with datum $u_0$; the linear equation has a
unique bounded classical solution (the operator $-\Delta + (1-c)W$ with
$W \ge \alpha > 0$ generates a strongly continuous contraction
semigroup), so $u_1 = u_2$ and hence $T_1 = T_2$.

\emph{(i) Positivity.} The Mehler kernel is strictly positive, and $u_0
\ge 0$, $u_0 \not\equiv 0$, so \eqref{eq:CS} is strictly positive for $t
> 0$; then $T = [(1-c)u]^{1/(1-c)} > 0$.

\emph{(ii) Boundedness.} Apply the parabolic maximum principle to
\eqref{eq:pde}. At an interior maximum $(x^*, t^*)$, $\nabla T = 0$,
$\Delta T \le 0$, $\partial_t T \ge 0$, so \eqref{eq:pde} gives
$\partial_t T = \Delta T - 0 - WT \le -\alpha T < 0$, contradicting
$\partial_t T \ge 0$ unless the maximum is $\le \sup_x T_0$. Gaussian
decay (iv) rules out escape of the supremum to infinity.

\emph{(iii) Mass monotonicity.} Integrating \eqref{eq:pde} over $\R^2$
with vanishing boundary terms (justified by (iv)), $\int \Delta T\, dx =
0$, and both $-c\int |\nabla T|^2/T\, dx \le 0$ (strict unless $T$
constant) and $-\int WT\, dx \le -\alpha\int T\, dx < 0$ give
$\frac{d}{dt}\int T\, dx < 0$ strictly whenever $T \not\equiv 0$. (At
the level of $u$, Remark~\ref{rem:mass} gives the exact law.)

\emph{(iv) Gaussian decay.} Writing $\Sigma_t := \sinh(\omega t)$,
$\Lambda_t := \cosh(\omega t)$ and completing the square, using
$\Lambda_t^2 - \Sigma_t^2 = 1$ and Cauchy--Schwarz $x \cdot y \le
|x||y|$,
\[
(|x|^2 + |y|^2)\Lambda_t - 2x\cdot y \ge |x|^2\,
\frac{\Sigma_t^2}{\Lambda_t},
\]
so from \eqref{eq:kernel},
\[
K_t(x, y; \omega) \le \frac{\omega}{4\pi\Sigma_t}
\exp\Bigl(-\frac{\omega\Sigma_t}{4\Lambda_t}|x|^2\Bigr)
= \frac{\omega}{4\pi\Sigma_t}
\exp\Bigl(-\frac{\omega}{4}\tanh(\omega t)|x|^2\Bigr),
\]
uniform in $y$. Hence $u(x, t) \le C(t)e^{-\lambda_u(t)|x|^2}$ with
$C(t) = e^{-(1-c)\alpha t}\omega\|u_0\|_{L^1}/(4\pi\Sigma_t)$ and
$\lambda_u(t) = \omega\tanh(\omega t)/4$. Applying $T =
[(1-c)u]^{1/(1-c)}$ gives Gaussian decay for $T$ with rate $\lambda(t) =
\lambda_u(t)/(1-c) = \omega\tanh(\omega t)/(4(1-c)) > 0$. \qed

\section{Proof of Theorem~\ref{thm:multi} and
Proposition~\ref{prop:multi-apriori}: Multi-Asset Existence,
Uniqueness, and the A Priori Bound}
\label{app:multiasset}

Apply the power-law substitution $u^m := (T^m)^{1-c^m}/(1-c^m)$
componentwise. Treating the cross-term as an inhomogeneous source,
\eqref{eq:multi} becomes
\[
\partial_t u^m = \Delta u^m - (1-c^m)W^m(x)\, u^m - \sum_{l \neq m}
\beta_{ml}\, G^{ml}(u^m, u^l, \nabla u^m, \nabla u^l),
\]
with $G^{ml} = \bigl[(1-c^l)u^l\bigr]^{c^l/(1-c^l)}
\bigl[(1-c^m)u^m\bigr]^{-1}(\nabla u^m \cdot \nabla u^l)$, smooth and
Lipschitz on the set $\mathcal{O}_{\delta,R} = \{u^m, u^l \in [\delta,
R],\ |\nabla u^m|, |\nabla u^l| \le R\}$.

\paragraph{Unperturbed solution.} When all $\beta_{ml} = 0$, the system
decouples into $M$ independent linear equations with the closed-form
Mehler solution \eqref{eq:CS}, giving smooth strictly positive $U^m$,
bounded above by $M_m$, with gradient bounds $\|\nabla U^m\| \le R_m$,
on every $[t_0, T^*]$, $t_0 > 0$ (Theorem~\ref{thm:field}). As for
Proposition~\ref{prop:multi-apriori}, $U^m$ has \emph{no} uniform-in-$x$
lower bound over all of $\R^2$ (it decays at infinity, like the field
itself), only the compact-set bound $U^m \ge \delta_{m,R}(t_0) > 0$ on
each fixed $B_R$.

\paragraph{Contraction (Theorem~\ref{thm:multi}).} On the ball $\mathcal{B}$ of radius
$\rho$ around $(U^1, \dots, U^M)$ in $X := C([t_0, T^*]; L^\infty \cap
L^2)^M$ with $\|u\|_X = \sup_t(\|u\|_{L^\infty} + \|u\|_{L^2} +
\|\nabla u\|_{L^\infty})$, the Duhamel map $\Phi^m(u) = U^m - \sum_{l
\neq m}\beta_{ml}\int_{t_0}^t \int \tilde{K}^m_{t-s}\, G^{ml}\, dy\, ds$
satisfies $\|\Phi^m(u) - U^m\|_X \le A\varepsilon$ and $\|\Phi(u) -
\Phi(v)\|_X \le B\varepsilon\|u - v\|_X$ with $\varepsilon :=
\max|\beta_{ml}|$, for constants $A, B$ depending only on the problem
parameters (via semigroup bounds on the Mehler kernel and the Lipschitz
constant of $G^{ml}$). This Lipschitz bound needs care: as the previous
paragraph noted, the factor $[(1-c^m)u^m]^{-1}$ inside
$G^{ml}$ is unbounded as $u^m \to 0$ at infinity, but it always appears
multiplying $\nabla u^m \cdot \nabla u^l$, which decays at a matching or
faster Gaussian rate (both fields inherit $U^m$'s decay on $\mathcal B$,
a ball of fixed radius around it); the product is therefore uniformly
bounded and itself Gaussian-decaying, by the same rate-matching
mechanism verified explicitly, for the analogous drift term, in the
proof below. Choosing $\rho = A\varepsilon$ and $\varepsilon <
\varepsilon^*_{\mathrm{local}} := \frac12 B^{-1}$ makes $\Phi$ a
contraction, giving a unique fixed point and hence a unique classical
solution of \eqref{eq:multi} on $[t_0, T^*]$.

\paragraph{The a priori bound (Proposition~\ref{prop:multi-apriori}),
via a corrected linearisation.} The naive way to linearise
\eqref{eq:multi} for a fixed-point argument --- freeze \emph{all} of
$(u^m, u^l, \nabla u^m, \nabla u^l)$ at their values under a candidate
$(u^m)$, and solve the resulting frozen-source linear equation for a new
candidate --- does not preserve the two-sided $L^\infty$ bound at each
step of the iteration, only at the eventual fixed point: with $\nabla
u^m$ frozen at the \emph{input's} value rather than treated as the
\emph{output's} own gradient, the cross term no longer vanishes at the
output's critical points, and the maximum principle argument below does
not apply to it. The bound is instead preserved at every step, by
keeping $\nabla u^m$ as the map's own unknown and freezing only the
\emph{other} assets $u^l$, $l \neq m$. Apply the power-law substitution
$u^m = (T^m)^{1-c^m}/(1-c^m)$ as in
Proposition~\ref{prop:powerlaw}, but now with $T^l$ ($l \neq m$) treated
as external data rather than set to zero; the cross term transforms as
\[
(T^m)^{-c^m}\left[-\beta_{ml}\,\frac{\nabla T^m \cdot \nabla
T^l}{T^m}\right]
\;=\; -\beta_{ml}\,(T^m)^{-1}\,\nabla u^m \cdot \nabla T^l,
\]
using $\nabla T^m = (T^m)^{c^m}\nabla u^m$ (verified by direct
computation: the two $c^m(T^m)^{-c^m-1}|\nabla T^m|^2$ terms from
$\Delta u^m$'s expansion cancel, just as in
Proposition~\ref{prop:powerlaw}, leaving this single term, linear in
$\nabla u^m$). Hence, for fixed $(u^l)_{l \neq m}$ (equivalently
$(T^l)_{l\neq m}$), $v^m := \Phi^m(u)$ is defined as the solution of the
\emph{linear drift-diffusion} equation
\begin{equation}\label{eq:drift-reform}
\partial_t v^m = \Delta v^m - a^m(x,t)\cdot\nabla v^m - (1-c^m)W^m(x)\,
v^m, \qquad a^m(x,t) := \sum_{l\neq m}\beta_{ml}\,[(1-c^m)u^m(x,t)]^{-1/(1-c^m)}\,\nabla T^l(x,t),
\end{equation}
with $u^m$ (not $v^m$) appearing in the drift coefficient $a^m$, so
$a^m$ is entirely determined by the input $(u^l)_{l=1}^M$ and does not
depend on the unknown $v^m$: \eqref{eq:drift-reform} is linear
in $v^m$. Since $u^m \ge \delta_m > 0$ (input constraint) and
$\|\nabla T^l\|_\infty \le R^T$ (the gradient bound on $\mathcal{C}$,
transferred from $u^l$ to $T^l$ via the bounded, monotone map
$\phi(z)=[(1-c^l)z]^{1/(1-c^l)}$ and its bounded derivative on
$[\delta_l, M_l^u]$, by the same argument as Step~3 of
Proposition~\ref{prop:w1-stability}'s proof), $a^m$ is bounded:
$\|a^m\|_\infty \le \bar a := \max_{l\neq m}|\beta_{ml}|\cdot
[(1-c^m)\delta_m]^{-1/(1-c^m)}\cdot R^T$, uniformly over $(u^l) \in
\mathcal{C}$ and over the coupling size $|\beta_{ml}|$ (the coupling
size enters only as a multiplicative prefactor on $\bar a$, never
obstructing boundedness itself).

At an interior maximum of $v^m$, $\nabla v^m = 0$, so the drift term
$a^m\cdot\nabla v^m$ vanishes there too (this is the point of keeping
$\nabla v^m$, rather than $\nabla u^m$, in the equation), and
\eqref{eq:drift-reform} gives $\partial_tv^m = \Delta v^m - 0 - W^mv^m
\le -\alpha^mv^m < 0$ at the maximum, contradicting $\partial_tv^m \ge
0$ unless the maximum is $\le \sup_xT_0^m(x)$ (transferring back from
$v^m$'s substitution variable). The componentwise minimum principle
gives strict positivity at every $(x,t)$, and the same compact-set
argument as Theorem~\ref{thm:field}'s gives a lower bound $\delta_R(t_0)$
on each fixed $B_R$ (Remark~\ref{rem:multi-decay}: no global lower bound
holds, or should be expected). This holds for \emph{any} bounded drift
$a^m$, hence for \emph{any} coupling size $|\beta_{ml}|$: the bounds of
Proposition~\ref{prop:multi-apriori} are exact, with no
coupling-dependent buffer, at every step of the iteration defined this
way --- not merely at its eventual fixed point. This is the content of
Proposition~\ref{prop:multi-apriori}; establishing existence of a
classical solution attaining it for coupling outside
Theorem~\ref{thm:multi}'s small-coupling regime is not pursued here.

Concrete numerical evaluation gives $\varepsilon^* \gtrsim 0.1$--$0.3$
for the calibration $c = 0.25$, $\alpha = 1$, $\beta = 1$, comfortably
above the empirically expected $|\beta_{ml}| = |\tanh(\gamma \cdot
\mathrm{corr}(v^m, v^l))|$ for typical equity-sector pairs. \qed

\section{Proof of Theorem~\ref{thm:kappa},
Corollaries~\ref{cor:forecast} and~\ref{cor:IV}, and
Lemma~\ref{lem:firstpassage}}
\label{app:kappa}

\paragraph{Existence.} With $\Phi_t = \exp(-\int_0^t \lambda_s\,ds)$
deterministic and continuously differentiable, $d\Phi_t/dt = -\lambda_t
\Phi_t$, hence $d(\Phi_t^{-1})/dt = \lambda_t\Phi_t^{-1}$ (ordinary
calculus; $\Phi$ carries no randomness). Apply It\^o's formula to the
single function $f(t,\kappa) := \Phi_t^{-1}\kappa$ evaluated along
$\kappa_t$. This is one application of the full time-dependent It\^o
formula, $df(t,\kappa_t) = \partial_t f\,dt + \partial_\kappa
f\,d\kappa_t + \tfrac12\partial_\kappa^2 f\,d\langle\kappa\rangle_t$,
not two separately composed rules whose order could matter: there is a
single, unambiguous expansion here, and the reason the last term
vanishes is only that $f$ is \emph{linear} in $\kappa$, so
$\partial_\kappa^2 f \equiv 0$ identically, term by term, with no
appeal to any commutativity property of derivatives or covariations.
Since $\partial_t f = \lambda_t\Phi_t^{-1}\kappa$ and $\partial_\kappa f
= \Phi_t^{-1}$,
\[
d(\Phi_t^{-1}\kappa_t) = \lambda_t\Phi_t^{-1}\kappa_t\,dt +
\Phi_t^{-1}\bigl[\lambda_t(\theta_t-\kappa_t)\,dt + \eta_t\,
dW_t^\kappa\bigr] = \lambda_t\theta_t\Phi_t^{-1}\,dt +
\eta_t\Phi_t^{-1}\,dW_t^\kappa.
\]
Integrating from $0$ to $t$ and multiplying by $\Phi_t$ gives
\eqref{eq:kappa-voc}. The stochastic integral is well defined:
$\Phi_t/\Phi_s = \exp(-\int_s^t\lambda_r\,dr)$ is continuous, hence
bounded, on the compact set $\{0\le s\le r\le t\}$, and $\eta_s$ is
locally bounded, so $\int_0^t(\Phi_t/\Phi_s)^2\eta_s^2\,ds < \infty$,
making the Wiener integral a genuine, well-defined Gaussian random
variable. Running the same computation in reverse confirms
\eqref{eq:kappa-voc} satisfies \eqref{eq:kappa-sde}, establishing
existence directly, without appeal to any general SDE
existence theorem.

\paragraph{Uniqueness.} Suppose $\kappa^{(1)}, \kappa^{(2)}$ both solve
\eqref{eq:kappa-sde} with the same initial value $\kappa_0$, driven by
the same $W^\kappa$. Their difference $\Delta_t :=
\kappa^{(1)}_t-\kappa^{(2)}_t$ satisfies
\[
d\Delta_t = \lambda_t(\theta_t-\kappa^{(1)}_t)\,dt -
\lambda_t(\theta_t-\kappa^{(2)}_t)\,dt + \eta_t\,dW_t^\kappa -
\eta_t\,dW_t^\kappa = -\lambda_t\Delta_t\,dt,
\]
since the diffusion coefficient $\eta_t$ does not depend on $\kappa_t$
at all and the drift is linear in $\kappa_t$: the two stochastic terms
cancel \emph{exactly}, term by term, leaving a deterministic linear
ODE. With $\Delta_0=0$, its unique solution is $\Delta_t \equiv 0$ for
all $t$. This is elementary calculus, not a general SDE uniqueness
theorem: additive noise together with a $\kappa$-linear drift makes the
difference of any two solutions deterministic outright, so no
Lipschitz-coefficient machinery is needed.

\paragraph{Gaussian law.} Since $\kappa_0$, $\Phi_t/\Phi_s$, $\lambda_s,
\theta_s, \eta_s$ are all deterministic, \eqref{eq:kappa-voc} expresses
$\kappa_t$ as a deterministic function of $t$ plus a Wiener integral of
a deterministic integrand against $W^\kappa$ --- precisely the
definition of a Gaussian random variable, with the stated mean and
variance; the baseline $\eta_1 = 0$ is degenerate.

\paragraph{Corollary~\ref{cor:forecast}.} On $[t, t+\Delta]$ with frozen
coefficients, $\Phi_s/\Phi_t = e^{-\lambda_t(s-t)}$, so $\kappa_s =
\theta_t + (\kappa_t - \theta_t)e^{-\lambda_t(s-t)} + \eta_t\int_t^s
e^{-\lambda_t(s-u)}\, dW_u^\kappa$. Integrating over $[t, t+\Delta]$ and
applying stochastic Fubini, $\int_t^{t+\Delta}\kappa_s\, ds$ is Gaussian
with mean $m$ and variance $s^2$ as stated. Since $W^\kappa \perp W^v$,
$\E[v_{t+\Delta} \mid \mathcal{F}_t] = \bar\theta + (v_t -
\bar\theta)\E[e^{-\int\kappa}] = \bar\theta + (v_t - \bar\theta)e^{-m +
s^2/2}$ by the Gaussian moment generating function.

\paragraph{Corollary~\ref{cor:IV}.} By Tonelli (the integrand is
non-negative) and the tower property,
\[
\E\Bigl[\int_t^{t+\Delta} v_s\, ds \,\Big|\, \mathcal{F}_t\Bigr]
= \int_0^\Delta \E[v_{t+u} \mid \mathcal{F}_t]\, du
= \int_0^\Delta \bigl(\bar\theta + (v_t - \bar\theta)\, e^{-m(u) +
s^2(u)/2}\bigr)\, du,
\]
where the inner expectation is Corollary~\ref{cor:forecast} at horizon
$u$. When $\kappa_s \equiv \kappa$ is frozen, $m(u) = \kappa u$, $s^2(u)
= 0$, and $\int_0^\Delta e^{-\kappa u}\, du = (1 -
e^{-\kappa\Delta})/\kappa$, giving the closed form. (Verified against
Monte Carlo at the Table~\ref{tab:params} calibration: analytic
$0.016807$ vs simulated $0.016801 \pm 0.000017$ over a 63-day window
from $v_t = 0.09$, $\kappa_t = 6$.)

\paragraph{Lemma~\ref{lem:firstpassage}.} Let $X_u = au + \sigma B_u$
and $\tau_{-x} := \inf\{u : X_u = -x\}$, $x > 0$. The event $\{\min_{u
\le \Delta} X_u \le -x\}$ equals $\{\tau_{-x} \le \Delta\}$ by path
continuity. The joint law of $(\min_{u \le \Delta} X_u, X_\Delta)$ for
drifted Brownian motion follows from the driftless reflection principle
and a Girsanov change of measure removing the drift (Karatzas and
Shreve, 1991, \S3.5.C): for $y \ge -x$,
\[
\Prob\bigl(\tau_{-x} \le \Delta,\ X_\Delta \in dy\bigr) =
e^{-2ax/\sigma^2}\, \frac{1}{\sigma\sqrt{2\pi\Delta}}
\exp\Bigl(-\frac{(y + 2x - a\Delta)^2}{2\sigma^2\Delta}\Bigr)\, dy.
\]
Integrating over $y \in [-x, \infty)$ gives $e^{-2ax/\sigma^2}\,
\Phi\bigl((-x + a\Delta)/(\sigma\sqrt{\Delta})\bigr)$ after the
substitution $z = y + 2x$, while the complementary event $\{X_\Delta \le
-x\} \subseteq \{\tau_{-x} \le \Delta\}$ contributes $\Phi\bigl((-x -
a\Delta)/(\sigma\sqrt{\Delta})\bigr)$; the two pieces partition
$\{\tau_{-x} \le \Delta\}$ according to $X_\Delta \ge -x$ or $X_\Delta <
-x$, yielding the stated formula. The sanity checks: at $a = 0$ the two
terms coincide and give the reflection principle
$2\Phi(-x/(\sigma\sqrt{\Delta}))$; as $x \downarrow 0$ the sum tends to
$\Phi(-a\sqrt{\Delta}/\sigma) + \Phi(a\sqrt{\Delta}/\sigma) = 1$. The
moment-matched application to the SDSV log-price
(Definition~\ref{def:gauge}), with $a = \mu -
\mathrm{IV}_t(\Delta)/(2\Delta)$ and $\sigma^2\Delta =
\mathrm{IV}_t(\Delta)$, matches the exact conditional mean and variance
of $\log(S_{t+\Delta}/S_t)$ when $\rho_{Sv} = 0$
(Dambis--Dubins--Schwarz time change conditional on the variance path,
followed by freezing the time change at its conditional mean); its
approximation quality for the running-minimum law is assessed
empirically in Section~\ref{sec:panel-e}. \qed

\section{Proof of Theorem~\ref{thm:wellposed}}
\label{app:wellposed}

\paragraph{Well-posedness of $\kappa_t$.} By
Theorem~\ref{thm:field}(iv), $T(\cdot,t)$ has Gaussian decay, so the
radial integrals defining $\lambda_t, \theta_t, \eta_t$ are finite; the
closed-form Mehler representation makes them locally Lipschitz and
locally bounded. The variation-of-constants formula \eqref{eq:kappa-voc}
(Appendix~\ref{app:kappa}) gives the unique strong solution.

\paragraph{Well-posedness and boundary behaviour of $v_t$.} Given the
locally bounded, $\mathcal{F}_t$-adapted process $\kappa_t$, the
variance SDE is a time-inhomogeneous CIR diffusion. For locally bounded
coefficients with $\bar\theta > 0$, $\xi > 0$, Yamada--Watanabe and
Engelbert--Schmidt theory (Ikeda and Watanabe, 1989) yields a unique
strong solution with $v_t \ge 0$ almost surely. For strict positivity:
on $\{\kappa_t \ge \kappa_* = \xi^2/(2\bar\theta)\}$, the comparison
theorem for one-dimensional diffusions gives $v_t \ge \underline{v}_t$
pathwise, where $\underline{v}_t$ is a constant-coefficient CIR process
with parameters $(\kappa_*, \bar\theta, \xi)$ satisfying the Feller
condition $2\kappa_*\bar\theta = \xi^2$, whose boundary at $0$ is
unattainable. Without a floor, the Gaussian $\kappa_t$ (extension) drops
below $\kappa_*$ with positive probability and the boundary is
attainable, so only non-negativity holds. Since $\kappa_t$ depends only
on $W^\kappa$ and the deterministic sentiment dynamics, it feeds the
$v$-equation non-anticipatively; the pair $(v_t, \kappa_t)$ is uniquely
determined. \qed

\section{Monte Carlo Convergence for Option Pricing}
\label{app:montecarlo}

\begin{proposition}\label{prop:mc2}
Let $\hat{C}_N = N^{-1}\sum_{i=1}^N e^{-rT}(S_T^{(i)} - K)^+$ be the
plain Monte Carlo estimator. Then $\hat{C}_N \to C(0)$ almost surely,
and if $\E^{\Q}[S_T^2] < \infty$ then
$\sqrt{N}(\hat{C}_N - C(0)) \xrightarrow{d} \mathcal{N}(0,
\mathrm{Var}(e^{-rT}(S_T - K)^+))$.
\end{proposition}

\begin{proof}
$e^{-rT}(S_T - K)^+ \le e^{-rT}S_T$, which is integrable by
Proposition~\ref{prop:martingale}, giving the strong law. Square
integrability is not automatic: by Andersen and Piterbarg (2007,
Prop.~3.1) the second moment of a Heston-type price explodes at a finite
time $T^\ast$ unless their case $D \ge 0$, $a < 0$ obtains, in which
case $T^\ast = \infty$. That case holds at the calibration of
Table~\ref{tab:params}, and whenever it holds $\E^{\Q}[S_T^2] < \infty$
and the CLT follows by the classical Monte Carlo central limit theorem
(Glasserman, 2003, Appendix A.2).
\end{proof}

\section{Existence and Uniqueness of the Asset-Price SDE}
\label{app:price}

Consider $dX_t = b(t, X_t)\, dt + a(t, X_t)\, dW_t$, $X_0 = x_0$, with
$b, a$ jointly measurable and progressively measurable, satisfying: (L)
local Lipschitz, meaning for every $R > 0$ there is $L_R$ with $|b(t, x)
- b(t, y)| + |a(t, x) - a(t, y)| \le L_R|x - y|$ for $x, y \in [-R, R]$;
and (G) linear growth, $|b(t, x)| + |a(t, x)| \le K(1 + |x|)$.

\begin{theorem}[Protter, 2005, Chapter V, \S3]\label{thm:protter}
Under (L) and (G), the SDE admits a unique strong solution $(X_t)_{t \ge
0}$ that is $\mathcal{F}_t$-adapted, continuous, satisfies
$\E[\sup_{s \le t}|X_s|^2] < \infty$, and is pathwise unique.
\end{theorem}

The asset-price SDE $dS_t = \mu(S_t, t)\, dt + \sqrt{v_t}\,\sigma(S_t,
t)\, dW_t$ has $b(t, x) = \mu(x, t)$ and $a(t, x) =
\sqrt{v_t}\,\sigma(x, t)$. By Theorem~\ref{thm:wellposed}, $v_t \ge 0$
a.s.\ and $\sqrt{v_t}$ is $\mathcal{F}_t$-adapted and locally bounded;
assuming $\mu, \sigma$ satisfy (L) and (G), the composite coefficients
inherit local Lipschitz and linear growth, so Theorem~\ref{thm:protter}
applies.

\section{Additional Simulation Evidence}
\label{app:robustness}

This appendix reports the one-at-a-time robustness study summarised in
Section~\ref{sec:exhibit3}. A single calibration cannot separate a
structural effect from an artifact of the chosen parameters.
Table~\ref{tab:robustness} varies each parameter one at a time from the
baseline of Table~\ref{tab:params} and reports three post-shock
statistics: the day-63 expected-volatility gap, the day-63 ATM
mispricing, and the regime ATM spread.

\begin{table}[ht]
\centering
\begin{tabular}{lccc}
\toprule
Configuration & $|\mathrm{gap}_{63}|$ (pp) & $|\mathrm{mis}_{63}|$ (\%)
& regime spread (\%) \\
\midrule
baseline & 3.37 & 7.8 & 8.6 \\
$\xi = 0.20$ & 3.52 & 7.9 & 8.7 \\
$\xi = 0.40$ & 3.12 & 7.7 & 8.3 \\
$\rho_{Sv} = -0.30$ & 3.37 & 7.8 & 8.5 \\
$\rho_{Sv} = -0.90$ & 3.37 & 7.8 & 8.6 \\
$\bar\theta = 0.06$ & 1.72 & 4.2 & 3.5 \\
$\bar\theta = 0.09$ & 0.27 & 0.3 & n/a \\
$\kappa_{\mathrm{shock}} = 4$ & 2.00 & 4.3 & 8.6 \\
$\kappa_{\mathrm{shock}} = 8$ & 4.28 & 10.5 & 8.6 \\
$\kappa_{\mathrm{shock}} = 10$ & 4.88 & 12.7 & 8.6 \\
$\lambda = 0.5$ & 3.49 & 8.0 & 8.6 \\
$\lambda = 2.0$ & 3.13 & 7.4 & 8.6 \\
\bottomrule
\end{tabular}
\caption{Robustness of the post-shock effects to one-at-a-time parameter
variation, from 40{,}000-path Monte Carlo at a fixed seed. Effects are
stable to vol-of-vol $\xi$, leverage $\rho_{Sv}$, and reversion speed
$\lambda$; they scale monotonically with shock size
$\kappa_{\mathrm{shock}}$; and they vanish as the initial variance
approaches its long-run level ($\bar\theta \to v_0 = 0.09$), where
$\kappa$ has nothing to revert and the model reduces to Heston. The
regime spread is computed at a fixed elevated reference variance and is
not comparable when $\bar\theta$ is raised above it, so it is omitted
for $\bar\theta = 0.09$.}
\label{tab:robustness}
\end{table}

First, invariance where it should hold: the volatility gap and
mispricing are essentially unchanged across vol-of-vol ($\xi \in \{0.20,
0.40\}$), leverage ($\rho_{Sv} \in \{-0.30, -0.90\}$), and
$\kappa$-reversion speed ($\lambda \in \{0.5, 2.0\}$), the parameters
that should not drive an effect whose source is the sentiment-driven
reversion channel. Second, scaling with the mechanism: the mispricing
rises from 4.3\% to 12.7\% as the shock $\kappa_{\mathrm{shock}}$ grows
from 4 to 10; the effect tracks the quantity that is supposed to cause
it. Third, correct degeneracy: at $\bar\theta = 0.09 = v_0$ the gap
collapses to 0.27 points and the mispricing to 0.3\%, because when
current variance already sits at its long-run level there is nothing for
a state-dependent $\kappa$ to act on, which is exactly the boundary at
which SDSV reduces to Heston.

The incremental value of observing $\kappa_t$ likewise scales with how
informative sentiment is. Table~\ref{tab:cov} reports the day-63
stationary forecast increment as a function of the coefficient of
variation of $\kappa_t$: it is negligible when sentiment barely moves
and reaches double digits only once $\kappa_t$ varies materially,
recovering the order of magnitude of the post-shock regime as
$\mathrm{CoV}(\kappa_t)$ approaches 0.6. For reference, the post-shock
increments at 5, 10, 21, and 63 days are 4.8\%, 8.8\%, 16.7\%, and
43.5\%; the corresponding stationary increments in an informative regime
($\mathrm{CoV} \approx 0.6$) are 1.5\%, 3.3\%, and 21.7\% at 5, 10, and
63 days.

\begin{table}[ht]
\centering
\begin{tabular}{ccc}
\toprule
$(\lambda, \eta)$ & $\mathrm{CoV}(\kappa_t)$ & 63-day increment (\%) \\
\midrule
$(1.0, 0.4)$ & 0.12 & 0.1 \\
$(1.0, 0.8)$ & 0.25 & 0.7 \\
$(1.0, 1.5)$ & 0.47 & 5.2 \\
$(0.5, 1.5)$ & 0.58 & 21.7 \\
\bottomrule
\end{tabular}
\caption{Stationary 63-day forecast increment $1 -
\mathrm{MSE}_{\mathrm{SDSV}}/\mathrm{MSE}_{\mathrm{Heston}}$ as a
function of the variability of $\kappa_t$. The information the reversion
speed carries about future variance grows with how far $\kappa_t$
ranges; when sentiment is quiet the second factor adds almost nothing
over the nested benchmark.}
\label{tab:cov}
\end{table}

\begin{figure}[ht]
\centering
\sdsvfigure{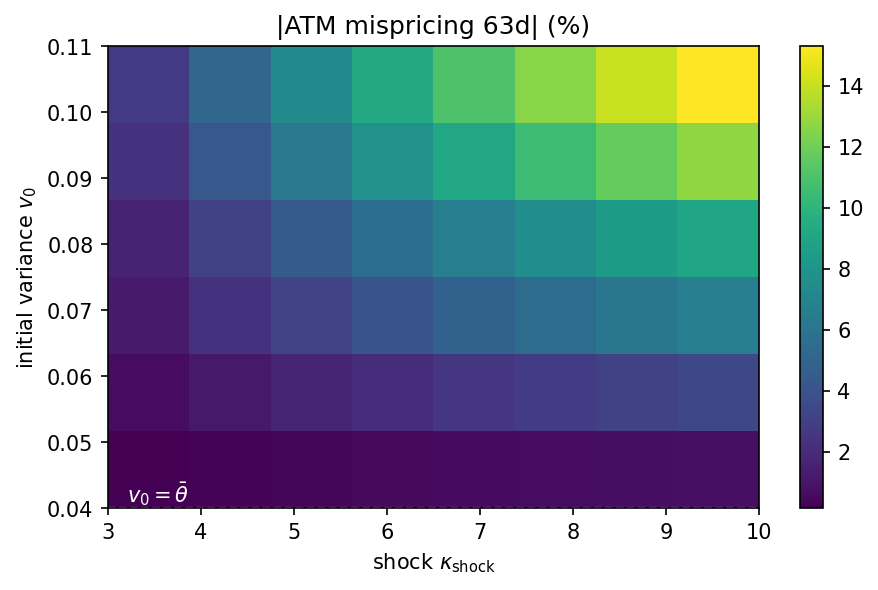}
\caption{\textbf{Interaction of shock size and variance displacement.}
Day-63 ATM mispricing over $\kappa_{\mathrm{shock}}$ and initial
variance $v_0$ (with $\bar\theta = 0.04$), from 40{,}000-path Monte
Carlo with common random numbers across cells. The effect is largest
where both the shock and the displacement $v_0 - \bar\theta$ are large,
reaching $16.3\%$ at the top right, and nearly vanishes along $v_0 =
\bar\theta$ (dashed line), where it stays below $1\%$: with the
expected-variance paths of the two models coinciding there, what
remains is vol-of-vol convexity, not the reversion channel. The
one-at-a-time variation in Table~\ref{tab:robustness} is therefore not
hiding an interaction.}
\label{fig:exhibit3}
\end{figure}

\end{document}